\documentclass[11pt,a4paper,requn]{amsart}
\usepackage[T1]{fontenc}
\usepackage{lmodern}
\usepackage{amsmath,amssymb,mathtools,mathrsfs}
\usepackage[shortlabels]{enumitem}
\usepackage[margin=1.1in]{geometry}
\usepackage[expansion=false]{microtype}
\usepackage{tikz,tikz-cd}
\usetikzlibrary{arrows.meta,calc,positioning}
\usepackage{booktabs,tabularx,array,longtable}
\usepackage{placeins,needspace}
\usepackage{xcolor}
\usepackage{hyperref}
\hypersetup{colorlinks,	linkcolor={red!50!black},citecolor={blue!50!black},urlcolor={blue!75!black},
bookmarksdepth=2,pdftitle={Cluster Geometry of Universal Schubert Polynomials I},pdfauthor={Jiarui Fei}}

\newtheorem{theorem}{Theorem}[section]
\newtheorem{proposition}[theorem]{Proposition}
\newtheorem{lemma}[theorem]{Lemma}
\newtheorem{corollary}[theorem]{Corollary}
\newtheorem{conjecture}[theorem]{Conjecture}
\newtheorem*{mainA}{Theorem A}
\newtheorem*{mainB}{Theorem B}
\newtheorem*{mainC}{Theorem C}
\newtheorem*{mainD}{Theorem D}
\theoremstyle{definition}
\newtheorem{definition}[theorem]{Definition}
\newtheorem{example}[theorem]{Example}
\theoremstyle{remark}
\newtheorem{remark}[theorem]{Remark}
\numberwithin{equation}{section}
\allowdisplaybreaks
\newcommand{\op}[1]{\operatorname{#1}}
\newcommand{\g}{\mathsf g}
\newcommand{\ZZ}{\mathbb Z}
\newcommand{\NN}{\mathbb N}
\newcommand{\QQ}{\mathbb Q}
\newcommand{\CC}{\mathbb C}
\newcommand{\Sch}{\mathfrak S}
\newcommand{\cM}{\mathcal M}
\newcommand{\cO}{\mathcal O}
\newcommand{\Gen}{\mathsf G}
\newcommand{\Can}{\mathsf T}
\newcommand{\MV}{\mathsf V}
\newcommand{\Code}{\mathsf C}
\newcommand{\Des}{\operatorname{Des}}
\newcommand{\one}{\mathbf 1}
\newcommand{\qbinom}[2]{\genfrac{[}{]}{0pt}{}{#1}{#2}}
\newcommand{\eps}{\varepsilon}
\newcommand{\Span}{\operatorname{span}}
\newcommand{\Av}{\operatorname{Av}}
\newcommand{\Exc}{\operatorname{Exc}}
\DeclareMathOperator{\Gr}{Gr}
\DeclareMathOperator{\Quot}{Quot}
\DeclareMathOperator{\Hom}{Hom}
\DeclareMathOperator{\End}{End}
\DeclareMathOperator{\Ext}{Ext}
\DeclareMathOperator{\coker}{coker}
\DeclareMathOperator{\rank}{rank}
\DeclareMathOperator{\inv}{inv}
\title[Universal Schubert polynomials I]{Cluster Geometry of Universal Schubert Polynomials I\\ Geometric Bases and Cluster Monomials}
\author{Jiarui Fei}
\address{School of Mathematical Sciences, Shanghai Jiao Tong University}
\email{jiarui@sjtu.edu.cn}
\thanks{The author was supported in part by the National Natural Science Foundation of China (Nos.~12131015 and 12571038).}
\date{}
\dedicatory{Dedicated to Professor William Fulton}
\subjclass[2020]{Primary 05E14, 13F60; Secondary 16G20, 20G42}
\keywords{universal Schubert polynomial, cluster algebra, Schubert polytope, preprojective algebra, geometric basis, quiver Grassmannian}

\newcommand{\CT}{\operatorname{CT}}
\newcommand{\RP}{\operatorname{RP}}
\newcommand{\HH}{\mathsf H}
\newcommand{\EE}{\mathsf E}
\newcommand{\parn}{\operatorname{par}}

\begin{document}
\begin{abstract}
We study Fulton's universal Schubert polynomials through the cluster
structure on the upper unitriangular group and representations of type
$A$. The convex hull of the initial tropical vectors of
standard-elementary monomials is unimodularly equivalent to a product
of simplices, and right-regular multisegments index three integral
geometric bases of the standard-elementary module. A recursive
prefix-record condition characterizes those universal Schubert
polynomials that are homogeneous under diagonal conjugation,
coincide with the corresponding generic and dual canonical basis
elements, and are extended cluster monomials. On this locus, cover
forests give squarefree factorizations into cluster variables.
Outside it, the initial Laurent expansion has a coefficient $-1$.
\end{abstract}
\maketitle
{\small\tableofcontents}

\section*{Introduction}
Schubert polynomials represent Schubert classes on complete flag
varieties~\cite{FultonFlags}. Fulton introduced \emph{universal Schubert polynomials}
$\Sch_w(c)$ by replacing the elementary symmetric functions in the
standard-elementary expansion with independent Chern-class variables
$c_r(k)$. He also defined \emph{universal double Schubert polynomials}
$\Sch_w(c;d)$ and their \emph{semi-double specializations}
$\Sch_w(c;y)$. These polynomials describe degeneracy loci for chains
of vector-bundle maps and specialize to the classical and quantum
Schubert polynomials~\cite{Fulton,FGP}.

We study the single universal polynomials $\Sch_w(c)$. Identifying their
variables with the entries above the diagonal of an upper unitriangular
matrix places them in a coordinate ring with a cluster structure.
This leads to the question of which $\Sch_w(c)$ are extended cluster
monomials. We answer it by a condition on records in
successive prefixes of $w$. It also characterizes
homogeneity under diagonal conjugation and equality with the generic
and dual canonical basis elements indexed by $w$.

Set $[N]=\{1,\ldots,N\}$, let $S_N$ be its symmetric group, and let
$U_N$ be the group of upper unitriangular $N\times N$ matrices.
Write $Z(c)\in U_N$ for the matrix with
$Z(c)_{p,q}=c_{q-p}(q-1)$ when $p\leq q$, with $c_0(k)=1$. For
$w=(w_1,\ldots,w_N)\in S_N$, define its \emph{inversion-end code}
and \emph{rank word} by
\[ e_j(w)=\#\{i<j:w_i>w_j\},\qquad r_j(w)=j-e_j(w). \]
The associated multisegment is
\[ \mathfrak m_w=\sum_{j=1}^N[r_j(w),j-1], \]
where empty intervals are omitted. At most one segment ends at each
vertex, so $\mathfrak m_w$ is \emph{right-regular}; conversely,
every right-regular multisegment arises uniquely in this way.
We write $\Gen_w$, $\Can_w$, and $\MV_w$ for the generic,
dual canonical, and Mirkovi\'c--Vilonen (MV) elements indexed by
$\mathfrak m_w$. In general these elements differ.

The standard triangular seed on $\CC[U_N]$ has variables
\[ x_{a,r}=\Delta_{[1,a],[r+1,r+a]}(Z(c)),\qquad a,r>0,\quad a+r\leq N, \]
with frozen boundary $a+r=N$.
The associated standard-elementary monomial
\[ \Code_w=\prod_{j=2}^N c_{e_j(w)}(j-1) \]
has a full initial tropical vector $\g_w$, including frozen
coordinates. The resulting polytope has a simple description.

\begin{mainA}
For $w\in S_N$ one has
\[ \g_w(a,r)=\one_{\{e_{a+r}(w)=r\}} -\one_{\{e_{a+r+1}(w)=r\}}, \qquad e_{N+1}(w)=0. \]
The convex hull of the vectors $\g_w$ is integrally unimodularly
equivalent to
\[ \Delta_1\times\Delta_2\times\cdots\times\Delta_{N-1}. \]
Its lattice points are exactly its $N!$ vertices
$\{\g_w:w\in S_N\}$.
\end{mainA}

The inverse tropical Chamber Ansatz gives the formula in
Theorem~A, and summation along fixed values of $r$ recovers the
inversion-end code (Theorem~\ref{thm:Schubert-polytope}). Thus the
vertices index the standard-elementary monomials and, via
$\mathfrak m_w$, the three geometric elements. The \emph{root degree} of $w$ is $d(w)=(d_1(w),\ldots,d_{N-1}(w))$,
defined by
$d_i(w)=\#\{j:r_j(w)\leq i<j\}$.
It is the dimension vector of the equioriented representation
$M_w$ with multisegment $\mathfrak m_w$.
The rank words $r(w)$ with a fixed root degree $d$ form a
parabolic Bruhat interval. Every right-regular equioriented representation
is a quotient of $P_\triangle=\bigoplus_{j=1}^{N-1}[1,j]$.
The quiver Grassmannian $X_d=\Quot_d(P_\triangle)$ is an iterated
Grassmannian bundle whose Schubert cells are indexed by this interval.
With $d_N=0$, its Poincar\'e polynomial is
\[ \sum_{w:d(w)=d}t^{\inv(r(w))} =\prod_{i=1}^{N-1}\qbinom{1+d_{i+1}}{d_i}_t. \]
Both maps in the framed-surjection correspondence of
Proposition~\ref{prop:framed-right-regular} are smooth. Consequently,
right-regular orbit closures and the corresponding Schubert-cell
closures in $X_d$ have smoothly equivalent local singularities.

Let $\cM_N$ be Fulton's standard-elementary $\ZZ$-module,
spanned by the monomials $\Code_w$. Restricting the unitriangular PBW
comparisons to right-regular multisegments gives three integral bases
of this module.

\begin{mainB}
For every $N$, one has
\[ \cM_N=\bigoplus_{w\in S_N}\ZZ\Gen_w =\bigoplus_{w\in S_N}\ZZ\Can_w =\bigoplus_{w\in S_N}\ZZ\MV_w. \]
Each of the three elements indexed by $w$ has diagonal-torus weight
$\sum_{i=1}^{N-1}d_i(w)\alpha_i$, where $\alpha_i$ is the $i$th simple root.
\end{mainB}

The parametrization and the basis comparisons are proved in
Part~I (Theorem~\ref{thm:geometric-bases-module}). Two familiar
pattern classes give more explicit identities. If $w$ avoids $321$,
then $\Sch_w(c)$ is a strict minor and an extended cluster monomial,
and it equals all three geometric elements. If $w$ avoids $312$, then
$\Code_w=\Gen_w=\Can_w=\MV_w$; this element equals
$\Sch_w(c)$ exactly when $w$ also avoids $1432$.

For arbitrary $w$, the cluster-character formula expresses
$F_w$ through Euler characteristics of quiver
Grassmannians~\cite{FKernelGLS}. We identify the relevant module
and give equations for these Grassmannians in terms of the rank word.
Let $T$ be the initial cluster-tilting object and $M$ a general point
of the preprojective component indexed by $\mathfrak m_w$. We identify
the corresponding $\End(T)^{\mathrm{op}}$-module
$E(M)=\Ext^1_{\Lambda_Q}(T,M)$ in terms of the graded operators $x,y$
of $M$: its vector spaces are
\[ E(M)_{a,r}=\ker\bigl(y^r:M_{a+r}\longrightarrow M_a\bigr)/x^aM_r. \]
We give explicit equations for the quiver Grassmannians of $E(M)$
parametrizing quotient representations. Their Euler characteristics determine the coefficients of $F_w$,
using the quotient convention fixed by the exchange-matrix sign
(Theorem~\ref{F:thm:general}).

For rank words satisfying the gap condition, these quiver
Grassmannians have a particularly simple form. The normalized
preprojective representation occurs as a summand of a cluster-tilting
object obtained by mutation from the initial one; in particular, it
is rigid. Its indecomposable summands correspond to the components
of a cover forest, and $xy$ has one Jordan block on each nonzero
space $E(M)_{a,r}$. Consequently, for each quotient dimension vector,
the underlying reduced quiver Grassmannian is either empty or a
single point. It follows that
\[ F_w(\mathbf y)= \sum_{q\in\mathcal P(r(w))\cap\ZZ^{\mu_N}}\mathbf y^q, \]
where $\mathcal P(r(w))$ is defined by integral difference
inequalities. The polytope is integral, the initial Laurent support is saturated in
the full initial exponent lattice, and every nonzero coefficient is one (Theorem~\ref{F:thm:gap} and
Corollary~\ref{F:cor:Laurent}).

Right-regularity does not, however, force the geometric bases
to coincide. In rank ten there are indices $v,w$ for which
\[ [\Can_v]\Gen_w=1,\qquad [\MV_v]\Can_w=1,\qquad [\MV_v]\Gen_w=2. \]
Here $[b]f$ denotes the coefficient of the basis element $b$ in
$f$ in the relevant basis (Proposition~\ref{prop:S10-deformation}).
Thus right-regularity by itself does not force the three functions
to agree.

We now return to the Schubert polynomials themselves. Let
$\mathcal C_N\subseteq S_N$ be the class defined by the recursive
\emph{prefix-record condition} of
Section~\ref{sec:Schubert-cluster-problem}: at each stage, the relevant
records of a standardized prefix must be internal. Equivalent
descriptions use empty rectangles, Rothe columns, and decorated
Bruhat arcs. The class is neither a classical pattern-avoidance class
nor one defined by finitely many ordinary mesh patterns.

\begin{mainC}
For $w\in S_N$, the following conditions are equivalent:
\[
\begin{aligned}
 w\in\mathcal C_N
 &\ \Longleftrightarrow\ \Sch_w(c)
           \text{ is homogeneous under diagonal conjugation}\\
 &\ \Longleftrightarrow\ \Sch_w(c)=\Can_w
           \ \Longleftrightarrow\ \Sch_w(c)=\Gen_w\\
 &\ \Longleftrightarrow\ \Sch_w(c)
           \text{ is an extended cluster monomial}\\
 &\ \Longleftrightarrow\ \Sch_w(c)
           \text{ has a nonnegative initial Laurent expansion}\\
 &\ \Longleftrightarrow\ \Sch_w(c)
           \text{ has a nonnegative Laurent expansion in every seed}.
\end{aligned}
\]
If $w\notin\mathcal C_N$, then there is $v\in S_N$ of different
root degree, $d(v)\ne d(w)$, such that
$[\mathbf x^{\g_v}]\Sch_w(c)=-1$.
\end{mainC}

The proof compares the standard-elementary expansion with the
root-lattice weight decomposition. The record condition forces the
expansion to have a single weight, and M\"obius inversion
together with a parabolic Kazhdan--Lusztig identity identifies
$\Sch_w(c)$ with $\Can_w$. The record condition also implies the gap condition. The
mutation construction for the cover representation then gives
$\Gen_w=\Can_w$ as a cluster monomial. In the other direction,
an external record yields the coefficient $-1$ in Theorem~C.
See Theorem~\ref{thm:schubert-characterization} for the full proof.

The three conditions have different roles. Right-regularity makes
the endomorphism algebra of the forward module triangular. The gap
condition gives a rigid cover representation with cyclic coefficient
spaces, whereas the record condition characterizes homogeneity of the
Schubert polynomial. The gap condition holds beyond $\mathcal C_N$, while
rigidity alone does not imply it
(Example~\ref{ex:rigid-nongap}).

For $w\in\mathcal C_N$, let $\mathcal F_w$ be its cover forest.
For each connected component $A$, define $v_A$ by the rank word
\[
 r_i(v_A)=\begin{cases}
  r_i(w),&i\in A,\\
  i,&i\notin A.
 \end{cases}
\]
This component decomposition determines the irreducible factors of
$\Sch_w(c)$.

\begin{mainD}
For every $w\in\mathcal C_N$,
\[ \Sch_w(c)=\prod_{A\in\pi_0(\mathcal F_w)}\Sch_{v_A}(c). \]
The nonunit factors are distinct variables of a common extended
cluster and are irreducible in $\CC[U_N]$. This is the squarefree
factorization of $\Sch_w(c)$ into irreducibles.
In particular, for $w\ne\mathrm{id}_N$,
\[
 \Sch_w(c)\text{ is an extended cluster variable}
 \quad\Longleftrightarrow\quad
 w\in\mathcal C_N\text{ and }\mathcal F_w\text{ is connected}.
\]
For $w\in\mathcal C_N\setminus\{\mathrm{id}_N\}$, these conditions
are also equivalent to irreducibility in $\CC[U_N]$.
\end{mainD}

The factorization is established in Theorem~\ref{solcomp:thm:closure},
and the variable criterion in Corollary~\ref{cor:Schubert-single-variable}.

The cover forests also determine compatibility: two generic functions
associated with gap words are monomials in a common seed precisely when
their mixed forests satisfy the matching condition of
Theorem~\ref{solmatch:thm:main}. For words in
$\mathcal C_N$, the criterion applies to the corresponding Schubert
polynomials. In the compatible case, we construct a seed in at most
$4\binom{N+1}{5}+2\binom{N+1}{4}$
mutations. Reduced pipe dreams for $w\in\mathcal C_N$ are in
weight-preserving bijection with the lattice points of
$\mathcal P(r(w))$. The bijection respects component factorization
and identifies the lattice point with the unique quotient of that
dimension as a geometric point (Theorem~\ref{solpipe:thm:main}).
The injection from pipe dreams is combinatorial; surjectivity follows
by counting, using the Schubert--generic coincidence. Finally, terminal insertion gives
an algebraic generating function for $\mathcal C_N$ and the
asymptotic formula of Theorem~\ref{thm:record-growth}.

The triangular seed is built from generalized minors
\cite{GLSinitial}; the generic basis and the conormal realization
come from~\cite{GLSgeneric}, and the MV parametrization uses the
bicrystal description of~\cite{BKKMV}. Part~I develops the Schubert
vertices, right-regular geometry, and generic $F$-polynomials for
arbitrary permutations. Part~II proves the homogeneous-locus theorem
and studies factorization, compatibility, pipe dreams, and enumeration.
The permutation $25413$ serves as a running example for the cover
representation, quotient points, cluster-variable criterion, and
reduced pipe dreams (Examples~\ref{ex:merging-cover-tree},
\ref{ex:component-variables}, and~\ref{ex:merging-cover-pipes}).
Theorem~C identifies $\Sch_w(c)$ with $\Gen_w$ and $\Can_w$ on
$\mathcal C_N$; whether it also equals $\MV_w$ remains open.

\subsection*{Conventions and notation}
We take $\NN=\{0,1,2,\ldots\}$. Unless a real interval is explicitly
specified, $[a,b]$ means the set of integers from $a$ to $b$; in a
multisegment it also denotes the corresponding interval module.
An interval is empty when $a>b$. The identity permutation is
$\mathrm{id}_N$, and empty products are one. Permutations act as
functions: right multiplication by $s_i$ interchanges positions $i$
and $i+1$. A statement about a general representation in an
irreducible variety means that it holds on some nonempty Zariski-open
subset. Finitely many such genericity conditions may be imposed
simultaneously.

The symbols $x_i$ are classical Schubert variables, $x_{a,r}$ are seed
minors, and the unindexed $x,y$ in a representation are graded linear
operators. The bold letter $\mathbf y=(y_{a,r})$ denotes the
coefficient variables. Standard coordinate vectors in seed and profile
lattices are $\mathbf e_v$; in a weight lattice we write
$\varepsilon_i$. The symbol $e_k(X)$ denotes the elementary
symmetric polynomial in the alphabet $X$, while
$\epsilon_i^{\mathrm R}$ is the right crystal string statistic.
\begin{center}
\small
\renewcommand{\arraystretch}{1.1}
\begin{tabularx}{\textwidth}{@{}l>{\raggedright\arraybackslash}X@{}}
	\toprule
	Notation & Meaning and convention \\
	\midrule
	$\cM_N$; $\Code_w$ & Standard-elementary $\ZZ$-module; its code basis element. \\
	$e_j(w)$; $r_j(w)$ & Inversion-end code; rank word $r_j(w)=j-e_j(w)$ (slot $j$ nonempty iff $r_j<j$). \\
	$\g_w$; $d(w)$ & Initial chamber vector (including frozen entries); root degree $\dim M_w$. \\
	$\mathfrak m_w$; $M_w$; $X_{r(w)}$ & Multisegment; forward-quiver module; preprojective cover module for gap words. \\
	$\Gen_w$; $\Can_w$; $\MV_w$ & Generic, dual canonical, and MV basis elements. \\
	$\mathcal I_N$; $\mu_N$ & All seed vertices; mutable vertices (complement frozen). \\
	$B$; $\widehat y_v$ & Exchange matrix (mutable rows, all columns; outgoing minus incoming); $\widehat y_v=\mathbf x^{B^{\mathsf T}\mathbf e_v}$. \\
	$T$; $E$; $q$ & Initial cluster-tilting object; $E=\Ext^1(T,M)$; quotient $q$, kernel dimension $\dim E-q$. \\
	$\mathcal P(r)$; $\mathsf R$ & Quotient-dimension polytope; $(\mathsf Rq)_s=\sum_aq_{a,s}$. \\
	$\mathcal F(r)$; $A$ & Cover forest; connected component. \\
	$\Pi_N$; $Q_N=\Pi_N^{-1}$ & Code-to-Schubert transition matrix; its inverse. \\
	$\mu(w)$; $\mu_{I_d}$ & Classical-specialization partition; M\"obius function of $I_d$. \\
	\bottomrule
\end{tabularx}
\end{center}
We say that a seed \emph{contains a polynomial} if the polynomial
is a monomial with nonnegative integer exponents in that seed's
extended cluster variables. We call a Laurent expansion
\emph{nonnegative} if all of its coefficients lie in
$\ZZ_{\ge0}$. Frozen variables are never inverted unless a
localization is explicitly stated.

\part{Schubert vertices and geometric bases}

\section{Universal Schubert polynomials and the standard-elementary module}
\label{sec:universal}

\subsection{Fulton's coordinates}

Let $U_N$ be the affine group scheme of upper unitriangular $N\times N$
matrices. Its coordinate ring is
\[ \ZZ[U_N]=\ZZ[z_{p,q}:1\leq p<q\leq N]. \]
Let $Z(c)$ be the tautological matrix, written in Fulton's coordinates by
\[ c_r(k)=z_{k+1-r,k+1} \qquad(1\leq r\leq k\leq N-1). \]
Thus
$Z(c)_{p,q}=c_{q-p}(q-1)$ for $p\leq q$,
with $c_0(k)=1$ and $c_r(k)=0$ for $r<0$ or $r>k$. Explicitly,
\[
 Z(c)=
 \begin{pmatrix}
  1&c_1(1)&c_2(2)&c_3(3)&\cdots&c_{N-1}(N-1)\\
  0&1&c_1(2)&c_2(3)&\cdots&c_{N-2}(N-1)\\
  0&0&1&c_1(3)&\cdots&c_{N-3}(N-1)\\
  \vdots&\vdots&\ddots&\ddots&\ddots&\vdots\\
  0&0&\cdots&0&1&c_1(N-1)\\
  0&0&\cdots&0&0&1
 \end{pmatrix}.
\]
Fulton's degeneracy-locus formulas use these variables as universal
Chern classes. We use the root grading determined by
\[ z_{p,q}(hZh^{-1})=(h_p/h_q)z_{p,q}(Z),\qquad \operatorname{wt}c_r(k)=\alpha_{k+1-r}+\cdots+\alpha_k, \]
where $h=\operatorname{diag}(h_1,\ldots,h_N)$ and
$\alpha_i=\varepsilon_i-\varepsilon_{i+1}$. Its height grading is
$\deg c_r(k)=r$; we call this the \emph{cohomological degree}, with
Chern classes assigned their complex codimension. Thus
$\Sch_w(c)$ has cohomological degree $\ell(w)$, although it need not
be homogeneous for the root grading.

For $w\in S_N$, the classical Schubert polynomial has a unique
expansion in the standard elementary monomials
\[ \Sch_w(x)= \sum a_{i_1,\ldots,i_{N-1}} e_{i_1}(x_1)\cdots e_{i_{N-1}}(x_1,\ldots,x_{N-1}). \]
Replacing each $e_r(x_1,\ldots,x_k)$ by $c_r(k)$ gives the
\emph{universal Schubert polynomial}
\[ \Sch_w(c)= \sum a_{i_1,\ldots,i_{N-1}} c_{i_1}(1)\cdots c_{i_{N-1}}(N-1). \]
For an upper unitriangular matrix $A$, we write $\Sch_w[A]$ for the
specialization $c_{q-p}(q-1)\mapsto A_{p,q}$; in particular,
$\Sch_w[Z(c)]=\Sch_w(c)$.

\subsection{Divided differences}
\label{subsec:divided-differences}

Let $x=(x_1,\ldots,x_N)$, and let $s_i$ interchange $x_i$ and $x_{i+1}$.
The divided difference operator is
\[ \partial_i f=\frac{f-s_if}{x_i-x_{i+1}} \qquad(1\leq i<N). \]
These operators satisfy the nil--Coxeter relations
\[ \partial_i^2=0,\qquad \partial_i\partial_j=\partial_j\partial_i\quad(|i-j|>1), \qquad
  \partial_i\partial_{i+1}\partial_i =\partial_{i+1}\partial_i\partial_{i+1}. \]
Thus $\partial_w$ is well defined from any reduced expression for $w$.
On classical Schubert polynomials,
\[
\partial_i\Sch_w(x)=
 \begin{cases}
  \Sch_{ws_i}(x),&\ell(ws_i)=\ell(w)-1,\\
  0,&\ell(ws_i)=\ell(w)+1.
 \end{cases}
\]

\subsection{The standard-elementary module}

For $\alpha=(\alpha_1,\ldots,\alpha_{N-1})$ with
$0\leq\alpha_k\leq k$, set
\[
 \mathcal E_\alpha(c)=\prod_{k=1}^{N-1}c_{\alpha_k}(k),
 \qquad
 \cM_N=\bigoplus_{0\leq\alpha_k\leq k}
 \ZZ\mathcal E_\alpha(c).
\]
Fulton's specialization
\begin{equation}\label{eq:Fulton-classical-specialization}
 c_r(k)\longmapsto e_r(x_1,\ldots,x_k)
\end{equation}
is injective on $\cM_N$~\cite[Lemma~2.1]{Fulton}. The universal
Schubert polynomials form a $\ZZ$-basis of this module.
Fulton defines additive endomorphisms, also denoted by $\partial_i$,
that commute with this specialization:
\[ \partial_i:\cM_N\longrightarrow\cM_N \qquad(1\leq i<N). \]
To give their formula in the $c$-coordinates, fix all factors of a
standard-elementary monomial except those at levels $i-1$ and $i$, and
write
\[ [a,b]_i= c_a(i-1)c_b(i) \prod_{k\ne i-1,i}c_{\alpha_k}(k). \]
We use $c_0(0)=1$ and $c_r(0)=0$ for $r\ne0$; as usual, every term with
an index outside its allowed range is zero. Fulton's formula is
\begin{equation}\label{eq:Fulton-divided-difference}
\partial_i[a,b]_i=
\begin{cases}
 \displaystyle
 \sum_{t\geq0}[a+t,b-1-t]_i
 -\sum_{t\geq1}[b-1-t,a+t]_i,
 &a\geq b-1,\\[3mm]
 \displaystyle
 \sum_{t\geq0}[b-1+t,a-t]_i
 -\sum_{t\geq1}[a-t,b-1+t]_i,
 &a\leq b-2.
\end{cases}
\end{equation}
The sums are finite under the preceding convention. Under the
specialization \eqref{eq:Fulton-classical-specialization}, the image of
$\partial_i f$ is the usual divided difference in the $x$-variables applied to the image of $f$.
In particular these operators satisfy
the nil--Coxeter relations and
\[
 \partial_i\Sch_w(c)=
 \begin{cases}
  \Sch_{ws_i}(c),&\ell(ws_i)=\ell(w)-1,\\
  0,&\ell(ws_i)=\ell(w)+1.
 \end{cases}
\]
Starting from the top polynomial
$\Sch_{w_0}(c)=\prod_{k=1}^{N-1}c_k(k)$, this is Fulton's
recursive definition of the single universal Schubert polynomials
\cite[formula~(6)]{Fulton}.

For $w\in S_N$, define the {\em inversion-end code}
\[ e_j(w)=\#\{i<j:w_i>w_j\},\qquad 1\leq j\leq N, \]
and set
\[ \Code_w=\prod_{j=2}^{N}c_{e_j(w)}(j-1). \]
The map $w\mapsto(e_2(w),\ldots,e_N(w))$ is a bijection from $S_N$ to
the sequences $0\leq e_j\leq j-1$. Hence the $\Code_w$ are the
standard-elementary basis reindexed by permutations.

The following proposition is the iterated form of Sottile's elementary
Pieri rule~\cite{Sottile}, lifted to Fulton's universal coordinates.

\begin{proposition}\label{prop:iterated-Pieri-transition}
For every $w\in S_N$,
\[ \Code_w=\sum_{u\in S_N}\Pi_N(w,u)\Sch_u(c), \qquad \Pi_N(w,u)\in\NN,\qquad \Pi_N(w,w)=1. \]
Here $\Pi_N(w,u)$ counts the iterated Sottile--Pieri chains
prescribed by the inversion-end code of $w$.
\end{proposition}

\begin{proof}
Apply Sottile's elementary Pieri rule~\cite{Sottile} to the factors
of $\Code_w$ in the order $j=2,3,\ldots,N$. This gives the stated
identity after the classical specialization
\eqref{eq:Fulton-classical-specialization}. Both sides lie in
$\cM_N$, so Fulton's injectivity lifts the identity. Fulton's
leading-term theorem gives the diagonal coefficient one.
\end{proof}

Let $w_0=w_0^{(N)}$ be the longest element of $S_N$. For $\pi\in S_N$, let
\[ L(\pi)_i=\#\{j>i:\pi_i>\pi_j\} \]
be its {\em Lehmer code}.

\begin{lemma}[Cauchy transpose]
\label{prop:Cauchy-transpose-Pieri}
For $\pi,\sigma\in S_N$, we have
\[ \Pi_N(\pi w_0,\sigma w_0) =[x^{L(\pi)}]\Sch_\sigma(x). \]
Every monomial exponent of every $\Sch_\sigma(x)$ is the Lehmer
code of a unique permutation in $S_N$.
\end{lemma}

\begin{proof}
Expand the finite Schubert Cauchy identity
\[ \prod_{i+j\leq N}(x_i+y_j) =\sum_{u\in S_N}\Sch_u(x)\Sch_{uw_0}(y) \]
using the elementary products defining $\Pi_N$; this gives the stated
coefficient identity. See \cite{FominStanley} for the Cauchy formula.
The pipe-dream formula places every exponent in the staircase
$0\leq a_i\leq N-i$, and such vectors are precisely the Lehmer codes of
permutations in $S_N$; see \cite{BJS}.
\end{proof}

\section{The triangular seed and the Schubert polytope}\label{sec:triangular-seed}

\subsection{Cluster conventions}\label{sec:cluster-conventions}
We follow Fomin and Zelevinsky \cite{FominZelevinskyI,FominZelevinskyIV} for seeds, mutation, cluster monomials,
coefficients, $F$-polynomials, and $\g$-vectors. We use ice seeds
$\Sigma=(\mathbf x,B)$, with mutable index set
$\mu$ and frozen index set $\op{fr}$; the rows of $B$ are indexed by
$\mu$, and its columns by $\mu\sqcup\op{fr}$. Frozen variables are
regular functions and are not inverted unless a localization is written
explicitly. Thus $\mathcal U(\Sigma)$ denotes the upper cluster algebra
with polynomial coefficients.

For $i\in\mu$, set
$\widehat y_i=\prod_{j\in\mu\sqcup\op{fr}}x_j^{b_{ij}}$.
A nonzero element is \emph{pointed at}
$\g\in\mathbb Z^{\mu\sqcup\op{fr}}$ if its Laurent expansion
in the initial seed has the form
\[ \mathbf x^{\g}\left(1+ \sum_{0\ne n\in\mathbb N^{\mu}}a_n\widehat y^{\,n}\right). \]
Throughout the paper, $\g$-vectors include the frozen coordinates; their
projection to the mutable coordinates is denoted by $\g_\mu$. An
\emph{extended cluster monomial} is a cluster monomial multiplied by a
monomial in the frozen variables with nonnegative exponents. An
\emph{extended cluster variable} means either a mutable cluster variable
or one frozen variable.

\subsection{The standard triangular seed}

We call the following ice seed the \emph{standard triangular seed}. Set
\[ \mathcal I_N=\{(a,r)\in\ZZ_{>0}^2:a+r\leq N\}. \]
For $(a,r)\in\mathcal I_N$, define the flag minor
\begin{equation}\label{eq:unipotent-seed-variable}
 x_{a,r}=\Delta_{[1,a],[r+1,r+a]}(Z(c)).
\end{equation}
The vertices with $a+r=N$ are frozen. The mutable indices are
$\mu_N=\{(a,r)\in\mathcal I_N:a+r<N\}$. The arrows are
\[ (a,r)\longrightarrow(a+1,r-1),\qquad (a,r)\longrightarrow(a-1,r),\qquad (a,r)\longrightarrow(a,r+1), \]
whenever the indicated vertices belong to $\mathcal I_N$.
Thus $b_{ij}$ is the number of arrows $i\to j$ minus the number
of arrows $j\to i$. We use the
polynomial coefficient convention: frozen variables are not inverted
unless we explicitly localize.

For $(a,r)\in\mathcal I_N$, set
\[ w_{a,r}=(a+1,\ldots,a+r,1,\ldots,a,a+r+1,\ldots,N). \]
This is the Grassmannian permutation of rectangular shape $(a^r)$. Under
the classical specialization, $\Sch_{w_{a,r}}$ becomes
$Q_{r,a}=s_{(a^r)}$.
The triangular adjacency is related to the type-$A$ $Q$-system recurrence; 
see \cite{DiFrancescoKedem2010}.

\FloatBarrier
\begin{figure}[ht]
\centering
\begin{tikzpicture}[
 every node/.style={font=\scriptsize},
 seed/.style={draw,rounded corners=1pt,inner sep=1.6pt,
 minimum width=0.72cm,minimum height=0.40cm},
 frozen/.style={seed,double,double distance=0.7pt},
 qarrow/.style={-{Latex[length=1.55mm]},thin}
]
 \node[frozen] (x14) at (0,3.6) {$x_{1,4}$};

 \node[seed] (x13) at (-0.9,2.7) {$x_{1,3}$};
 \node[frozen] (x23) at (0.9,2.7) {$x_{2,3}$};

 \node[seed] (x12) at (-1.8,1.8) {$x_{1,2}$};
 \node[seed] (x22) at (0,1.8) {$x_{2,2}$};
 \node[frozen] (x32) at (1.8,1.8) {$x_{3,2}$};

 \node[seed] (x11) at (-2.7,0.9) {$x_{1,1}$};
 \node[seed] (x21) at (-0.9,0.9) {$x_{2,1}$};
 \node[seed] (x31) at (0.9,0.9) {$x_{3,1}$};
 \node[frozen] (x41) at (2.7,0.9) {$x_{4,1}$};

 \draw[qarrow] (x13)--(x14);
 \draw[qarrow] (x12)--(x13);
 \draw[qarrow] (x22)--(x23);
 \draw[qarrow] (x11)--(x12);
 \draw[qarrow] (x21)--(x22);
 \draw[qarrow] (x31)--(x32);

 \draw[qarrow] (x23)--(x13);
 \draw[qarrow] (x22)--(x12);
 \draw[qarrow] (x32)--(x22);
 \draw[qarrow] (x21)--(x11);
 \draw[qarrow] (x31)--(x21);
 \draw[qarrow] (x41)--(x31);

 \draw[qarrow] (x13)--(x22);
 \draw[qarrow] (x12)--(x21);
 \draw[qarrow] (x22)--(x31);
\end{tikzpicture}
\caption{The standard triangular seed on $U_5$. Double boxes are frozen,
and arrows between frozen vertices are omitted.}
\label{fig:unipotent-triangular-seed}
\end{figure}
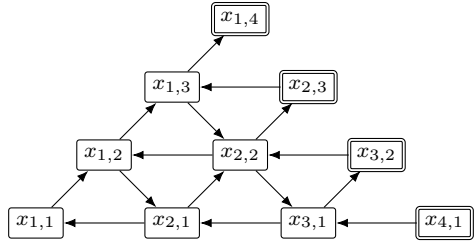
\FloatBarrier

The next proposition recalls the generalized-minor cluster
structure and identifies its initial minors with universal Schubert
polynomials.

\begin{proposition}\label{thm:unipotent-seed}
Let $\Sigma_N^U$ be the ice seed above. The generalized-minor
realization induces an isomorphism
\[ \mathcal U(\Sigma_N^U)\xrightarrow{\ \sim\ }\CC[U_N]. \]
Moreover,
\begin{equation}\label{eq:initial-single-Schubert-label}
 x_{a,r}=\Sch_{w_{a,r}}(c).
\end{equation}
\end{proposition}

\begin{proof}
The maximal unipotent subgroup carries the generalized-minor cluster
structure. In the present indexing, the initial functions are the minors
\eqref{eq:unipotent-seed-variable}. The equality with the upper cluster
algebra with polynomial coefficients follows from the type-$A$ unipotent
construction; see \cite{GLSinitial}.

The permutation $w_{a,r}$ is Grassmannian of rectangular shape $(a^r)$.
Fulton's Grassmannian determinant gives
$\Sch_{w_{a,r}}(c) =\Delta_{[1,a],[r+1,r+a]}(Z(c))$,
which proves \eqref{eq:initial-single-Schubert-label}.
\end{proof}

\begin{definition}\label{def:reachable}
All reachability in this paper is relative to the standard triangular seed
$\Sigma_N^U$. A vector $\g\in\mathbb Z^{\mathcal I_N}$ is
\emph{reachable} if there is a seed obtained from $\Sigma_N^U$ by a finite
sequence of mutations and an extended cluster monomial in that seed whose
$\g$-vector, computed relative to $\Sigma_N^U$, is $\g$.
\end{definition}
A preprojective representation is called \emph{reachable} if it is a
(possibly decomposable) summand of a cluster-tilting object reached
from the initial cluster-tilting object by finitely many mutations.
Such a representation is rigid, since cluster-tilting objects have
vanishing self-$\Ext^1$. Thus reachability for representations
implies rigidity; reachability for vectors is the condition in
Definition~\ref{def:reachable}.
\subsection{Schubert vertices}
\label{sec:schubert-gvectors}
For $w\in S_N$, define
\begin{equation}\label{eq:Schubert-g-vector}
 \g_w(a,r)=
 \one_{\{e_{a+r}(w)=r\}}-
 \one_{\{e_{a+r+1}(w)=r\}},
\end{equation}
where an out-of-range indicator is zero. The inverse tropical Chamber
Ansatz for the code monomial gives
\[ \g_{\Sigma_N^U}(\Code_w)=\g_w. \]
We call $\g_w$ the \emph{Schubert vertex}. It is the full initial
tropical vector of the code monomial $\Code_w$. The Schubert polynomial
$\Sch_w(c)$ need not be homogeneous under diagonal conjugation or
pointed at this vector.

For $2\leq j\leq N$, let $E_j\simeq\ZZ^{j-1}$ have basis
$\mathbf e_{j,1},\ldots,\mathbf e_{j,j-1}$, and set $\mathbf e_{j,0}=0$. For
$\g\in\ZZ^{\mathcal I_N}$ and $1\leq r<j\leq N$, set
\[ \eta_{j,r}(\g)=\sum_{k=j}^{N}\g(k-r,r). \]
Then
\begin{equation}\label{eq:eta-Schubert-code}
 \eta_{j,r}(\g_w)=\one_{\{e_j(w)=r\}}.
\end{equation}
Indeed, \eqref{eq:Schubert-g-vector} telescopes in $k$. Thus the integral
linear map
\[
\Phi_N:\ZZ^{\mathcal I_N}\longrightarrow\bigoplus_{j=2}^{N}E_j,
 \qquad
 \Phi_N(\g)=\sum_{j=2}^{N}\sum_{r=1}^{j-1}\eta_{j,r}(\g)\mathbf e_{j,r},
\]
has inverse
\[ \g(a,r)=\eta_{a+r,r}-\eta_{a+r+1,r}, \qquad \eta_{N+1,r}=0. \]
Hence $\Phi_N$ is unimodular and
\[ \Phi_N(\g_w)=\sum_{j=2}^{N}\mathbf e_{j,e_j(w)}. \]
Thus the Schubert vertices are the full initial tropical vectors of
the standard-elementary basis $\{\Code_w:w\in S_N\}$. 
\begin{definition}
We call $\mathcal P_N^{\mathrm{Sch}}=\op{conv}\{\g_w:w\in S_N\}$
the \emph{Schubert polytope} of rank $N$. 
\end{definition}
The generic, dual canonical, and MV elements of
Section~\ref{sec:bases-standard-elementary} are indexed by these same
vectors through the inverse Chamber Ansatz.

\begin{theorem}\label{thm:Schubert-polytope}
The convex hull of the Schubert vertices is integrally unimodularly
equivalent to
\[ \Delta_1\times\Delta_2\times\cdots\times\Delta_{N-1}. \]
Its lattice points are exactly its $N!$ vertices $\g_w$.
\end{theorem}

\begin{proof}
The inversion-end code is a bijection
\[ S_N\longrightarrow \{0,1\}\times\{0,1,2\}\times\cdots\times\{0,\ldots,N-1\}. \]
By \eqref{eq:eta-Schubert-code}, the $j$th component of
$\Phi_N(\g_w)$ is one of
$0,\mathbf e_{j,1},\ldots,\mathbf e_{j,j-1}$, 
independently for each $j$. Their convex hull is the standard simplex
$\Delta_{j-1}$. Since $\Phi_N$ is unimodular, the Schubert polytope is
the stated product. An integral point of a standard simplex has
nonnegative integral coordinates with sum at most one, and is therefore a
vertex. The same is true componentwise for the product.
\end{proof}

\section{Right-regular geometry and root-degree fibers}
\label{sec:right-regular-representations}

\label{par:Schubert-label-notation}
We now identify the Schubert vertices with right-regular
multisegments. The forward module $M_w$, its conormal component
$Z_{\mathfrak m_w}$, and a general preprojective module
$X_w\in Z_{\mathfrak m_w}$ will be used throughout the paper.

\subsection{The equioriented realization}

For a rank word $r=(r_1,\ldots,r_N)$, with $1\leq r_j\leq j$, set
\begin{equation}\label{eq:rank-word-root-degree}
 d_i(r)=\#\{j:r_j\leq i<j\},\qquad 1\leq i<N.
\end{equation}
We call $d(r)=(d_1(r),\ldots,d_{N-1}(r))$ the \emph{root degree}
of $r$. In the equioriented realization below, this is the dimension
vector of the associated representation. Its root-lattice weight is
$\sum_{i=1}^{N-1}d_i(r)\alpha_i$. For $w\in S_N$, set
\[ r_j(w)=j-e_j(w),\qquad d_i(w)=d_i(r(w)). \]
The inverse tropical Chamber Ansatz identifies the presentation weight
$-\g_w$ with the PBW datum
\begin{equation}\label{eq:schubert-multisegment}
 \mathfrak m_w=
 \sum_{\substack{2\leq j\leq N\\ e_j(w)>0}}
 [r_j(w),j-1]
\end{equation}
for the equioriented quiver
\[ Q=Q_{N-1}:1\longleftarrow2\longleftarrow\cdots\longleftarrow N-1. \]
Let $M_w$ be the equioriented representation with interval decomposition
\eqref{eq:schubert-multisegment}. We sometimes record an empty interval to keep track of its final
index. By construction,
\[ d_i(w)=d_i(r(w))=\dim(M_w)_i =\#\{j:r_j(w)\leq i<j\}. \]

Let $\Lambda_Q$ be the preprojective algebra of $Q_{N-1}$. For a
dimension vector $\mathbf d$, let
\[ E_{\mathbf d}=\op{Rep}(Q_{N-1},\mathbf d) \]
be the equioriented representation space, and let $\Lambda_{\mathbf d}$
be Lusztig's nilpotent variety. If $\mathfrak m$ is a multisegment of
dimension vector $\mathbf d$, denote by
$\cO_{\mathfrak m}\subset E_{\mathbf d}$ its orbit and set
\[ Z_{\mathfrak m} =\overline{T^*_{\cO_{\mathfrak m}}E_{\mathbf d}} \subset\Lambda_{\mathbf d}. \]
We write $Z_{\g_w}=Z_{\mathfrak m_w}$ and let $X_w$ denote a general point of this component.

A multisegment is called \emph{right-regular} if at most one of its
segments ends at each vertex. The map $w\mapsto\mathfrak m_w$ is a
bijection from $S_N$ to the right-regular multisegments supported on
$1,\ldots,N-1$.

For a representation $M$ of the equioriented quiver, set
\[ t_i(M)=\dim\coker(M_{i+1}\longrightarrow M_i). \]
In the interval decomposition of $M$, the integer $t_i(M)$ is the number of
segments ending at $i$.

\begin{lemma}\label{lem:right-regular-generization}
Suppose $\cO_M\subseteq\overline{\cO_{M'}}$. If $M$ is right-regular,
then $M'$ is right-regular.
\end{lemma}

\begin{proof}
The rank of each arrow is lower semicontinuous. Hence
\[ \rank(M'_{i+1}\to M'_i)\geq\rank(M_{i+1}\to M_i). \]
The dimension vector is fixed, so $t_i(M')\leq t_i(M)$. If all the latter
integers are at most one, the same is true for $M'$.
\end{proof}

\subsection{Rank words and parabolic Bruhat intervals}

For a rank word $r=(r_1,\ldots,r_N)$, set $d_0=d_N=0$ and use
the root degree $d(r)$ of \eqref{eq:rank-word-root-degree}.

For $p<q$, suppose $r_p<r_q\leq p$.
Interchanging $r_p$ and $r_q$ is the right-regular form of the elementary
Zelevinsky operation \cite{Zelevinsky}
\[ [r_p,p-1]+[r_q,q-1] \longmapsto [r_q,p-1]+[r_p,q-1]. \]
We orient this move from the first rank word to the second and call the
resulting directed graph the \emph{right-regular linked-move graph}.
Under the rank-word identification, each directed edge increases
multiset Bruhat order. We use
the opposite convention for the Zelevinsky order:
\[ r(w)\leq_{\mathrm B}r(z) \quad\Longleftrightarrow\quad z\preceq_{\rm Z}w. \]
Thus the graph records elementary Zelevinsky incidences, but its arrows
point downward in $\preceq_{\rm Z}$.

Let
\[ m_i(r)=\#\{j:r_j=i\}. \]

\begin{lemma}\label{lem:d-content}
One has
\[ m_i(r)=1+d_i(r)-d_{i-1}(r). \]
Consequently, two rank words have the same root degree if and only if
they have the same multiset of entries.
\end{lemma}

\begin{proof}
Since $r_j\leq j$, all positions $j\leq i$ contribute to
$\#\{j:r_j\leq i\}$. The remaining contributions are the
segments crossing vertex $i$. Therefore
$\#\{j:r_j\leq i\}=i+d_i(r)$.
Taking successive differences gives the formula.
\end{proof}

\begin{lemma}\label{lem:root-degree-fiber-connectivity}
The connected components of the underlying undirected linked-move
graph are precisely the root-degree fibers.
\end{lemma}

\begin{proof}
A linked move preserves the multiset of entries and hence preserves $d$ by
Lemma~\ref{lem:d-content}. Conversely, suppose $p<q$ and $r_p>r_q$.
Interchanging these two entries gives another rank word, because
$r_q<r_p\leq p<q$. This is the reverse of a linked move. Repeatedly
removing inversions connects any rank word to the weakly increasing
arrangement of its entries.
\end{proof}

We call $d=(d_1,\ldots,d_{N-1})$ \emph{admissible} when $d_i\geq0$ and
$d_i-d_{i-1}\geq-1$ for $1\leq i\leq N$, where $d_0=d_N=0$.
Fix such a vector $d$, and set
$m_i=1+d_i-d_{i-1}$.
Rank words of root degree $d$ are multiset permutations of
$1^{m_1}\cdots N^{m_N}$ satisfying $r_j\leq j$. They may be viewed as
minimal representatives for
$S_N/(S_{m_1}\times\cdots\times S_{m_N})$.
Let $r_d^{\min}$ be the weakly increasing word. To construct
$r_d^{\max}$, choose at position $p$ the largest unused entry not
exceeding $p$.

\begin{proposition}\label{prop:bruhat-interval}
In multiset Bruhat order, the rank words of root degree $d$ form
the interval $[r_d^{\min},r_d^{\max}]$. 
The transitive closure of the linked moves is the Bruhat order on this
interval.
\end{proposition}

\begin{proof}
The condition $r_j\leq j$ is downward closed in multiset Bruhat order.
Indeed, a downward cover exchanges an inversion $b,a$ at positions $p<q$
into $a,b$. The upper word satisfies the bounds, so $b\leq p$, and the lower word satisfies them as well.

Let $r$ satisfy the bounds and differ from $r_d^{\max}$, and let $p$ be the first
position where the two words differ. The greedy entry $b$ at $p$ is an
unused entry $b$ with $r_p<b\leq p$ that occurs later in $r$.
Interchanging the two entries increases Bruhat order and preserves
the bounds. Repeating this move reaches $r_d^{\max}$.
Hence every word satisfying the bounds lies in the stated interval.

A Bruhat cover inside the interval moves a larger entry $b$ to an earlier
position $p$. The upper word satisfies the bounds, so $b\leq p$; the cover is therefore
a linked move. Conversely, every linked move is a Bruhat increase. This
proves the last assertion.
\end{proof}

\subsection{Extremal elements and pattern avoidance}

Write $I_d=[r_d^{\min},r_d^{\max}]$ for the interval of
Proposition~\ref{prop:bruhat-interval}. Its minimum and maximum are
the unique source and sink of the linked-move graph. The following
lemma identifies their permutations by pattern avoidance.

\begin{lemma}\label{lem:rank-patterns}
For $w\in S_N$:
\begin{enumerate}[(i)]
 \item $r(w)$ is weakly increasing if and only if $w$ avoids $231$;
 \item no pair $p<q$ satisfies $r_p(w)<r_q(w)\leq p$ if and only if $w$
 avoids $312$.
\end{enumerate}
\end{lemma}

\begin{proof}
Suppose $p<q$ and $r_p>r_q$. Then $w_p>w_q$. If there were no
$i<p$ with $w_q<w_i<w_p$, every earlier entry below $w_p$ would also lie
below $w_q$, giving $r_p\leq r_q$. Hence such an $i$ exists and
$i,p,q$ form a $231$ pattern.

Conversely, choose an occurrence $i<p<q$ of $231$ with $q$ minimal. No
entry between $p$ and $q$ is smaller than $w_q$, for otherwise there would
be an occurrence with smaller final position. Hence all entries before
$q$ and below $w_q$ occur before $p$, while $w_i$ is below $w_p$ but above
$w_q$. Thus $r_p>r_q$. This proves (i).

Suppose next that $p<q$ and $r_p<r_q\leq p$. If $w_p<w_q$, then
$e_q=q-r_q\geq q-p$. Since there are only $q-p-1$ positions strictly
between $p$ and $q$, some $i<p$ satisfies $w_i>w_q>w_p$, giving a $312$
pattern. If $w_p>w_q$, the inequality $r_p<r_q$ forces an index
$p<t<q$ with $w_t<w_q$; the bound $r_q\leq p$ then yields $i<p$ with
$w_i>w_q$, and $i,t,q$ form a $312$ pattern.

Conversely, choose $i<p<q$ with $w_i>w_q>w_p$, first with $q$ minimal and
then with $p$ maximal. The inequality $w_p<w_q$ gives $r_p<r_q$. The
choices of $p,q$ imply that every entry strictly between $p$ and $q$ is
larger than $w_q$. Together with $w_i>w_q$, this gives
$e_q\geq q-p$, and hence $r_q\leq p$. This proves (ii).
\end{proof}

\subsection{The staircase quiver Grassmannian}

Let $Q_{N-1}:1\longleftarrow2\longleftarrow\cdots\longleftarrow N-1$,
and let $P_j=[1,j]$ be the indecomposable projective at $j$. Set
\[ P_\triangle=\bigoplus_{j=1}^{N-1}P_j. \]
A representation is right-regular if and only if it is a quotient of
$P_\triangle$: the multiplicity of the simple top at $j$ is the number of
segments ending at $j$.

For a dimension vector $d$, define the \emph{staircase quiver Grassmannian}
\[ X_d=\Gr_{\underline{\dim}P_\triangle-d}(P_\triangle) \cong \Quot_d(P_\triangle). \]
We use the quotient realization when discussing right-regular quotient strata.

Let $\operatorname{Cat}_N=\frac{1}{N+1}\binom{2N}{N}$ be the $N$th Catalan number.

\begin{theorem}\label{thm:right-regular-chamber-ansatz}
\begin{enumerate}[(i)]
 \item The underlying undirected linked-move graph has exactly
$\operatorname{Cat}_N$ connected components, indexed by
 the vectors
 \[ d_0=d_N=0,\qquad d_i\geq0,\qquad d_i-d_{i-1}\geq-1. \]
 \item Under the rank-word indexing, the root-degree fiber $d$ is the
parabolic Bruhat interval $[r_d^{\min},r_d^{\max}]$. The source
and sink correspond, respectively, to the unique $231$-avoiding and
$312$-avoiding permutations of root degree $d$.
 \item The variety $X_d$ is nonempty exactly for these vectors. It is a
 smooth projective iterated Grassmannian bundle with successive fibers
 \[ \Gr(d_i,1+d_{i+1}),\qquad i=N-1,N-2,\ldots,1. \]
 Its strata of fixed quotient isomorphism type are affine Schubert
 cells indexed by the rank words of root degree $d$; their closure
 order is the Bruhat order on the interval above.
 \item Its Poincar\'e polynomial, with $t$ recording complex cell
 dimension, and its Euler characteristic are
 \[
  \sum_{w:d(w)=d}t^{\inv(r(w))}
  =\prod_{i=1}^{N-1}\qbinom{1+d_{i+1}}{d_i}_t,
  \qquad
  \chi(X_d)=\prod_{i=1}^{N-1}\binom{1+d_{i+1}}{d_i}.
 \]
\end{enumerate}
\end{theorem}

\begin{proof}
Parts (i) and (ii) follow from Lemmas~\ref{lem:d-content},
\ref{lem:root-degree-fiber-connectivity}, and \ref{lem:rank-patterns}, together with
Proposition~\ref{prop:bruhat-interval}. The admissible vectors are
\L{}ukasiewicz excursions, and are counted by $\operatorname{Cat}_N$.

For (iii), identify
\[ P_{\triangle,i}=V_i=\langle e_i,\ldots,e_{N-1}\rangle, \]
with the arrows $V_{i+1}\hookrightarrow V_i$. A quotient of dimension
$d$ is determined by a kernel flag
\[ K_{N-1}\subseteq\cdots\subseteq K_1, \qquad K_i\subseteq V_i, \qquad \dim K_i=N-i-d_i. \]
Given $K_{i+1}$, the quotient $K_i/K_{i+1}$ has dimension
$1+d_{i+1}-d_i$ in the $(1+d_{i+1})$-dimensional space
$V_i/K_{i+1}$. This gives the stated Grassmannian fiber and the
nonemptiness inequalities.

Finally, the isomorphism class of an equioriented type-$A$ representation
is determined by its composition ranks. For $1\le i\le j<N$,
\[ \rank\bigl((P_\triangle/K)_j\longrightarrow(P_\triangle/K)_i\bigr) =\dim V_j-\dim(K_i\cap V_j)=N-j-\dim(K_i\cap V_j). \]
Indeed, the image is $V_j/(K_i\cap V_j)$. These intersection
dimensions are the Schubert rank data of the kernel flag, so the quotient
strata are the Schubert cells, with the asserted closure order.
The Grassmannian bundle description gives the product formula in
part~\textup{(iv)}. Evaluation at $t=1$ gives the Euler characteristic.
\end{proof}

Let $E_d^{\rm rr}$ be the union of right-regular orbits in the
representation space of dimension vector $d$, and set
\[ \widetilde E_d^{\rm rr} =\{(M,\phi):M\in E_d^{\rm rr},\ \phi:P_\triangle\twoheadrightarrow M\}. \]
There is a correspondence
\[ E_d^{\rm rr}\xleftarrow{\ p\ } \widetilde E_d^{\rm rr} \xrightarrow{\ q\ }X_d. \]

\begin{proposition}\label{prop:framed-right-regular}
The map $p$ is smooth; its fiber over $M$ is the open set of
surjective maps in $\Hom_Q(P_\triangle,M)$. The map $q$ is a
principal $G_d=\prod_i\op{GL}_{d_i}$-bundle. In particular,
the relative closures $\overline{\cO_M}\cap E_d^{\rm rr}$
and the corresponding Schubert-cell closures in $X_d$ are smoothly
equivalent; equivalently, their germs along right-regular strata are
smoothly equivalent.
\end{proposition}

\begin{proof}
The total space $\widetilde E_d^{\rm rr}$ is the open subset of
the vector bundle $\Hom_Q(P_\triangle,-)$ consisting of surjections,
so $p$ is smooth. Changing bases in the quotients makes $q$ a
principal $G_d$-bundle; see \cite{ReinekeFramed}.
\end{proof}

We distinguish rigidity for the forward quiver from rigidity for its
preprojective algebra. A representation $M$ of the
oriented quiver $Q$ is called \emph{$Q$-rigid} if
$\Ext_Q^1(M,M)=0$; this is equivalent to the $G_d$-orbit of $M$ being
open in its representation variety. An irreducible component $Z$ in the
preprojective or conormal space is called \emph{$\Lambda_Q$-rigid} if a
general point $X\in Z$ satisfies
$\Ext_{\Lambda_Q}^1(X,X)=0$. These are distinct conditions; the
gap hypothesis below will imply
rigidity over $\Lambda_Q$ without requiring rigidity over $Q$.

\begin{proposition}\label{prop:self-extension-codimension}
Let $\displaystyle M_w=\bigoplus_{\substack{2\leq j\leq N\\r_j(w)<j}}
 [r_j(w),j-1]$ and $\zeta(w)=\inv(r(w))$.
Then
\begin{equation}\label{eq:self-extension-linked-pairs}
 \dim\Ext_Q^1(M_w,M_w)
 =\#\{(p,q):p<q,\ r_p(w)<r_q(w)\leq p\}
 =\dim X_{d(w)}-\zeta(w).
\end{equation}
Thus the stratum indexed by $w$ has codimension
$\dim\Ext_Q^1(M_w,M_w)$ in $X_{d(w)}$. In particular, the $312$-avoiding element in a root-degree fiber
indexes the unique open quotient stratum and its $Q$-rigid isomorphism
class.
\end{proposition}

\begin{proof}
For the quiver $1\leftarrow\cdots\leftarrow N-1$, the projective
resolution
\[ 0\longrightarrow P_{a-1}\longrightarrow P_b \longrightarrow[a,b]\longrightarrow0 \]
gives
\[
 \dim\Ext_Q^1([a,b],[c,d])=
 \begin{cases}
  1,&c<a\leq d+1\leq b,\\
  0,&\text{otherwise}.
 \end{cases}
\]
Applied to the interval summands of $M_w$, this condition is exactly
$r_p(w)<r_q(w)\leq p$ for $p<q$, proving the first equality in
\eqref{eq:self-extension-linked-pairs}.

By Proposition~\ref{prop:framed-right-regular}, the quotient stratum and the
orbit $\cO_{M_w}$ in the open right-regular locus are smoothly equivalent,
so they have the same codimension. For a representation $M$ of a hereditary
quiver, the Euler-form identity, together with
$\dim\cO_M=\dim G_d-\dim\End(M)$, gives
\[ \op{codim}_{\op{Rep}(Q,d)}\cO_M =\dim\Ext_Q^1(M,M). \]
The right-regular locus is open, hence the same formula computes the
codimension in $X_d$. The Schubert-cell description in
Theorem~\ref{thm:right-regular-chamber-ansatz} identifies the stratum
dimension with $\inv(r(w))$, proving the second equality. The final
assertion follows from the characterization of the sink in
Lemma~\ref{lem:rank-patterns}.
\end{proof}

\subsection{The root-degree projection of the Schubert polytope}

The linear root-degree map sends the Schubert polytope to the convex
hull of the admissible root degrees. The generating function below
counts the vertices in each root-degree fiber.

Introduce independent variables $\xi_{a,j}$ for $1\leq a\leq j\leq N-1$
and set
\begin{equation}\label{eq:separated-Schubert-generating-function}
 \widetilde{\mathcal F}_N(\boldsymbol\xi)
 =\prod_{j=1}^{N-1}\left(1+\sum_{a=1}^{j}\xi_{a,j}\right).
\end{equation}
Its Newton polytope is
$\Delta_1\times\cdots\times\Delta_{N-1}$. Under the substitution
\begin{equation}\label{eq:root-degree-substitution}
 \xi_{a,j}\longmapsto u_a u_{a+1}\cdots u_j,
\end{equation}
write the resulting polynomial as $\mathcal F_N(\mathbf u)$.

\begin{corollary}\label{prop:root-degree-projection}
Under the unimodular identification of Theorem~\ref{thm:Schubert-polytope},
\eqref{eq:root-degree-substitution} is the linear map
$\g_w\mapsto d(w)$. Its root-degree fiber on the vertices is the interval $I_d$. Moreover,
\[
\mathcal F_N(\mathbf u)
 =\prod_{j=1}^{N-1}
  \left(1+u_j+u_{j-1}u_j+\cdots+u_1u_2\cdots u_j\right)
 =\sum_d \chi(X_d)\mathbf u^d,
\]
where $d_N=0$ and the sum is over the admissible root degrees. The
product formula for $\chi(X_d)$ is given in
Theorem~\ref{thm:right-regular-chamber-ansatz}\textup{(iv)}.
\end{corollary}

\begin{proof}
A monomial of \eqref{eq:separated-Schubert-generating-function} chooses,
for each right endpoint $j$, either no interval or one interval $[a,j]$.
These choices are precisely the right-regular multisegments, and
hence are indexed by permutations. The interval $[a,j]$ contributes one
to $d_i$ exactly for $a\leq i\leq j$, which proves the assertion about the
linear map and its root-degree fibers. The coefficient of $\mathbf u^d$ therefore counts the permutations of
root degree $d$, and Theorem~\ref{thm:right-regular-chamber-ansatz}\textup{(iv)} identifies this number
with $\chi(X_d)$.
\end{proof}

\section{Geometric bases and right-root strings}
\label{sec:bases-standard-elementary}\label{sec:root-crystal-actions}
\subsection{Three bases on the Schubert vertex set}

For an irreducible component $Z\subset\Lambda_{\mathbf d}$, let
$\rho_Z\in\CC[U_N]$ be the corresponding dual semicanonical function.
Write $b^*_{\mathfrak m}$ for the dual canonical basis element with
Lusztig datum $\mathfrak m$, and write $b^{\rm MV}_{\mathfrak m}$ for
the MV basis element corresponding to the same crystal element.

For every $w\in S_N$, we set
\[
 Z_{\g_w}:=Z_{\mathfrak m_w},\qquad
 \Gen_w:=\rho_{Z_{\mathfrak m_w}},\qquad
 \Can_w:=b^*_{\mathfrak m_w},\qquad
 \MV_w:=b^{\rm MV}_{\mathfrak m_w}.
\]
\begin{proposition}\label{prop:Schubert-geometric-bases}
The function $\Gen_w$ is the generic basis element pointed at $\g_w$,
$\Can_w$ is the dual canonical element with Lusztig datum $\mathfrak m_w$,
and $\MV_w$ is the MV basis element corresponding to the same bicrystal
element. Thus the three families are indexed by the correspondence
$\g_w\longleftrightarrow\mathfrak m_w$.
\end{proposition}

\begin{proof}
The first correspondence is the inverse tropical Chamber Ansatz. The
comparison between generic cluster characters and dual semicanonical
functions identifies the element attached to this presentation weight with
$\rho_{Z_{\mathfrak m_w}}$; see \cite{GLSgeneric}. The dual canonical
basis is indexed by the same Lusztig data. Finally, the MV basis is a
biperfect basis indexed by the bicrystal $B(\infty)$ \cite{BKKMV}; we use
the crystal identification to denote by $b^{\rm MV}_{\mathfrak m_w}$ the
element whose Lusztig datum in the chosen equioriented word is
$\mathfrak m_w$.
\end{proof}

\begin{remark}For the corresponding quantum unipotent subgroup, Qin proved that, 
after normalization and localization at the frozen quantum minors, 
the dual canonical basis is the common triangular basis
\cite[Theorem~1.2.1\textup{(I)}]{QinTriangular}.
\end{remark}

For $\bullet\in\{\mathrm{gen},\mathrm{can},\mathrm{MV}\}$, set
\[ \mathsf B_w^{\mathrm{gen}}=\Gen_w,\qquad \mathsf B_w^{\mathrm{can}}=\Can_w,\qquad \mathsf B_w^{\mathrm{MV}}=\MV_w. \]
We write $X_w$ for a general point of $Z_{\g_w}=Z_{\mathfrak m_w}$.

\begin{theorem}\label{thm:geometric-bases-module}
Each of the three geometric families is an integral basis of
Fulton's standard-elementary module:
\[ \cM_N=\bigoplus_{w\in S_N}\ZZ\Gen_w =\bigoplus_{w\in S_N}\ZZ\Can_w =\bigoplus_{w\in S_N}\ZZ\MV_w. \]
Every element indexed by $w$ is homogeneous of root-lattice weight
$\sum_i d_i(w)\alpha_i$.
\end{theorem}
\begin{proof}
The dual PBW element with datum $\mathfrak m_w$ is $\Code_w$.
The PBW comparisons with the three geometric bases are integral and
unitriangular in the orbit-closure order. By
Lemma~\ref{lem:right-regular-generization}, every datum occurring in
these comparisons is again right-regular. The resulting finite
unitriangular changes of basis prove the assertion. We give the
comparison formulas and the MV support argument in
Appendix~\ref{app:classical-transitions}. Homogeneity follows from the
weight grading of the three bases.
\end{proof}

For $\bullet\in\{\mathrm{gen},\mathrm{can},\mathrm{MV}\}$, define
$D_N^\bullet$ by
\[ \Code_w=\sum_{v\in S_N}D_N^\bullet(w,v)\mathsf B_v^\bullet. \]
The explicit formulas for these transition matrices are given in
Appendix~\ref{app:classical-transitions}.

Kang--Kashiwara--Kim--Oh proved that every quantum cluster
monomial in the unipotent quantum coordinate algebra belongs, up to a
power of $q^{1/2}$, to the upper global basis \cite{KKKO}. After
specialization at $q=1$, every extended cluster monomial for the standard
cluster structure on $\CC[U_N]$ is therefore a dual canonical basis
element.

\begin{corollary}\label{cor:generic-cluster-monomial-canonical}
If $\Gen_w$ is an extended cluster monomial, then $\Gen_w=\Can_w$.
\end{corollary}

\begin{proof}
    By the theorem of Kang--Kashiwara--Kim--Oh, one has
    $\Gen_w=\Can_v$ for some $v\in S_N$.  Inverting the generic and canonical
    code transitions gives unitriangular code expansions with leading terms
    $\Code_w$ and $\Code_v$, respectively. Hence $v=w$.
\end{proof}

\subsection{Root actions and divided differences}

The right-root string statistics bound the number of nonconstant
factors of a geometric basis element. For $1\leq i<N$, set
$u_i(s)=I+sE_{i,i+1}$. For
$f\in\CC[U_N]$, define
\begin{align*}
 (\mathcal L_i(s)f)(Z)
 &=f(u_i(s)Z)=\sum_{r\geq0}s^r\mathcal L_i^{(r)}f(Z),\\
 (\mathcal R_i(s)f)(Z)
 &=f(Zu_i(s))=\sum_{r\geq0}s^r\mathcal R_i^{(r)}f(Z).
\end{align*}
On minors, $\mathcal L_i^{(1)}$ and $\mathcal R_i^{(1)}$ are the corresponding elementary row and column operations.

Write $\mathcal R_i=\mathcal R_i^{(1)}$. Right multiplication by
$u_i(s)$ changes only column $i+1$ of $Z(c)$, and hence
\begin{equation}\label{eq:right-root-on-c-coordinate}
 \mathcal R_i(s)c_r(k)
 =c_r(k)+\one_{\{k=i\}}s\,c_{r-1}(i-1).
\end{equation}
Thus
\[ \mathcal R_i c_r(k) =\one_{\{k=i\}}c_{r-1}(i-1). \]
The infinitesimal right-root operator is a derivation of
$\CC[U_N]$, whereas Fulton's divided difference is defined as an
additive endomorphism of the standard-elementary module. They agree on the coordinate generators, but
not on arbitrary products.

\begin{proposition}\label{prop:divided-difference-straightened-root-action}
For the standard-elementary monomial $[a,b]_i$ of
\eqref{eq:Fulton-divided-difference}, one has
\begin{equation}\label{eq:root-action-nonstandard-product}
 \mathcal R_i[a,b]_i
 =c_a(i-1)c_{b-1}(i-1)
  \prod_{k\ne i-1,i}c_{\alpha_k}(k).
\end{equation}
Fulton's divided difference $\partial_i[a,b]_i$ is the
standard-elementary expression in
\eqref{eq:Fulton-divided-difference}; its classical specialization equals
that of \eqref{eq:root-action-nonstandard-product}. Thus $\partial_i$ is
the standard-elementary lift of the right-root derivation, but the two
operators need not agree in $\CC[U_N]$. On a coordinate generator the
straightening is trivial and
\begin{equation}\label{eq:root-divided-difference-on-generator}
 \partial_i c_r(k)
 =\mathcal R_i c_r(k)
 =\one_{\{k=i\}}c_{r-1}(i-1).
\end{equation}
\end{proposition}

\begin{proof}
Equation~\eqref{eq:right-root-on-c-coordinate} follows directly from
column addition. Since every standard-elementary monomial has one factor
at level $i$, its coefficient of $s$ is
\eqref{eq:root-action-nonstandard-product} and all higher coefficients
vanish.

Let $X_{i-1}=(x_1,\ldots,x_{i-1})$. Every factor of the classical
specialization of $[a,b]_i$ other than the level-$i$ factor is invariant
under $s_i$, and
\[ \partial_i\bigl(e_a(X_{i-1})e_b(X_{i-1},x_i)\bigr) =e_a(X_{i-1})e_{b-1}(X_{i-1}). \]
The right-hand side is the classical specialization of
\eqref{eq:root-action-nonstandard-product}. By Fulton's compatibility,
it is also the specialization of
\eqref{eq:Fulton-divided-difference}. Taking all other factors equal to
$1$ proves \eqref{eq:root-divided-difference-on-generator}.
\end{proof}

\begin{remark}
Straightening is necessary because the right-root derivation
$\mathcal R_i$ need not preserve $\cM_N$. For example,
$\mathcal R_2\Sch_{321}(c)=c_1(1)^2$, whereas
$\partial_2\Sch_{321}(c) =\Sch_{312}(c) =c_1(1)c_1(2)-c_2(2)$.
Thus the two operators agree on the coordinate generators, while on a
standard-elementary monomial the root derivation may first produce a
nonstandard product.
\end{remark}

\subsection{Right-root degrees and factorization}

We use the bicrystal convention in which the unstarred multisegment
operators of \cite{ClaxtonTingley} record the right-root action. For the
multisegment \eqref{eq:schubert-multisegment}, let $\epsilon_i^{\mathrm R}$ be the
corresponding string statistic.

\begin{lemma}\label{lem:descent-statistic}
For $w\in S_N$, one has $\epsilon_i^{\mathrm R}(\mathfrak m_w)=\one_{\{i\in\Des(w)\}}$.
\end{lemma}

\begin{proof}
Use the right-endpoint signature of Claxton--Tingley
\cite[Definition~3.2]{ClaxtonTingley}. Only segments ending at $i-1$ and $i$ enter this signature.
Right-regularity leaves at most the two segments
\[ [r_i(w),i-1],\qquad [r_{i+1}(w),i]. \]
The first contributes a right parenthesis and has length $e_i(w)$; the
second contributes a left parenthesis and has length $e_{i+1}(w)$. If
$w_i<w_{i+1}$, then $e_{i+1}(w)\leq e_i(w)$. In the equality case,
the secondary ordering in the right-endpoint signature places the left
parenthesis before the right one. Thus the left parenthesis is absent or
precedes and cancels the right parenthesis. If
$w_i>w_{i+1}$, then $e_{i+1}(w)\geq e_i(w)+1$, so the right parenthesis
precedes the left parenthesis and exactly one left parenthesis remains.
\end{proof}

For $0\ne f\in\CC[U_N]$, define its $i$th \emph{right-root degree} by
\[ \op{rdeg}_i(f) =\deg_s f(Zu_i(s)) =\max\{r\geq0:\mathcal R_i^{(r)}f\ne0\}. \]

\begin{proposition}\label{prop:factor-descent}
Let $w\in S_N$ and $\bullet\in\{\mathrm{gen},\mathrm{can},\mathrm{MV}\}$. If
$\mathsf B_w^\bullet=f_1\cdots f_m$ in $\CC[U_N]$, with every $f_a$ nonconstant, then
$m\leq |\Des(w)|$.
More precisely, the sets
\[ S_a=\{i:\op{rdeg}_i(f_a)>0\},\qquad 1\leq a\leq m, \]
are nonempty and form a set partition of $\Des(w)$.
\end{proposition}

\begin{proof}
For nonzero $f,g\in\CC[U_N]$, the right-root coaction is multiplicative:
\[ (fg)(Zu_i(s))=f(Zu_i(s))g(Zu_i(s)). \]
The two factors on the right specialize at $s=0$ to $f$ and $g$,
respectively, and are therefore nonzero. Since $\CC[U_N][s]$ is a domain,
\begin{equation}\label{eq:right-root-degree-additive}
 \op{rdeg}_i(fg)
 =\op{rdeg}_i(f)+\op{rdeg}_i(g).
\end{equation}
If $\op{rdeg}_i(f)=0$ for every $i$, then $f$ is invariant under
each simple right-root subgroup. These subgroups generate $U_N$, so $f$
is invariant under every right translation of $U_N$ and is therefore
constant. Hence each $S_a$ is nonempty.

The coefficient operators $\mathcal R_i^{(r)}$ are the divided powers of
the infinitesimal right-root operator. By
Proposition~\ref{prop:Schubert-geometric-bases}, the three elements are the
dual semicanonical, dual canonical, and MV basis elements indexed by
the crystal element $\mathfrak m_w$. These bases are biperfect, so
perfectness for the
right action identifies the largest nonzero divided power with the
corresponding crystal string statistic; see
\cite[Sections~1.2, 1.4, and 1.7]{KamnitzerPerfect}. Thus
Lemma~\ref{lem:descent-statistic} gives
\[ \op{rdeg}_i(\mathsf B_w^\bullet) =\epsilon_i^{\mathrm R}(\mathfrak m_w) =\one_{\{i\in\Des(w)\}}. \]
Applying \eqref{eq:right-root-degree-additive} to the given factorization
yields
\[ \one_{\{i\in\Des(w)\}} =\sum_{a=1}^{m}\op{rdeg}_i(f_a) \qquad(1\leq i<N). \]
Therefore every positive summand is equal to one, no index belongs to two
of the sets $S_a$, and their union is $\Des(w)$. The asserted bound
follows.
\end{proof}

\subsection{A rank-ten comparison}\label{sec:separation}

The three geometric bases of Theorem~\ref{thm:geometric-bases-module}
need not coincide, even on right-regular multisegments.

\begin{proposition}\label{prop:S10-deformation}
Let
\[
\begin{split}
 X={}&[1,2]+[1,3]+[3,4]+[3,5]+[2,6]+[2,7]+[4,8]+[4,9],\\
 Y={}&[1,3]+[4,4]+[2,5]+[3,6]+[1,7]+[4,8]+[2,9].
\end{split}
\]
Then $X=\mathfrak m_w$ and $Y=\mathfrak m_v$ for
\[ w=(9,10,6,1,8,7,3,2,5,4),\qquad v=(7,8,10,3,9,4,6,1,5,2). \]
The corresponding transition coefficients are
\begin{equation}\label{eq:S10-basis-coefficients}
 [\Can_v]\Gen_w=1,\qquad
 [\MV_v]\Can_w=1,\qquad
 [\MV_v]\Gen_w=2.
\end{equation}
In particular, the three geometric basis elements indexed by $w$ are
pairwise distinct.
\end{proposition}

Appendix~\ref{app:rank-ten-geometric-comparison} derives the coefficients
from Leclerc's type-$A_5$ example and then applies crystal operators to
obtain the right-regular pair $(X,Y)$.

\section{Strict minors and code monomials}
\label{sec:pattern-coincidence}

\subsection{Strict minors and \texorpdfstring{$321$}{321}-avoidance}
\label{sec:strict-minors}

For $w\in S_N$, let
\[ F(w)=\{i:w(i)>i\},\qquad G(w)=w(F(w)). \]
If $w$ avoids $321$, write $F(w)=\{f_1<\cdots<f_t\}$ and
$G(w)=\{g_1<\cdots<g_t\}$. 
For $1\le r\le N-p$, set
\[ h_r(p)=\det Z(c)_{[p,p+r-1],[p+1,p+r]}=(-1)^r(Z(c)^{-1})_{p,p+r}, \]
with conventions $h_0(p)=1$ and $h_r(p)=0$ for $r<0$.
For equal-cardinality subsets $I=\{i_1<\cdots<i_s\}$ and
$J=\{j_1<\cdots<j_s\}$, call $(I,J)$ \emph{strict} if $i_a<j_a$ for
all $a$. If $v$ avoids $321$, then
\[ I(v)=\Exc(v),\qquad J(v)=v(I(v)) \]
is strict; conversely, every strict pair determines a unique
$321$-avoiding permutation.

\begin{theorem}\label{thm:single-strict-minor}
Let $w\in\Av_N(321)$ and set $v=w^{-1}$. Then
\begin{equation}\label{eq:single-minor}
 \Sch_w(c)
 =\det\bigl(h_{g_i-f_j}(f_j)\bigr)_{1\leq i,j\leq t}
 =\Delta_{G(w)^c,F(w)^c}(Z(c))
 =\Delta_{I(v),J(v)}(Z(c)).
\end{equation}
\end{theorem}

\begin{proof}
The Billey--Jockusch--Stanley formula is the corresponding flagged
Jacobi--Trudi determinant after classical specialization
\cite{BJS,ACT}. Reversing rows and columns gives the first determinant
in \eqref{eq:single-minor}. The identity
$h_r(p)=(-1)^r(Z(c)^{-1})_{p,p+r}$ and Jacobi complementation give
$\Delta_{G(w)^c,F(w)^c}(Z(c))$. Every determinant monomial contains at
most one factor of each level $k$, because level $k$ occurs only in column
$k+1$; after adjoining the missing factors $c_0(k)=1$, it is a
standard-elementary monomial. Hence the minor lies in $\cM_N$. Its
classical specialization is $\Sch_w(x)$, so Fulton's injectivity proves
the universal identity. Deleting the common fixed points in the
complementary row and column sets gives
$\Delta_{I(v),J(v)}(Z(c))$.
\end{proof}

Every nonzero minor of an upper unitriangular matrix can be reduced
uniquely to a strict minor by deleting aligned unit pivots. A strict pair
$(I,J)$ is cut after $a$ whenever $i_{a+1}>j_a$; upper triangularity gives
\begin{equation}\label{eq:strict-factorization}
 \Delta_{I,J}(Z(c))
 =\prod_{\nu}\Delta_{I_\nu,J_\nu}(Z(c)),
\end{equation}
where the factors are the connected pieces of the strict pair. For a seed
index $(a,r)$, let $\mathbf e_{a,r}$ denote the corresponding standard
basis vector, with $\mathbf e_{0,r}=0$.

\begin{theorem}\label{prop:strict-minor-cluster}
For every strict pair $(I,J)$, the minor $\Delta_{I,J}(Z(c))$ is an
extended cluster monomial for the standard cluster structure on
$\CC[U_N]$. Its factorization into nonunit strict minors is the decomposition
\eqref{eq:strict-factorization} into connected pieces. If $(I,J)$
is connected, the minor is either a mutable cluster variable
or one of the frozen flag minors. Its $\g$-vector is
\begin{equation}\label{eq:strict-g-vector}
 \g_{I,J}
 =\sum_{a=1}^{s}
 \bigl(
  \mathbf e_{i_a,j_a-i_a}
  -\mathbf e_{i_a-1,j_a-i_a}
 \bigr).
\end{equation}
\end{theorem}

\begin{proof}
The empty minor is one, so assume the strict pair is nonempty.
We first check the basis normalization. The flag minor
$\Delta_{[1,s],J}$ belongs to every biperfect basis
\cite[Remark~2.10]{BKKMV}. If $p\in R$, $p+1\notin R$, and
$R'=(R\setminus\{p\})\cup\{p+1\}$, then
\[ \Delta_{R,J}(u_p(t)Z) =\Delta_{R,J}(Z)+t\Delta_{R',J}(Z). \]
Starting at $R=[1,s]$, increase the row indices one at a time to reach
$I$, beginning with the largest. Every intermediate row set
satisfies $R\le I\le J$, so every derivative is nonzero and is the top
left-root divided power. By perfectness, $\Delta_{I,J}$ belongs to every biperfect basis,
with no scalar ambiguity.
In particular it is a dual canonical element. Its determinant expansion
has diagonal Lusztig datum $\sum_a[i_a,j_a-1]$; every non-diagonal
term is lower in the Zelevinsky order, so the leading coefficient is one.

Under the Hernandez--Leclerc dictionary of
Lemma~\ref{lem:finite-HL-seed-comparison}, the minor is consequently the
simple class with highest $\ell$-weight
\[ m_{I,J}=\prod_{a=1}^sY_{j_a-i_a,\,2-i_a-j_a}; \]
here standard modules correspond to dual PBW elements and simple modules
to dual canonical elements \cite[Theorem~6.1]{HLQuantumGrothendieck}.
Read the spectral points in reverse order. For successive pairs
$(i,j)$, $(i',j')$, with $i<i'$ and $j<j'$, their spectral
difference is $(i'-i)+(j'-j)$. It is at least
$\lvert(j'-i')-(j-i)\rvert+2$, with the required parity. Thus the data
form a type-$A$ snake. The prime-snake upper bound
$(i'-i)+(j'-j)\le (j-i)+(j'-i')$
is equivalent to $i'\le j$; the other upper bound is automatic since
$j'-i\le N$. Therefore the prime factors are exactly the connected
pieces in \eqref{eq:strict-factorization}
\cite[Section~3.2 and Theorem~3.3]{DLL}.

By \cite[Theorem~1.4]{DuanSchiffler}, the snake class is a cluster
monomial in the ambient Hernandez--Leclerc algebra, and its prime
factors are cluster variables \cite[Theorem~5.9]{DLL}.
Use the ambient seed supplied by
Lemma~\ref{lem:finite-HL-seed-comparison}. Since the minor belongs to
$\CC[U_N]=\mathcal U(\Sigma_N^U)$, its Laurent expansion lies in
\[ \CC[x_{a,r}^{\pm1}:a+r<N]\,[x_{a,r}:a+r=N] \]
by Proposition~\ref{thm:unipotent-seed}. It has no negative exponent on
the finite frozen boundary and does not involve any exterior variable.
Lemma~\ref{lem:cluster-monomial-freezing} therefore gives a cluster
containing all its factors by mutations confined to $a+r<N$.
The boundary separation in
Lemma~\ref{lem:finite-HL-seed-comparison} shows that these are exactly
mutations of $\Sigma_N^U$, without specialization or rescaling by
frozen monomials. A connected factor is consequently a mutable variable
or one of the finite frozen flag minors.

Finally, the diagonal Lusztig datum is
$\sum_a[i_a,j_a-1]$. Its inverse tropical Chamber Ansatz is
\eqref{eq:strict-g-vector}, by the correspondence in
\eqref{eq:schubert-multisegment}.
\end{proof}

\begin{corollary}\label{cor:321-transition}
If $w$ avoids $321$, then
$\Gen_w=\Can_w=\MV_w=\Sch_w(c)$.
\end{corollary}

\begin{proof}
Theorem~\ref{thm:single-strict-minor} and
Theorem~\ref{prop:strict-minor-cluster} identify $\Sch_w(c)$ with an
extended cluster monomial with $\g$-vector $\g_w$. The generic basis contains
every extended cluster monomial, so $\Sch_w(c)=\Gen_w$.
Corollary~\ref{cor:generic-cluster-monomial-canonical} then gives
$\Gen_w=\Can_w$. 
The flag-minor and top-divided-power argument in the proof of
Theorem~\ref{prop:strict-minor-cluster} applies equally to the MV
basis. Thus the strict minor is the MV element with the same
right-regular Lusztig datum, giving $\MV_w=\Sch_w(c)$.
\end{proof}

\paragraph{Full support of the MV transition.}
The strict-minor identities also determine the support of the MV
expansion of a code monomial. Unlike its nonnegative coefficients,
full support in the orbit-closure order is special to the
right-regular family.
\begin{theorem}\label{prop:MV-full-support}
If $z\preceq_{\rm Z}w$, then
\begin{equation}\label{eq:MV-row-monotonicity}
 D_N^{\rm MV}(w,u)\geq D_N^{\rm MV}(z,u)
 \qquad(u\in S_N).
\end{equation}
Consequently,
\begin{equation}\label{eq:MV-full-support}
 D_N^{\rm MV}(w,z)>0
 \quad\Longleftrightarrow\quad
 z\preceq_{\rm Z}w.
\end{equation}
\end{theorem}

\begin{proof}
It is enough to prove \eqref{eq:MV-row-monotonicity} for an elementary
linked move. Thus let $p<q$ and
\[ a=r_p(w)<r_q(w)=b\leq p, \]
and let $z$ be obtained by interchanging $a$ and $b$ in positions $p$ and
$q$. Set
\[
 H=\prod_{\substack{2\leq j\leq N\\j\neq p,q}}Z_{r_j(w),j}.
\]
The $2\times2$ determinantal identity gives
\begin{equation}\label{eq:linked-Plucker-code-difference}
 \Code_w-\Code_z
 =H\bigl(Z_{a,p}Z_{b,q}-Z_{b,p}Z_{a,q}\bigr)
 =H\Delta_{\{a,b\},\{p,q\}}(Z(c)).
\end{equation}
Since $a<b\leq p<q$, the pair $(\{a,b\},\{p,q\})$ is strict. Hence the
minor in \eqref{eq:linked-Plucker-code-difference} is an MV basis element by
Corollary~\ref{cor:321-transition}. Each nonunit factor of $H$ is a flag
minor and hence an MV basis element. Positivity of multiplication in the MV
basis therefore gives a nonnegative MV expansion. The product belongs to
$\cM_N$ by \eqref{eq:linked-Plucker-code-difference}, so only the
Schubert-indexed MV elements occur:
\[ H\Delta_{\{a,b\},\{p,q\}}(Z(c)) =\sum_{u\in S_N}c_u\MV_u, \qquad c_u\in\ZZ_{\geq0}; \]
see \cite{BKKMV}. Comparing MV coefficients in
\eqref{eq:linked-Plucker-code-difference} yields
\[ D_N^{\rm MV}(w,u)-D_N^{\rm MV}(z,u)=c_u\geq0. \]
Iteration along a chain of linked moves proves
\eqref{eq:MV-row-monotonicity}. Taking $u=z$ gives
\[ D_N^{\rm MV}(w,z)\geq D_N^{\rm MV}(z,z)=1. \]
The converse implication in \eqref{eq:MV-full-support} is the support
restriction in \eqref{eq:DMV}.
\end{proof}

\begin{remark}
Anderson and Kogan's fusion order gives the corresponding support restriction
for general type-$A$ fusion products, but the converse fails in general
\cite{AndersonKoganAlgebra}. Theorem~\ref{prop:MV-full-support} shows that the
converse holds for the right-regular products $\Code_w$.
\end{remark}

\begin{lemma}\label{lem:extended-variable-irreducible}
The cluster algebra of the standard triangular seed equals
$\CC[U_N]$. Every extended cluster variable is irreducible in
this polynomial ring.
\end{lemma}
\begin{proof}
Every coordinate $z_{p,q}$ is a connected singleton strict minor,
hence an extended cluster variable by
Theorem~\ref{prop:strict-minor-cluster}. These coordinates generate
$\CC[U_N]$, whereas the cluster algebra is contained in the upper
cluster algebra of Proposition~\ref{thm:unipotent-seed}. Thus the
cluster algebra, upper cluster algebra, and coordinate ring agree.
Mutable cluster variables are irreducible by
\cite[Theorem~1.3\textup{(ii)}]{GLSfactorial}, with no frozen
variables inverted. For a frozen variable $x_{a,N-a}$, the rectangular
permutation in \eqref{eq:initial-single-Schubert-label} has exactly one
descent. Corollary~\ref{cor:321-transition} and
Proposition~\ref{prop:factor-descent} therefore rule out a factorization
into two nonconstant polynomials. This proves the frozen case as well.
\end{proof}

\begin{remark}[Increasing injections]
Let $\nu:[N]\hookrightarrow[M]$ be increasing. Matrix insertion defines
an injective homomorphism
\[ \op{Ins}_\nu:\CC[U_N]\longrightarrow\CC[U_M], \qquad Z^{(N)}_{p,q}\longmapsto Z^{(M)}_{\nu(p),\nu(q)}. \]
For every equal-cardinality pair $(I,J)$,
$\op{Ins}_\nu(\Delta_{I,J}) =\Delta_{\nu(I),\nu(J)}$.
Consequently, if $w\in\Av_N(321)$ corresponds to the strict pair
$(I,J)$ and $w^\nu\in\Av_M(321)$ corresponds to
$(\nu(I),\nu(J))$, then
\[ \op{Ins}_\nu(\Sch_w(c))=\Sch_{w^\nu}(c). \]
\end{remark}

\subsection{Code monomials and \texorpdfstring{$312$}{312}-avoidance}
For $1\leq p\leq q\leq N$, set
\[ J_{p,q}=[1,p-1]\cup\{q\}. \]
The matrix entry $Z(c)_{p,q}$ is the flag minor
\[ Z(c)_{p,q}=\Delta_{[1,p],J_{p,q}}(Z(c)). \]
Associate to this entry the arc $(p,q+1)$ in a convex $(N+1)$-gon. For
$i<j$, the two code factors $Z(c)_{r_i(w),i}$ and
$Z(c)_{r_j(w),j}$ are incompatible exactly when
$r_i(w)<r_j(w)\leq i$,
which is equivalent to a $312$ pattern by
Lemma~\ref{lem:rank-patterns}. Strongly separated flag minors lie in
a common cluster. Conversely, compatible flag minors quasi-commute
and are weakly separated.

\begin{theorem}\label{prop:312-code}
The code monomial $\Code_w$ is an extended cluster monomial for the
standard cluster structure if and only if $w$ avoids $312$.
\end{theorem}

\begin{proof}
If $w$ avoids $312$, its code arcs are noncrossing; their flag minors form a
strongly separated family and occur together in a cluster associated
with a reduced wiring diagram. Hence their product is an extended cluster
monomial.

Conversely, $\CC[U_N]$ is a unique factorization domain, and the nonunit
entries $Z(c)_{p,q}$ are pairwise nonassociate irreducibles. If $\Code_w$
is an extended cluster monomial, all its nonunit entry factors occur in one cluster; the frozen variables
are available in every cluster. Their quantum lifts quasi-commute, so their
index sets are
weakly separated. For the hook sets $J_{p,q}$ this excludes the crossing
condition above, and
therefore excludes $312$.
\end{proof}

\begin{corollary}\label{cor:code-transition-extremals}
For $w\in S_N$, the following conditions are equivalent:
\begin{enumerate}[label=\textup{(\roman*)}]
\item $w$ avoids $312$;
\item $M_w$ is $Q$-rigid;
\item $\Code_w$ is an extended cluster monomial;
\item $\Code_w=\Can_w$;
\item $\Code_w=\MV_w$;
\item $\Code_w=\Gen_w=\Can_w=\MV_w$.
\end{enumerate}
In contrast,
\[ \Sch_w(c)=\Code_w\quad\Longleftrightarrow\quad w\in\Av_N(312,1432). \]
\end{corollary}
\begin{proof}
Proposition~\ref{prop:self-extension-codimension} and Theorem~\ref{prop:312-code} give the equivalence of the first three conditions.
For $312$-avoiding $w$, the orbit of $M_w$ is open in the
equioriented representation space of dimension vector $d(w)$;
equivalently, its quotient stratum is open in $X_{d(w)}$. Thus the
PBW comparisons for $\Code_w$ have no off-diagonal terms. This gives the sixth
condition and hence the fourth and fifth. Conversely, the canonical
comparison has a positive coefficient at every generization, and the
same is true of the MV comparison by Theorem~\ref{prop:MV-full-support}.
Either the fourth or the fifth condition therefore forces the orbit to
be open. The last equivalence is Woodruff's criterion for a Schubert
polynomial to be a standard-elementary monomial~\cite{Woodruff}, lifted
by Fulton's injectivity.
\end{proof}

\begin{remark}
Avoidance of $312$ characterizes coincidence with the code basis;
it does not, by itself, characterize coincidence with the universal
Schubert basis. For example,
$w=(1,4,3,2)$ avoids $312$ but not $1432$, so
$\Code_w=\Gen_w=\Can_w=\MV_w\ne\Sch_w(c)$.
\end{remark}
\subsection{The strict-minor coalgebra}
For $f\in\CC[U_N]$, write $\operatorname{cop}(f)(A,B)=f(AB)$.
Set
\[ \mathscr C_N^{321}=\operatorname{span}_{\CC} \{\Sch_w(c):w\in\Av_N(321)\}. \]

We use the two-sided weak order: $v\leq_{\mathrm{LR}}w$ when
$w=avb$ and $\ell(w)=\ell(a)+\ell(v)+\ell(b)$. A Schubert-basis
lower ideal is the span of the $\Sch_w(c)$ indexed by a lower set in
this order. The left root derivative is
$\mathcal L_i f=[s]f((I+sE_{i,i+1})Z)$.

\begin{theorem}\label{thm:root-coalgebra}
The space $\mathscr C_N^{321}$ is a graded subcoalgebra of
$\CC[U_N]$. For a strict pair $(I,J)$,
\[ \operatorname{cop}\Delta_{I,J} =\sum_{I\leq K\leq J}\Delta_{I,K}\otimes\Delta_{K,J}, \]
where aligned unit pivots are removed from either factor. Equivalently,
\begin{equation}\label{eq:coalgebra-321}
 \operatorname{cop}\Sch_w(c)
 =\sum_{\substack{uv=w\\\ell(u)+\ell(v)=\ell(w)}}
       \Sch_u(c)\otimes\Sch_v(c)
 \qquad(w\in\Av_N(321)).
\end{equation}
This coalgebra is the graded dual of the finite nil--Temperley--Lieb
algebra. On its Schubert basis, the left and right root derivatives
act by the corresponding left and right descent operators. Among
Schubert-basis lower ideals in the two-sided weak order, it is the
largest on which these root-derivative formulas hold.
\end{theorem}

\begin{proof}
Cauchy--Binet gives the first formula. Upper triangularity removes all
summands except $I\leq K\leq J$. After deleting aligned unit pivots,
each remaining minor is strict, so
Theorem~\ref{thm:single-strict-minor} proves closure under the
coproduct.

We explain the permutation indexing. On
$\bigoplus_{k=0}^{N}\bigwedge^k\CC^N$, let $a_i$ replace $i+1$ by $i$
in a wedge basis vector when the replacement is possible, and act by
zero otherwise. The operators satisfy
\[ a_i^2=0,\qquad a_ia_j=a_ja_i\quad(|i-j|>1),\qquad a_ia_{i+1}a_i=a_{i+1}a_ia_{i+1}=0. \]
They give the usual faithful exterior-power representation of the
finite nil--Temperley--Lieb algebra. Its basis consists of the
$a_w$ with $w$ fully commutative, or equivalently $321$-avoiding;
see~\cite{FominStanley,Stembridge}.

For a strict pair $I<J$, move the particles of $J$ to $I$ without
interchanging their order. The resulting reduced product is the
$321$-avoiding element attached to $(I,J)$ in
Theorem~\ref{thm:single-strict-minor}. The matrix coefficient of this
product is one. Every permitted intermediate particle set is exactly
a set $K$ with $I\leq K\leq J$. Thus multiplication of the basis
$a_u a_v$ corresponds to the intermediate sets in Cauchy--Binet:
it is $a_{uv}$ when lengths add and $uv$ is fully commutative, and is
zero otherwise. Dualizing gives \eqref{eq:coalgebra-321} and the
asserted graded duality.

A single particle move changes one index in a row or column set of a minor. In the
same indexing it removes a left or right simple reflection from a
reduced word when the relevant descent is present, and otherwise gives
zero. This is the asserted compatibility with the root derivatives.

For maximality, the fully commutative elements form a lower set in the
two-sided weak order. If $w$ is not fully commutative, some reduced
word contains $s_i s_{i+1}s_i$ or $s_{i+1}s_i s_{i+1}$ as consecutive
factors. Thus every lower set containing $w$ contains
$u_i=s_i s_{i+1}s_i$. Fulton's recursion gives
\[ \Sch_{u_i}(c)=c_1(i)c_2(i+1)-c_3(i+1), \qquad \mathcal R_{i+1}\Sch_{u_i}(c)=c_1(i)^2-c_2(i). \]
This polynomial lies outside $\cM_N$: it contains the monomial
$c_1(i)^2$ with coefficient one, whereas each standard-elementary
monomial uses at most one factor from level $i$. It therefore cannot be the
Schubert descent of $\Sch_{u_i}(c)$. This excludes every larger lower
ideal and proves maximality.
\end{proof}

\section{Generic \texorpdfstring{$F$}{F}-polynomials from rank words}
\label{sec:full-generic-F}\label{sec:schubert-F-polynomials}
We describe the generic $F$-polynomial associated with each Schubert
vertex, keeping the coefficient variables independent. The
cluster-character formula is standard~\cite{FKernelGLS}; here we
identify the relevant Ext-module and describe its quiver Grassmannians
in terms of the rank word. Throughout this section,
$B$ has mutable rows and all seed vertices as columns, with the
outgoing-minus-incoming convention of Section~\ref{sec:triangular-seed}.
Thus
\begin{equation}\label{eq:separation}
 \Gen_w=\mathbf x^{\g_w}
 F_w\bigl((\mathbf x^{B^{\mathsf T}\mathbf e_v})_{v\in\mu_N}\bigr).
\end{equation}
We write $\chi$ for the compactly supported topological Euler
characteristic. On the projective varieties below it is the ordinary
Euler characteristic.

\subsection{The representation attached to a rank word}
A preprojective representation is a graded vector space
$M=\bigoplus_{k=1}^{N-1}M_k$ with operators
\[ y:M_k\longrightarrow M_{k-1},\qquad x:M_k\longrightarrow M_{k+1},\qquad xy=yx, \]
where $M_0=M_N=0$. For the rank word $r=r(w)$, choose basis vectors
$v_{j,k}$ with $r_j\leq k<j$, and set
\[
 yv_{j,k}=\begin{cases}v_{j,k-1},&k>r_j,\\0,&k=r_j.\end{cases}
\]
For each pair $i<j$ with $r_i<r_j\leq i$, introduce a parameter
$\lambda_{ij}$ and define
\begin{equation}\label{F:eq:genericx}
 xv_{i,k}=
 \sum_{\substack{j>i,\ r_i<r_j\leq i\\r_j-1\leq k\leq i-1}}
       \lambda_{ij}v_{j,k+1}.
\end{equation}
These maps commute with $y$. Conversely, a degree-one map between two
forward intervals is determined by the image of its top vector; the
possible nonzero maps have exactly the supports in
\eqref{F:eq:genericx}. Thus the parameters are affine coordinates on the conormal fiber
over the forward module $M_w$. We choose them in a dense open subset
when a general point of $Z_{\mathfrak m_w}$ is required.

For $(a,r)\in\mathcal I_N$, let
\[ T_{a,r}=\CC[x,y]/(x^a,y^r),\qquad \deg(1)=r,\quad\deg(x)=1,\quad\deg(y)=-1. \]
The forward restriction of $T_{a,r}$ has multisegment
$[1,r]+[2,r+1]+\cdots+[a,a+r-1]$. The rectangular modules
$T_{a,r}$ are the summands of the initial cluster-tilting object
$T=\bigoplus_{(a,r)\in\mathcal I_N}T_{a,r}$; their characters are the
initial minors $x_{a,r}$. At the boundary $a+r=N$, they are
projective-injective. We use the $\End_{\Lambda_Q}(T)^{\mathrm{op}}$-module
$E=\Ext^1_{\Lambda_Q}(T,M)$ that enters the cluster character.
We shall refer to $E$ as the \emph{Ext-module}.

\begin{lemma}\label{F:lem:kernel}
There are natural identifications
\begin{equation}\label{F:eq:kernel}
 E_{a,r}=\Ext^1_{\Lambda_Q}(T_{a,r},M)
 =\ker(y^r:M_{a+r}\longrightarrow M_a)/x^aM_r.
\end{equation}
The contravariant action of $\End_{\Lambda_Q}(T)$ is generated by
\begin{equation}\label{F:eq:arrows}
 (a,r)\xrightarrow{x}(a+1,r),\qquad
 (a,r)\xrightarrow{y}(a,r-1),\qquad
 (a,r)\xrightarrow{1}(a-1,r+1),
\end{equation}
whenever both vertices are mutable. At frozen indices, $T_{a,r}$ is projective-injective and the
corresponding Ext space is zero.
\end{lemma}

\begin{proof}
The projective cover of $T_{a,r}$ is
$P_r=\CC[x,y]/(x^{N-r},y^r)$, with generator in degree $r$.
Its kernel is $x^aP_r$. A map from this kernel to $M$ is determined by
an element of $M_{a+r}$ annihilated by $y^r$. The remaining relation
has image in $M_N=0$. Restriction from $\Hom(P_r,M)=M_r$ has image
$x^aM_r$, which proves \eqref{F:eq:kernel}.

The three maps come from the quotient, multiplication-by-$y$, and
multiplication-by-$x$ maps between adjacent rectangles. Their lifts to
projective covers induce the stated operators on the quotients.
They generate the action of the whole endomorphism algebra, as shown
in Appendix~\ref{app:rectangle-maps}. The formula gives zero when
$a+r=N$.
\end{proof}

\subsection{Quotient Grassmannians and coefficients}
Let $\Quot_q(E)$ denote the \emph{quiver Grassmannian} of
quotient representations of $E$ with dimension vector $q$. It is
canonically isomorphic to the quiver Grassmannian
$\Gr_{\dim E-q}(E)$ of subrepresentations. Its Euler
characteristic is taken in the complex topology. Quotient dimensions
agree with the exchange-matrix sign convention fixed above.

\begin{lemma}\label{F:lem:normalization}
For a general point $M\in Z_{\mathfrak m_w}$, set $b=\dim E$. Then
\begin{equation}\label{F:eq:quotF}
 F_w(\mathbf y)
 =\sum_q\chi(\Quot_q(E))\mathbf y^q
 =\mathbf y^b\sum_e\chi(\Gr_e(E))\mathbf y^{-e}.
\end{equation}
Moreover, $B^{\mathsf T}$ has an integer left inverse.
\end{lemma}

\begin{proof}
The GLS character formula expresses the character in terms of the
submodule polynomial of $E$, with coefficient exponents
$-B^{\mathsf T}$~\cite{FKernelGLS}. Passing from a submodule to its
quotient gives the reciprocal polynomial on the right of
\eqref{F:eq:quotF}. Its constant term is one. The Hom-space calculation
in Appendix~\ref{app:F-normalization} gives an integer left inverse
of $B^{\mathsf T}$ and verifies the sign. In particular, pointed
expressions in the cone $B^{\mathsf T}\NN^{\mu_N}$ have a unique
initial exponent. Since the generic character is pointed at $\g_w$,
its reciprocal expression is precisely \eqref{eq:separation}.
\end{proof}

\begin{example}[The choice of preprojective representative]
\label{ex:forward-versus-generic}
For $w=312$, the rank word is $(1,1,2)$ and the forward module is
$M_w=S_1\oplus S_2$. Its conormal fiber consists of the maps
$x:(M_w)_1\to(M_w)_2$. A nonzero map gives
$X\simeq T_{2,1}$, whereas the zero map gives $(M_w,0)$.
Formula~\eqref{F:eq:kernel} distinguishes them:
\[ \Ext^1_{\Lambda_Q}(T_{1,1},X)=0,\qquad \Ext^1_{\Lambda_Q}(T_{1,1},(M_w,0))\simeq\CC. \]
Thus $F_{312}=1$ and $\Gen_{312}=x_{2,1}$. Computing at
$(M_w,0)$ instead would give the polynomial $1+y_{1,1}$.
The generic choice in Lemma~\ref{F:lem:normalization} is essential.
\end{example}

We give the coefficient varieties directly in the original graded
spaces. Set
\[ K_{a,r}=\ker(y^r:M_{a+r}\longrightarrow M_a),\qquad I_{a,r}=x^aM_r. \]
For $0\leq q_v\leq b_v$, let $\mathcal V_q(r)$ consist of families
of subspaces $L_{a,r}\subseteq M_{a+r}$ satisfying
\begin{equation}\label{F:eq:flags}
\begin{gathered}
 I_{a,r}\subseteq L_{a,r}\subseteq K_{a,r},\qquad
 \dim L_{a,r}=\dim K_{a,r}-q_{a,r},\\
 xL_{a,r}\subseteq L_{a+1,r},\qquad
 yL_{a,r}\subseteq L_{a,r-1},\qquad
 L_{a,r}\subseteq L_{a-1,r+1}.
\end{gathered}
\end{equation}
Only inclusions associated with valid mutable arrows are imposed.
These conditions define a closed subscheme of a product of
Grassmannians, identified with $\Quot_q(E)$. Because only
finitely many dimension vectors occur, the parameters in
\eqref{F:eq:genericx} can be chosen in one dense open subset on
which all relevant Euler characteristics are constant.

Let
\[
 \mathsf R:\ZZ^{\mu_N}\longrightarrow\ZZ^{\max(N-2,0)},\qquad
 (\mathsf Rq)_r=\sum_aq_{a,r},\qquad
 \widehat F_w(\mathbf t)=F_w(y_{a,r}=t_r).
\]
We call this specialization the row collapse.

\begin{theorem}\label{F:thm:general}
For every $w\in S_N$, the rank word $r(w)$ determines the equations
\eqref{F:eq:flags}. For conormal parameters in a common dense open subset, the
resulting projective varieties satisfy
\[ [\mathbf y^q]F_w=\chi(\mathcal V_q(r(w))),\qquad [\mathbf t^c]\widehat F_w =\sum_{\mathsf Rq=c}\chi(\mathcal V_q(r(w))). \]
Thus the equations \eqref{F:eq:flags}, evaluated at generic conormal
parameters, determine every coefficient of $F_w$.
\end{theorem}

\begin{proof}
Lemma~\ref{F:lem:kernel} identifies the equations with quotient
Grassmannians, and Lemma~\ref{F:lem:normalization} gives the first
coefficient formula. Specialization gives the second. On the open
parameter locus where all matrix ranks are constant, the relative
Grassmannians are projective over the parameter space.
Constructibility of their Euler characteristics gives a common dense
open subset on which these values are constant.

Kernels, images, and the induced maps are obtained by linear algebra
over $\QQ(\lambda_{ij})$. Appendix~\ref{app:F-effective} describes
Grassmannian charts for these varieties, yielding a finite system
of polynomial equations. Under the gap condition, their reduced schemes are
empty or points, and the coefficient formula becomes a lattice-point
count.
\end{proof}

\subsection{Newton polytopes before and after collapse}
The Euler characteristic of a nonempty quiver Grassmannian may vanish.
Nevertheless, the nonempty quotient Grassmannians determine the Newton
polytope, even after row collapse.

\begin{lemma}\label{F:lem:vertices}
Let $E$ be a finite-dimensional representation, and let $\mathcal S(E)$
be the set of dimension vectors of its nonempty quotient
Grassmannians. Then
\[ \operatorname{Newt}\left(\sum_q\chi(\Quot_q(E))\mathbf y^q\right) =\operatorname{conv}\mathcal S(E). \]
Suppose that $\mathsf R$ is an integer linear map satisfying
$\ker \mathsf R\cap\mathbb R_{\geq0}^{\mu_N}=\{0\}$. After specialization by
$\mathsf R$, the Newton polytope is $\mathsf R\operatorname{conv}\mathcal S(E)$.
Every coefficient at a vertex of either polytope is one.
\end{lemma}

\begin{proof}
We pass to submodules of $E^*$, which have the same dimension vectors
as the quotients of $E$. Let $c$ be a vertex of the projected convex
hull, and choose a linear functional exposing it. If $U$ and $V$ are
submodules with $\mathsf R\dim U=\mathsf R\dim V=c$, then
\[ \dim(U\cap V)+\dim(U+V)=\dim U+\dim V. \]
The exposing functional attains its maximum on $U$ and $V$. It follows
that it attains the same maximum on their intersection and sum. Since
$c$ is the unique maximizing point after projection,
$\mathsf R\dim(U\cap V)=\mathsf R\dim(U+V)=c$. The vector
$\dim(U+V)-\dim U$ is nonnegative and belongs to $\ker \mathsf R$, so it is
zero. Hence $U=V$.

Thus there is exactly one submodule over a projected vertex. Its
reduced parameter space is a point and contributes Euler
characteristic one; in particular, no cancellation occurs there.
All polynomial exponents lie in the projected feasible set, proving
the projected assertion. Taking $\mathsf R$ to be the identity proves the
first assertion. Compare the submodule-polytope description
in~\cite{FKernelFei}.
\end{proof}

The row-collapse map satisfies the hypothesis of the lemma. Hence, for every $w\in S_N$,
\[
 \operatorname{Newt}(\widehat F_w)
 =\mathsf R\operatorname{Newt}(F_w)
 =\mathsf R\operatorname{conv}\{q:\mathcal V_q(r(w))\ne\varnothing\}.
\]
At a nonvertex of the collapsed polytope, contributions from
several quotient dimensions may combine; such coefficients need not
be one.

\subsection{The gap condition and cover pairs}
\begin{definition}\label{def:gap-word}
A rank word $r=(r_1,\ldots,r_N)$ satisfies the \emph{gap condition} if
\begin{equation}\label{F:eq:gap}
 p<j,\quad 1<r_j\leq r_p<p
 \quad\Longrightarrow\quad
 \exists t\ (p<t<j),\quad r_j-1\leq r_t\leq r_p.
\end{equation}
\end{definition}

We describe the conormal representative used under this condition.
Let $\operatorname{Slot}(r)=\{i:r_i<i\}$ and $V_i=[r_i,i-1]$.
For $i,j\in\operatorname{Slot}(r)$, write
\begin{align}
 i\vdash j&\quad\Longleftrightarrow\quad
 i=j\ \text{or}\ i<j\ \text{and}\ r_i\leq r_j\leq i-1,
 \label{eq:record-Hom-relation}\\
 i\prec j&\quad\Longleftrightarrow\quad
 i<j\ \text{and}\ r_i<r_j\leq i.
 \label{eq:record-Ext-relation}
\end{align}
These are the conditions for $\Hom_Q(V_i,V_j)\ne0$ and
$\Ext_Q^1(V_j,V_i)\ne0$, respectively. Set
$\mathcal V(r)=\{(i,j):i\vdash j\}$ and
$\mathcal U(r)=\{(i,j):i\prec j\}$. An extension pair $(i,j)$ is a
\emph{cover pair} if there is no $z\notin\{i,j\}$ with
$i\vdash z\prec j$ or $i\prec z\vdash j$. The \emph{cover graph}
has vertex set $\operatorname{Slot}(r)$ and an edge $i\to j$ for
each cover pair.

Under the gap condition, the cover graph specifies a representative
of the dense conormal orbit, so generic parameters are no longer
needed to compute the character.

\paragraph{The cover-pair orbit calculation.}
For reverse-arrow parameters $\lambda_{ij}$, the
infinitesimal $\operatorname{Aut}_Q(M_w)$-action on
$\operatorname{Ext}^1_Q(M_w,M_w)^*$ is represented in the interval
bases by a matrix whose rows correspond to extension pairs
$(i,j)\in\mathcal U(r)$ and columns to Hom pairs
$(i,j)\in\mathcal V(r)$. With these conventions, its
potentially nonzero entries are
\begin{equation}\label{eq:record-commutator-matrix}
 M(\lambda)_{(i,j),(i,j')}=\lambda_{(j',j)},\qquad
 M(\lambda)_{(i,j),(i',j)}=-\lambda_{(i,i')},
\end{equation}
and all other entries zero.
Appendix~\ref{app:record-rigidity} verifies the orientation and
duality conventions for this matrix. The next theorem computes a
maximal minor of the orbit map.

\begin{proposition}[Cover-pair minor under the gap condition]
\label{thm:gap-cover-minor}
Let $\rho$ be a rank word satisfying
\eqref{F:eq:gap}. Write $F\subseteq\mathcal U=\mathcal U(\rho)$
for its cover pairs and define
\begin{equation}\label{eq:record-compatible-pivot}
 c(i,j)=\max\{i\le k<j:\rho_i\le\rho_k<\rho_j\}
 \qquad((i,j)\in\mathcal U).
\end{equation}
The graph $F$ is a forest with at most one outgoing edge at each vertex.
For a noncover $(i,j)$, one has $c(i,j)>i$, $i\vdash c(i,j)$,
and $(c(i,j),j)\in F$.

Set $\lambda_e=0$ for $e\notin F$ and $\lambda_e\ne0$ for
$e\in F$. In the commutator matrix choose column $(i,i)$ for each
cover $(i,j)$, and column $(i,c(i,j))$ for each noncover.
Together with all rows, these columns form a square submatrix whose
determinant, in the ordering specified below, is
\begin{equation}\label{eq:record-compatible-minor}
 \prod_{(i,j)\in F}\lambda_{(i,j)}
 \prod_{(i,j)\in\mathcal U\setminus F}\lambda_{(c(i,j),j)}.
\end{equation}
Here covers are ordered first, by increasing first index, and
noncovers next, by increasing second index; columns follow their rows.
Thus the infinitesimal orbit map is surjective, and the represented
module lies in the dense orbit of the conormal component. Setting
$\lambda_e=1$ on $F$ makes the determinant equal to one; this includes
$\mathcal U=\varnothing$, with the empty-minor convention.
\end{proposition}

\begin{proof}
We use the following consequence of the gap condition:
\begin{equation}\label{eq:gap-exact-predecessor}
 q<j,\quad 1<\rho_j\le\rho_q<q
 \quad\Longrightarrow\quad
 \exists t\ (q<t<j),\quad \rho_t=\rho_j-1.
\end{equation}
Indeed, take the rightmost $t\in[q,j)$ with
$\rho_j-1\le\rho_t\le\rho_q$. If $\rho_t\ge\rho_j$, the gap
condition for $(t,j)$ supplies a later such position, a contradiction.

Let $F$ be the set of covers. Each vertex has at most one outgoing
edge in $F$. Otherwise take covers $(i,j),(i,k)$, with $j<k$.
If $\rho_j\le\rho_k$, then $i\prec j\vdash k$, contradicting
that $(i,k)$ is a cover. If $\rho_j>\rho_k$,
\eqref{eq:gap-exact-predecessor} supplies $j<t<k$ with
$\rho_t=\rho_k-1$. Since $\rho_i\le\rho_t\le i-1$, one has
$i\vdash t\prec k$, the same contradiction. The edges increase their
indices; an undirected cycle would give two outgoing edges at its least
vertex. Thus $F$ is a forest, each component having a unique sink.

For $(i,j)\in\mathcal U$, the set defining $c(i,j)$ is nonempty,
since it contains $i$. Then $(c(i,j),j)$ is a cover. To see this, set $c=c(i,j)$.
An intermediary $c\vdash z\prec j$ contradicts maximality of $c$.
An intermediary $c\prec z\vdash j$ does so as well unless
$\rho_z=\rho_j$; in that case \eqref{eq:gap-exact-predecessor}
applied to $(z,j)$ supplies a still later position of value
$\rho_j-1\ge\rho_c$, again a contradiction. Hence $(c,j)\in F$.
If $(i,j)$ is a noncover, then $c>i$ and
$\rho_i\le\rho_c<\rho_j\le i$, so $i\vdash c$.

Set all noncover parameters to zero. Choose column $(i,i)$ for each
cover $(i,j)$, and column $(i,c(i,j))$ for each noncover $(i,j)$.
These columns are distinct: equality of two noncover columns would give
two outgoing covers from their common second index.
Rows indexed by covers vanish in columns $(a,b)$ with $a\ne b$,
and rows indexed by noncovers vanish in columns $(a,a)$.
Order the cover pairs and their selected columns by increasing first index;
the resulting matrix is upper triangular with diagonal $\lambda_{(i,j)}$.
For the column $(i,c)$ selected for a noncover, where $(c,j)\in F$,
the commutator formula gives the entry $\lambda_{(c,j)}$ in row
$(i,j)$.
Every other nonzero entry is $-\lambda_{(a,i)}$ in a row $(a,c)$,
whose second index $c$ is smaller than $j$. Thus ordering
noncovers by increasing second index gives another upper triangular
matrix. The selected maximal minor therefore has determinant
\eqref{eq:record-compatible-minor}.
This holds for every choice of nonzero cover parameters.

Since the minor is nonzero, the infinitesimal orbit map is surjective.
Over $\CC$, the orbit is open in $\Ext_Q^1(M,M)^*$, and hence the
corresponding conormal component is $\Lambda_Q$-rigid
\cite[Sections~3.1--3.3 and Theorem~4.4]{LapidGreedy}.

Finally, rescale the summands of $M$ to make each nonzero cover
coefficient equal to one. Since the cover graph is a forest, the required scalars can be
chosen successively along its edges.
This gives the representative used in
Section~\ref{sec:full-generic-F}.
\end{proof}

\begin{theorem}[The generic cover representation]
\label{thm:gap-cover-representation}
Suppose that $r=r(w)$ satisfies the gap condition. Its cover graph is
a forest with at most one outgoing edge at each vertex; denote it by
$\mathcal F(r)$, with vertex set $\operatorname{Slot}(r)$. In
\eqref{F:eq:genericx}, set $\lambda_{ij}=1$ on cover pairs and
$\lambda_{ij}=0$ on all other extension pairs, and denote the resulting
preprojective representation by $X_r$. Then $X_r$ lies in the open
$\operatorname{Aut}_Q(M_w)$-orbit of the conormal fiber over $M_w$.
Its $G_{d(w)}$-orbit is dense in $Z_{\mathfrak m_w}$, and
\[ \Ext^1_{\Lambda_Q}(X_r,X_r)=0. \]
Any choice of nonzero cover coefficients and zero noncover coefficients
gives an isomorphic representation.
\end{theorem}

\begin{proof}
Proposition~\ref{thm:gap-cover-minor} proves the forest assertion and
exhibits a maximal minor of the infinitesimal orbit map equal to one
at these parameters. Set
\[ h=\dim\End_Q(M_w),\qquad e=\dim\Ext_Q^1(M_w,M_w),\qquad d=d(w). \]
The orbit map onto the conormal fiber is therefore submersive at
$X_r$. Its stabilizer is $\operatorname{Aut}_{\Lambda_Q}(X_r)$, so
$\dim\End_{\Lambda_Q}(X_r)=h-e$. The quiver Euler form gives
$(d,d)=2(h-e)$, where $(\ ,\ )$ is the symmetric Cartan pairing.
The preprojective relation complex gives
\[ \dim\Ext^1_{\Lambda_Q}(X_r,X_r) =2\dim\End_{\Lambda_Q}(X_r)-(d,d)=0; \]
see also \cite[Sections~3.1--3.3]{LapidGreedy}.

The conormal bundle over the forward orbit has dimension
$\dim G_d-h+e$, which is also the dimension of the $G_d$-orbit
of $X_r$. This orbit is dense in its conormal component.
Finally, rescaling the interval summands along the edges of each tree
normalizes every nonzero cover coefficient to one.
\end{proof}

\begin{proposition}\label{prop:cover-components-indecomposable}
Let $r$ be a gap word. For $A\in\pi_0(\mathcal F(r))$, let $X_r[A]$
be the subrepresentation spanned by the interval vectors $v_{j,k}$ with
$j\in A$. Then
\[ X_r=\bigoplus_{A\in\pi_0(\mathcal F(r))}X_r[A] \]
is its decomposition into pairwise nonisomorphic indecomposable
summands. If there is at least one nonempty slot, then
\begin{equation}\label{eq:cover-endomorphism-quotient}
 \End_{\Lambda_Q}(X_r)/\operatorname{rad}\End_{\Lambda_Q}(X_r)
 \simeq\prod_{A\in\pi_0(\mathcal F(r))}\CC.
\end{equation}
For the empty forest, $X_r=0$.
\end{proposition}
\begin{proof}
The forward interval summands $V_j=[r_j,j-1]$ are pairwise
nonisomorphic, have scalar endomorphisms, and satisfy
$\Hom_Q(V_i,V_j)=0$ for $i>j$. Consequently, diagonal entries define
an algebra homomorphism from $\End_Q(\bigoplus_jV_j)$ to a product
of copies of $\CC$; its kernel consists of strictly triangular
endomorphisms and is nilpotent.

Write $f_{ii}$ for the scalar diagonal entry of a preprojective
endomorphism $f$. At a cover pair $i\to j$, the commutation equation
for the reverse maps is
$(f_{ii}-f_{jj})\lambda_{ij}=0$.
Indeed, an off-diagonal contribution would factor $(i,j)$ through
$i\vdash z\prec j$ or $i\prec z\vdash j$, which is excluded by
the definition of a cover. Equivalently, the cover row of
\eqref{eq:record-commutator-matrix} vanishes on all off-diagonal
columns. Since $\lambda_{ij}=1$, the diagonal is constant on each
connected component.

All nonzero reverse maps remain within a component, so the projections
onto $X_r[A]$ are preprojective endomorphisms. The diagonal map
therefore surjects onto one copy of $\CC$ per component. Its kernel
is a nilpotent ideal, and its quotient is semisimple; hence the kernel
is exactly the Jacobson radical, proving
\eqref{eq:cover-endomorphism-quotient}. Applied to one component, this
says that its endomorphism algebra is local, so $X_r[A]$ is
indecomposable. Distinct components cannot give isomorphic modules:
their forward restrictions have disjoint, nonempty sets of interval
summands. This proves the asserted decomposition.
\end{proof}

The preceding proposition also recovers the generic decomposition
of \cite[Theorem~3.4\textup{(6)}]{LapidGreedy} for cover representations.
Theorem~\ref{solseed:thm:main} establishes the further fact that the
summands occur in a reachable cluster-tilting object.

We now compute the character of $X_r$. If $j$ has an outgoing cover
edge, denote its parent by $\operatorname{par}(j)$. Define
\[ \delta_0(j)=j,\qquad \delta_a(j)=\max_{i\to j}\delta_{a-1}(i),\qquad \max\varnothing=0. \]
Thus $\delta_a(j)$ is the greatest index of a descendant at distance
$a$ from $j$. For $(a,r)\in\mu_N$, set
\begin{equation}\label{F:eq:J}
 J_{a,r}=\{j:a<r_j\leq a+r<j,\ r\geq\delta_a(j)\},
 \qquad b_{a,r}=|J_{a,r}|.
\end{equation}
For an arrow $u\to v$ in \eqref{F:eq:arrows}, set
\[
 \rho_{uv}=
 \begin{cases}
 |\{\operatorname{par}(j):j\in J_u,\
                   \operatorname{par}(j)\in J_v\}|,&\text{for an $x$-arrow},\\
 |J_u\cap J_v|,&\text{for a $y$- or identity arrow}.
 \end{cases}
\]
A missing parent makes no contribution. Finally, define
\begin{equation}\label{F:eq:P}
 \mathcal P(r)=\left\{q\in\mathbb R^{\mu_N}:
 \begin{array}{l}
 0\leq q_v\leq b_v\quad(v\in\mu_N),\\
 q_v-q_u\leq b_v-\rho_{uv}\quad(u\to v)
 \end{array}\right\}.
\end{equation}

The next lemma explains why the nonzero coefficients are one. The
necessary basis and barrier lemmas are proved in
Appendix~\ref{app:F-cyclic}.

\begin{lemma}\label{F:lem:chain}
For fixed $(a,r)$, the elements of $J_{a,r}$, in increasing order,
form a consecutive cover path. On a nonzero $E_{a,r}$, the central
operator $z=xy$ is one regular nilpotent Jordan block.
\end{lemma}
\begin{proof}
Let $k=a+r$, and take consecutive elements $p<q$ of $J_{a,r}$.
If $r_p\geq r_q$, Lemma~\ref{F:lem:barrier} gives a path of length
$r_q-1\geq a$ ending at $q$, entirely after $p>k>r$.
This would give $\delta_a(q)>r$, contrary to \eqref{F:eq:J}.
Therefore $r_p<r_q$. The same argument excludes any intervening
nonempty index of rank at least $r_q$.

Choose the last index $i\in[p,q)$ of greatest nonempty rank on this
interval. Its rank is greater than $a$ and less than $r_q$, and all
subsequent nonempty ranks are smaller. Since $r_q\leq k<i$, the pair
$i\to q$ is a cover. If $\delta_a(i)>r$, a path of length $a$
ending at $i$, followed by this edge, would contain a path of length
$a$ ending at $q$ that still begins after $r$. This is impossible.
Hence $i\in J_{a,r}$, and consecutiveness forces $i=p$.

On these basis vectors, $xy$ sends an index to its parent with
coefficient one whenever the parent survives, and sends the final
index to zero. The consecutive path is therefore one Jordan chain.
\end{proof}

Recall that a lattice polytope $P$ has the \emph{integer
decomposition property} if, for every integer $k\ge1$, each lattice
point of $kP$ is a sum of $k$ lattice points of $P$.

\begin{theorem}\label{F:thm:gap}
If $r(w)$ satisfies the gap condition, then
\begin{equation}\label{F:eq:Fgap}
 F_w(\mathbf y)=\sum_{q\in\mathcal P(r(w))\cap\ZZ^{\mu_N}}\mathbf y^q,
 \qquad \operatorname{Newt}(F_w)=\mathcal P(r(w)).
\end{equation}
Every nonzero coefficient is one, and $\mathcal P(r(w))$ is an
integral polytope with the integer decomposition property. For each
quotient dimension, the underlying reduced quotient Grassmannian is
either empty or a single point.
\end{theorem}

\begin{proof}
Lemma~\ref{F:lem:survivors} gives bases of $E_{a,r}$ indexed by
$J_{a,r}$, in which the three maps are induced by the parent map
or the identity. By Lemma~\ref{F:lem:chain}, the central operator
$z=xy$ acts on every nonzero $E_{a,r}$ as one nilpotent Jordan block.
Its invariant subspace of dimension $e_v$ is therefore unique. Every
arrow is $\CC[z]$-linear, and an arrow of rank $\rho_{uv}$ sends this
subspace into the unique subspace of dimension
$\max(0,e_u-b_u+\rho_{uv})$. 
The subspaces form a subrepresentation exactly when
$e_u-e_v\leq b_u-\rho_{uv}$ for every arrow. Writing $q=b-e$ gives
precisely \eqref{F:eq:P}. Thus each allowed integral vector contributes
one point to \eqref{F:eq:quotF}.

The constraint matrix in \eqref{F:eq:P} has rows of the form
$\pm\mathbf e_i$ or $\mathbf e_i-\mathbf e_j$ and is totally
unimodular. To see this, consider a square submatrix. If a row has at
most one nonzero entry, expand along it and argue by induction.
Otherwise every row has one $1$ and one $-1$, so the sum of its columns
is zero. All square determinants are therefore $0$, $1$, or $-1$.
The right sides are integers, and the polytope is bounded. Its vertices
are consequently integral, proving the Newton-polytope assertion.

For integer decomposition, let $q\in k\mathcal P(r)\cap\ZZ^{\mu_N}$
and write $q_i=ku_i+s_i$, with $0\leq s_i<k$. Set
\[ q_i^{(h)}=u_i+\one_{\{s_i\geq h\}},\qquad 1\leq h\leq k. \]
These vectors sum to $q$ and satisfy the coordinate bounds. If
$q_i-q_j\leq kc$ for an integer $c$, then either $u_i-u_j<c$, or
$u_i-u_j=c$ and $s_i\leq s_j$. In either case
$q_i^{(h)}-q_j^{(h)}\leq c$. Hence each $q^{(h)}$ lies in
$\mathcal P(r)$. This proves the integer decomposition property.
\end{proof}

\begin{corollary}\label{F:cor:collapse}
Under the gap condition,
\[
 [\mathbf t^c]\widehat F_w
 =\#\{q\in\mathcal P(r(w))\cap\ZZ^{\mu_N}:\mathsf Rq=c\},
 \qquad \operatorname{Newt}(\widehat F_w)=\mathsf R\mathcal P(r(w)).
\]
\end{corollary}
\begin{proof}
Apply the row-collapse specialization to \eqref{F:eq:Fgap}.
\end{proof}

\begin{corollary}\label{F:cor:Laurent}
Under the gap condition, every nonzero coefficient in the initial
Laurent expansion of $\Gen_w$ is one, and
\[ \operatorname{supp}_{\mathbf x}(\Gen_w) =\bigl(\g_w+B^{\mathsf T}\mathcal P(r(w))\bigr) \cap\ZZ^{\mathcal I_N}. \]
Its Laurent Newton polytope is
$\g_w+B^{\mathsf T}\mathcal P(r(w))$; with its induced
integral lattice, this polytope has the integer decomposition property.
\end{corollary}
\begin{proof}
The integer left inverse in Lemma~\ref{F:lem:normalization} gives
\[ B^{\mathsf T}\mathbb R^{\mu_N}\cap\ZZ^{\mathcal I_N} =B^{\mathsf T}\ZZ^{\mu_N}. \]
Thus the exponent map in the separation formula is injective and
identifies the exponent lattice with the integral lattice of its image. Apply
Theorem~\ref{F:thm:gap} and translate by $\g_w$.
\end{proof}

\subsection{Classical specialization}
\label{sec:collapsed-F-polynomials}
Under Fulton's specialization, the initial rectangular Schubert
polynomial becomes
$x_{a,r}\mapsto(x_1\cdots x_r)^a$.
Let $\mathsf M:\ZZ^{\mathcal I_N}\to\ZZ^N$ be the corresponding
map on exponents,
$\mathsf M(\mathbf e_{a,r})=a(\eps_1+\cdots+\eps_r)$. Define
\[ \lambda(w)=\operatorname{sort}^{\downarrow}(e_1(w),\ldots,e_N(w)), \qquad \mu(w)=\lambda(w)'. \]
Here a prime denotes the conjugate partition.

\begin{lemma}\label{lem:seed-identities}
For every $w$ and every mutable index $(a,r)$,
\[ \mathsf M(\g_w)=\mu(w),\qquad \mathsf M(B^{\mathsf T}\mathbf e_{a,r})=-\eps_r+\eps_{r+1}. \]
\end{lemma}
\begin{proof}
Summing \eqref{eq:Schubert-g-vector} with coefficient $a$ gives
$\sum_a a\g_w(a,r)=\#\{j:e_j(w)=r\}$. The $k$th coordinate of
$\mathsf M(\g_w)$ is therefore $\#\{j:e_j(w)\geq k\}=\mu_k(w)$.
For the second identity, substitute
$\mathsf M(\mathbf e_{a,r})=a(\eps_1+\cdots+\eps_r)$ in the three
outgoing and three incoming directions of the triangular seed. All
terms cancel except $-\eps_r+\eps_{r+1}$. A missing zero-width
rectangle contributes zero, so the same calculation includes the
boundary.
\end{proof}

\begin{proposition}\label{prop:collapsed-character}
The classical specialization of the generic function is
\[ P_w^{\rm gen}(x)=x^{\mu(w)}\widehat F_w \left(\frac{x_2}{x_1},\frac{x_3}{x_2},\ldots, \frac{x_{N-1}}{x_{N-2}}\right), \]
with an empty list of ratios when $N\leq2$. This expression is a
polynomial.
\end{proposition}
\begin{proof}
Apply $\mathsf M$ to \eqref{eq:separation} and use
Lemma~\ref{lem:seed-identities}. Polynomiality follows from
$\Gen_w\in\cM_N$ and Fulton's specialization.
\end{proof}

\subsection{Examples}

\begin{example}[A cover tree with a merge]\label{ex:merging-cover-tree}
For $w=(2,5,4,1,3)$, we have
\[ e(w)=(0,0,1,3,2),\qquad r(w)=(1,2,2,1,3),\qquad d(w)=(1,2,2,1). \]
The nonempty slots are $3,4,5$, and the cover edges are $3\to5$ and
$4\to5$. Thus $X_{r(w)}$ is indecomposable by
Proposition~\ref{prop:cover-components-indecomposable}. Its full
initial vector is
\[ \g_w=-\mathbf e_{1,1}+\mathbf e_{2,1}-\mathbf e_{2,2} +\mathbf e_{3,2}+\mathbf e_{1,3}. \]
The nonempty indexing sets are
$J_{1,1}=\{3\},\ J_{2,1}=\{5\}$, and $J_{2,2}=\{5\}$.
The nonzero arrows are $E_{1,1}\to E_{2,1}$ and
$E_{2,2}\to E_{2,1}$, and
\[ F_w=1+y_{1,1}+y_{2,2}+y_{1,1}y_{2,2} +y_{1,1}y_{2,1}y_{2,2}. \]
Writing $(q_{1,1},q_{2,1},q_{2,2})=(u,v,t)$, its polytope is
\[ 0\le u,v,t\le1,\qquad v\le u,\qquad v\le t. \]
The two inequalities express the two arrows entering $E_{2,1}$.
The record criterion places $w$ in $\mathcal C_5$. The coincidence
theorem will therefore identify this character with $\Sch_w(c)$;
Example~\ref{ex:component-variables} identifies it as a single
cluster variable. We return to the same five
lattice points in Example~\ref{ex:merging-cover-pipes}.
\end{example}

\begin{example}[Rigidity without the gap condition]\label{ex:rigid-nongap}
The multiplicity-one conclusion fails without the gap hypothesis,
even when the preprojective representation is rigid. For $w=(1,2,5,4,3)$, the
rank word is $(1,2,3,3,3)$; the gap condition fails at $(p,j)=(4,5)$.
The forward multisegment is
$[3,3]+[3,4]$. The coefficient spaces at $(1,2)$ and $(2,1)$ have
dimension two, and those at $(1,3)$ and $(2,2)$ have dimension one.
The two identity maps and the compatible inclusions identify the
kernels of a quotient of dimension one at all four vertices with a
single line in a two-dimensional space. This quotient Grassmannian is
$\mathbb P^1$, so $[y_{1,2}y_{1,3}y_{2,1}y_{2,2}]F_w=2$.
Nevertheless there are no extension pairs between the two forward
intervals, so $M_w$ is $Q$-rigid and its conormal fiber is zero. The
preprojective relation formula used in
Theorem~\ref{thm:gap-cover-representation} shows that $(M_w,0)$ is
also $\Lambda_Q$-rigid. Here
\[ \Gen_w=\Code_w=c_1(3)c_2(4),\qquad \Sch_w(c)=c_1(3)c_2(4)-c_3(4). \]
Thus rigidity does not imply the gap condition or Schubert--generic
coincidence. This rigid example nevertheless has a positive-dimensional quotient
variety. Under the gap condition, the stronger single-Jordan-block
property of the coefficient spaces yields multiplicity one.
\end{example}

\clearpage
\part{The homogeneous Schubert cluster locus}

\section{The homogeneous-locus theorem}
\label{sec:Schubert-cluster-problem}\label{sec:homogeneous-canonical}

We characterize the permutations $w$ for which $\Sch_w(c)$ is an
extended cluster monomial. The criterion is recursive in the
standardized prefixes of $w$. We give
intrinsic descriptions in Section~\ref{sec:intrinsic-record-locus}.

\subsection{Records}
For a profile
$\alpha=(\alpha_1,\ldots,\alpha_{N-1})\in
\op{Prof}_N:=\prod_{j=1}^{N-1}\{0,\ldots,j\}$, let
$\tau_\alpha$ be the unique permutation with
$e_{j+1}(\tau_\alpha)=\alpha_j$, and set
\[ \rho_1(\alpha)=1,\qquad \rho_q(\alpha)=q-\alpha_{q-1}\quad(2\le q\le N). \]
Thus $\rho(\alpha)=r(\tau_\alpha)$. Write
$\alpha=(\bar\alpha,t)$, set $h=N-t$, and let
$u=\tau_{\bar\alpha}$, with $p_z=u^{-1}(z)$.
The permutation $\tau_\alpha$ is obtained by appending an entry of
rank $h$ to $u$.

If $t>0$, an \emph{R1 record} is a value $x<h$ such that
\[ p_x>\max_{x<z\le h}p_z. \]
For $q=p_x$ and $s=h-x$, the corresponding profile is
\[ R1_x(\alpha)=\alpha-s\mathbf e_{q-1}+s\mathbf e_{N-1}. \]
The record is \emph{internal} if $\rho_q(\alpha)=x$.
If $h\ge2$, an \emph{R2 record} is a value $h<y\le N-1$
such that
\[
 p_{y-1}>\max_{\substack{h-1\le z\le y\\z\ne y-1}}p_z.
\]
The corresponding profile and internality condition are
\[ R2_y(\alpha)=\alpha-\mathbf e_{p_{y-1}-1}+\mathbf e_{N-1}, \qquad \rho_{p_{y-1}}(\alpha)=h-1. \]
Both profiles belong to $\op{Prof}_N$ by the calculation in
Appendix~\ref{app:record-obstructions}. We define $\mathcal C_0=\{\varnothing\}$,
$\mathcal C_1=\{1\}$, and
\[
 \tau_\alpha\in\mathcal C_N
 \quad\Longleftrightarrow\quad
 \tau_{\bar\alpha}\in\mathcal C_{N-1}
 \quad\text{and every record of $\alpha$ is internal.}
\]

\begin{theorem}[The homogeneous Schubert locus]
\label{thm:schubert-characterization}\label{thm:laurent-record-characterization}
For $w\in S_N$, the following are equivalent:
\begin{enumerate}[label=\textup{(\roman*)},leftmargin=2.2em]
\item $w\in\mathcal C_N$.
\item $\Sch_w(c)$ is homogeneous for diagonal conjugation on $U_N$.
\item $\Sch_w(c)=\Can_w$.
\item $\Sch_w(c)=\Gen_w$.
\item $\Sch_w(c)$ is an extended cluster monomial.
\item Its Laurent expansion in the initial triangular seed is nonnegative.
\item Its Laurent expansion in every seed is nonnegative.
\end{enumerate}
Laurent positivity refers to coefficients. Moreover, if
$w\notin\mathcal C_N$, there is a $v\in S_N$, with
$d(v)\ne d(w)$, such that
\begin{equation}\label{eq:negative-Schubert-vertex}
 \Pi_N^{-1}(w,v)
 = [\mathbf x^{\g_v}]\Sch_w(c)=-1
\end{equation}
in the initial Laurent expansion.
\end{theorem}

\paragraph{Proof outline.}
We first show that the record criterion is equivalent to the
standard-elementary expansion of $\Sch_w(c)$ being supported in
$I_{d(w)}$. M\"obius inversion and a parabolic Kazhdan--Lusztig
identity then identify this polynomial with $\Can_w$. The record
criterion also implies the gap condition. The reachability theorem of
Section~\ref{sec:compatibility-seeds} identifies
$\Gen_w=\Can_w$ with an extended cluster monomial; its proof does
not assume Schubert--generic coincidence.

To rule out initial Laurent positivity outside the record class, we
construct a coefficient $-1$ from an external record. We then show
that no other code monomial in the supported expansion contributes
to the same Laurent exponent. The coefficient obtained in this way has root degree different from
$d(w)$, as asserted.

\subsection{The coefficient criterion}
Write $\mathcal H_N$ for the locus in condition~\textup{(ii)} and
set $Q_N=\Pi_N^{-1}$. Thus
\[ \Sch_w(c)=\sum_{v\in S_N}Q_N(w,v)\Code_v. \]
The weight of $\Code_v$ is
$\sum_i d_i(v)(\eps_i-\eps_{i+1})$, and $Q_N(w,w)=1$.
\begin{lemma}\label{lem:torus-weight-criterion}
For $w\in S_N$,
\begin{equation}\label{eq:hom-profile-cut-weight}
 w\in\mathcal H_N
 \quad\Longleftrightarrow\quad
 Q_N(w,v)\ne0\ \Longrightarrow\ d(v)=d(w).
\end{equation}
\end{lemma}
\begin{proof}
The code monomials are linearly independent weight vectors, and the
coefficient of $\Code_w$ is one.
\end{proof}
For a general profile, we use the same notation
$d(\alpha)=d(\tau_\alpha)$ and
$\mathcal E_\alpha=\Code_{\tau_\alpha}$.

Let
\[ \theta(v)=(r_N(v)-1,\ldots,r_1(v)-1)=L(vw_0), \]
where $L$ denotes the Lehmer code. The Cauchy transpose identity and
the nil-Hecke pairing give the following formula.
\begin{lemma}\label{lem:complementary-code}
For $\ell(w)=\ell(v)$, we have
\[ Q_N(w,v)=\partial_{ww_0}x^{\theta(v)}. \]
If the lengths differ, $Q_N(w,v)=0$.
\end{lemma}
In a product of divided differences, the rightmost operator acts
first. When the lengths are equal, this divided difference is a scalar.
The inversion argument appears in Appendix~\ref{app:record-obstructions}.

\begin{proposition}\label{prop:record-obstructions}
Let $\alpha=(\bar\alpha,t)$.
For every profile $\beta\in\op{Prof}_{N-1}$, we have
\begin{equation}\label{eq:Q-rank-reduction}
 Q_N(\tau_\alpha,\tau_{(\beta,t)})
   =Q_{N-1}(\tau_{\bar\alpha},\tau_\beta).
\end{equation}
If $\gamma$ is either $R1_x(\alpha)$ or $R2_y(\alpha)$ for a
record of $\alpha$, then
\begin{equation}\label{eq:record-obstruction-coefficient}
 Q_N(\tau_\alpha,\tau_\gamma)=-1.
\end{equation}
Furthermore, $d(\gamma)=d(\alpha)$ if and only if that record is
internal. An internal R1 record exchanges the rank letters $x,h$;
an internal R2 record exchanges $h-1,h$.
\end{proposition}
\begin{proof}
Apply Lemma~\ref{lem:complementary-code} to the reduced word obtained by
removing its first Lehmer string. For a fixed final profile coordinate,
this gives \eqref{eq:Q-rank-reduction}. At a profile arising from a record, the first string has a unique
negative contribution to the constant term, giving
\eqref{eq:record-obstruction-coefficient}. The complete string
calculation is in Appendix~\ref{app:record-obstructions}. The last
assertion follows by comparing the two changed rank letters and using
Lemma~\ref{lem:d-content}.
\end{proof}

\begin{theorem}[Record criterion]\label{thm:record-characterization}
\label{thm:homogeneous-record-outer-bound}\label{prop:hom-transition-inputs}
One has $\mathcal H_N=\mathcal C_N$. Every member of this class
satisfies the gap condition of Definition~\ref{def:gap-word}.
\end{theorem}
\begin{proof}
Suppose that $\Sch_w(c)$ is homogeneous. Equation~\eqref{eq:Q-rank-reduction}
shows that its standardized prefix is homogeneous: equal root degrees
are equal multisets of rank letters, and the common last letter may be
deleted. An external record would contradict
\eqref{eq:hom-profile-cut-weight}, by
Proposition~\ref{prop:record-obstructions}. Induction gives
$w\in\mathcal C_N$. The same record calculation excludes a violation
of the gap condition; see Lemma~\ref{lem:sealed-gap}.

For the converse, we first work before specialization: a
multilinear matrix polynomial records the column weights of the
Schubert coefficients, and one marks the largest entry of the
permutation. Every admissible marked permutation falls into one of
four cases. If its maximum lies before the mark, one reduces the
marked position. A proper skew decomposition gives a product on
disjoint row sets. An increasing suffix permits deletion of its
terminal maximum. In the remaining case, the marked-insertion
normal form removes an intervening block of consecutive ranks; the
corresponding exterior-insertion identity preserves column weight.
The induction is lexicographic in the marked position and the size,
so each reduction is strictly smaller. Homogeneity of the resulting
matrix polynomial survives the row restriction that specializes it to
$\Sch_w(c)$. The complete statements, including the nonvanishing
needed to rule out an empty specialization, are in
Appendix~\ref{app:record-sufficiency}, culminating in
Theorem~\ref{thm:rs-saturation}.
\end{proof}

\subsection{Canonical coincidence}
The order on the root-degree fiber $I_d$ is the multiset Bruhat order
of Part~I. Write $\mu_{I_d}$ for its M\"obius function. We need two
inverse-row identities.

\begin{proposition}\label{prop:single-fiber-Mobius-row}
If the row $Q_N(w,-)$ is supported on $I_d$, where $d=d(w)$,
then
\[ Q_N(w,v)=\mu_{I_d}(w,v)\qquad(v\in I_d). \]
\end{proposition}
\begin{proof}
The key-polynomial filtration by decreasing rearrangement of the
exponent vector gives
\[
 \sum_{\substack{v\in I_d\\v\le z}}Q_N(w,v)=\delta_{w,z}.
\]
Terms of a different rearrangement vanish by the support assumption.
M\"obius inversion proves the result. We prove the key-polynomial
identity, with its order conventions, in
Appendix~\ref{app:single-fiber-rows}.
\end{proof}

Fix a root degree $d$, and set $W=S_N$. Let $W_J$ be the
parabolic subgroup whose consecutive blocks have sizes
$m_i=1+d_i-d_{i-1}$ for $1\le i\le N$, where $d_0=d_N=0$;
zero sizes are omitted. For $r\in I_d$,
let $y_r\in W^J$ list the positions of the letters of $r$ in
increasing letter order, with the positions of equal letters listed
increasingly. Write $w_J$ for the longest element of $W_J$ and
$P_{u,v}(q)$ for the ordinary Kazhdan--Lusztig polynomial. Set
\[ D^{\mathrm{can}}_{N,d}(r,v;q) =P_{y_rw_J,y_vw_J}(q). \]
Its specialization at $q=1$ is the canonical PBW comparison in
Appendix~\ref{app:classical-transitions}.

\begin{theorem}\label{thm:canonical-Mobius-row}
If $r(w)$ satisfies the gap condition, then
\[ \sum_{v\in I_d}\mu_{I_d}(w,v) D^{\mathrm{can}}_{N,d}(v,z;q)=\delta_{w,z} \qquad(z\in I_d). \]
\end{theorem}
\begin{proof}
Set $\phi(u)=w_0uw_J$. For two rank words of the same content,
the relevant entry of the inverse canonical matrix becomes, by
spherical--antispherical inversion, a signed parabolic sum
\[
 A_J(x,\tau;q)
 =\sum_{\substack{t\in W_J\\xt\le\tau}}
       (-1)^{\ell(t)}P_{xt,\tau}(q),
 \qquad x=\phi(y_v),\quad\tau=\phi(y_w).
\]
The calculation in Appendix~\ref{app:parabolic-KL} evaluates this sum
coefficientwise in $q$: it is one when the indexing parabolic fiber
is a singleton, and zero otherwise. To see the cancellation mechanism,
first normalize the lower word within the strict-ascent runs of the
upper word. Repeated values in a parabolic block admit a
sign-reversing involution preserving every Kazhdan--Lusztig polynomial.
Otherwise the gap condition supplies a common cancellable position,
whose deletion preserves the signed sum and the gap condition.
Induction then leaves the singleton case. Deodhar's relative
M\"obius formula identifies these values, with the inversion signs,
with $\mu_{I_d}(w,v)$; this proves the matrix identity.
All orientation and indexing conventions for this inversion are
verified in the appendix.
\end{proof}

\begin{corollary}\label{thm:homogeneous-canonical}
The polynomial $\Sch_w(c)$ is homogeneous if and only if
$\Sch_w(c)=\Can_w$. For $w\in\mathcal C_N$,
\begin{equation}\label{eq:homogeneous-Mobius-expansion}
 \Sch_w(c)=\sum_{v\in I_{d(w)}}\mu_{I_{d(w)}}(w,v)\Code_v.
\end{equation}
\end{corollary}
\begin{proof}
The canonical element is homogeneous for the root grading, proving
one implication. Conversely, the coefficient criterion and
Proposition~\ref{prop:single-fiber-Mobius-row} give
\eqref{eq:homogeneous-Mobius-expansion}. By Theorem~\ref{thm:record-characterization}, the rank word
satisfies the gap condition. Substituting the canonical expansion of each
$\Code_v$ and applying Theorem~\ref{thm:canonical-Mobius-row} at
$q=1$ leaves only $\Can_w$.
\end{proof}

\paragraph{The negative coefficient.}
To complete the characterization, we must exclude initial Laurent
positivity outside the record class.
The technical input is the exposed-boundary calculation in
Proposition~\ref{prop:lp-boundary}; the final reduction is short enough
to include here. For a profile $\alpha\in\op{Prof}_N$, set
$q_w(\alpha)=Q_N(w,\tau_\alpha)$ and
$\g(\alpha)=\g_{\tau_\alpha}$. The root-degree-preserving dominance
order on profiles is defined in \eqref{eq:lp-order}. Call a supported
profile \emph{maximal in its root degree} when no different supported
profile of the same root degree dominates it.

Write $\op{Laur}F$ for the Laurent expansion in the initial
seed and, for $u\in S_n$, write
$T_p(u)=(u_1,\ldots,u_p,n+1,u_{p+1},\ldots,u_n)$.

\begin{theorem}\label{thm:lp-negative-coefficient}
For every $w\notin\mathcal C_N$, there is a profile $\alpha$
of root degree different from $d(w)$, maximal among the
supported profiles of its root-degree fiber in the standard-elementary
expansion of $\Sch_w(c)$, such that
\[ q_w(\alpha)=-1,\qquad [\mathbf x^{\g(\alpha)}]\op{Laur}\Sch_w(c)=-1. \]
\end{theorem}
\begin{proof}
Suppose the assertion about the standard-elementary coefficient
is false. Choose a counterexample $w$ of least rank, and among those
choose one whose maximum lies as far to the right as possible. Write
$w=T_p(u)$.
If $u\notin\mathcal C_n$, minimality and Lemma~\ref{lem:lp-transport-max}(i)
produce the forbidden coefficient in $w$. Hence $u\in\mathcal C_n$.
The maximum of $w$ is not terminal: appending a terminal maximum to
a member of $\mathcal C$ creates no new record. If $T_{p+1}(u)$ lies outside $\mathcal C$,
its maximum is farther right, so it has the required coefficient;
Lemma~\ref{lem:lp-transport-max}(ii) transports it to $w$.
Thus $T_{p+1}(u)$ belongs to $\mathcal C$, and Lemma~\ref{lem:lp-bad-good}
again produces the coefficient. This contradiction proves the existence of the asserted
standard-elementary coefficient. By
Lemma~\ref{lem:lp-triangular}, maximality in its root degree prevents
other code monomials from changing the coefficient at the associated
initial Laurent exponent. For the label $v$
with profile $\alpha$, one has $q_w(\alpha)=\Pi_N^{-1}(w,v)$
and $d(v)\ne d(w)$, as asserted in
Theorem~\ref{thm:laurent-record-characterization}.
\end{proof}

\begin{proof}[Proof of Theorem~\ref{thm:schubert-characterization}]
The preceding results prove the equivalence of
\textup{(i)--(iii)}. By Theorem~\ref{solseed:thm:main}, the canonical
cover representation of a member of $\mathcal C_N$ is reachable.
Its character is therefore an extended cluster monomial, and
Corollary~\ref{cor:generic-cluster-monomial-canonical} identifies it
with $\Gen_w=\Can_w=\Sch_w(c)$. This proves
\textup{(i)}$\Rightarrow$\textup{(iv)--(v)}. Generic elements and
cluster monomials are homogeneous, so both converses follow.

Cluster positivity~\cite{ReachLeeSchiffler} gives
\textup{(v)}$\Rightarrow$\textup{(vii)}$\Rightarrow$\textup{(vi)};
the polynomial frozen convention is justified by
Proposition~\ref{thm:unipotent-seed}. Finally,
Theorem~\ref{thm:lp-negative-coefficient} constructs
\eqref{eq:negative-Schubert-vertex} for every $w\notin\mathcal C_N$.
It starts with a coefficient $-1$ that is maximal among the
supported profiles of its root degree. Path coordinates show that no
other profile can contribute to its leading Laurent monomial. This
proves \textup{(vi)}$\Rightarrow$\textup{(i)} and the last assertion.
\end{proof}

\begin{corollary}\label{cor:homogeneous-unit-rows}
For $\bullet\in\{\mathrm{gen},\mathrm{can}\}$, one has
\[ \mathsf B_w^\bullet=\Sch_w(c) \quad\Longleftrightarrow\quad w\in\mathcal C_N. \]
In this case
\[
 (D_N^\bullet)^{-1}(w,v)=
 \begin{cases}
 \mu_{I_{d(w)}}(w,v),&d(v)=d(w),\\
 0,&d(v)\ne d(w).
 \end{cases}
\]
\end{corollary}
\begin{proof}
The equality $\mathsf B_w^\bullet=\Sch_w(c)$ is part of
Theorem~\ref{thm:schubert-characterization}. Expanding it in the Code
basis and using \eqref{eq:homogeneous-Mobius-expansion} gives the last
formula.
\end{proof}
It remains to determine whether the same conclusion holds for the
MV basis.

\begin{conjecture}\label{conj:MV-homogeneous-locus}
For every $w\in\mathcal C_N$, one has $\MV_w=\Sch_w(c)$. 
\end{conjecture}

The reverse implication already follows from
Theorem~\ref{thm:schubert-characterization}: since $\MV_w$ is
homogeneous, equality $\MV_w=\Sch_w(c)$ forces $w\in\mathcal C_N$. The conjecture holds on the strict-minor family by
Corollary~\ref{cor:321-transition}. It also holds whenever
$\Sch_w(c)=\Code_w$: in that case $w$ avoids $312$, and
Corollary~\ref{cor:code-transition-extremals} gives
$\Code_w=\MV_w$. The conjecture would therefore extend
these proved equalities to
\[ \Sch_w(c)=\Gen_w=\Can_w=\MV_w \qquad(w\in\mathcal C_N). \]

\section{Intrinsic descriptions and component factorization}
\label{sec:intrinsic-record-locus}\label{sec:record-description}

We first express the record criterion in the permutation diagram.
These conditions determine the admissible positions for inserting a
new maximum. We then show that the connected components of the
cover representation give the irreducible factors of $\Sch_w(c)$,
each of which is again a universal Schubert polynomial.

\subsection{Empty rectangles}
Let $P(w)=\{(i,w_i):1\le i\le N\}$. All rectangles in this
section are open. A \emph{type I configuration} consists of
$i<j<k<l$ such that
\[ w_i>w_l>w_j>w_k, \qquad P(w)\cap((j,l)\times(w_j,w_i))=\varnothing. \]
A \emph{type II configuration} consists of $i<j<k<l$ such that,
for $a=\min(w_i,w_j)$ and $d=\max(w_i,w_j)$,
\[ a<w_l<w_k<d, \qquad P(w)\cap((k,l)\times(a,d))=\varnothing. \]
For the latter set
\[ L(k,l)=\#\{s:k<s<l,\ w_s<w_l\},\qquad M(k,l)=\#\{s<k:w_l<w_s<w_k\}. \]

\begin{theorem}\label{thm:record-rectangles}
A permutation belongs to $\mathcal C_N$ if and only if it has no
type I configuration and every type II configuration satisfies
\begin{equation}\label{eq:rectangle-balance}
 L(k,l)=M(k,l)+1.
\end{equation}
\end{theorem}
\begin{proof}
Consider the insertion at position $l$, and standardize its parent.
An R1 record selects a position $j$ with $w_j<w_l$. The upper
comparison value is the least earlier value above $w_l$; let its
position be $i$. The record condition places $i$, and every
comparison value between $w_j$ and that upper value, before $j$.
The record is external precisely when a smaller value occurs at some
$k$ with $j<k<l$. These four positions form a type I
configuration. Conversely, a type I configuration supplies that R1
record, using the least earlier value above $w_l$ as upper pivot.
Its smaller value at $k$ makes the record external.

An R2 record selects $k$ with $w_k>w_l$. Its lower and upper
comparison values are the greatest earlier value below $w_l$ and
the least earlier value above $w_k$. Both precede $k$, in
either order, and the record condition is exactly the emptiness of
the intervening rectangle. Thus these records are precisely the type
II configurations, with the two extreme comparison values chosen as
anchors. If $A=\#\{s<k:w_s<w_l\}$, then
\[ r_k=1+A+M(k,l),\qquad r_l=1+A+L(k,l). \]
Internality is $r_k=r_l-1$, which is
\eqref{eq:rectangle-balance}. Applying these observations at every
prefix proves the theorem.
\end{proof}

\subsection{Rothe columns and Bruhat covers}
The Rothe diagram is
\[ D(w)=\{(i,c):c<w_i,\ i<w^{-1}(c)\}. \]
Let $h_c$ be the height of its $c$th column. We decorate the dot
$(i,w_i)$ by
\[ \kappa_i=i-h_{w_i}=r_i(w). \]
For $p<q$, with $a=w_p$, $b=w_q$, call $(p,q)$ a
\emph{visible arc} if no value strictly between $a$ and $b$
occurs strictly between positions $p$ and $q$. Set
\[
\begin{aligned}
 \alpha_{pq}&=\max\bigl(\{0\}\cup
  \{w_t:p<t<q,\ w_t<\min(a,b)\}\bigr),\\
 \beta_{pq}&=\min\bigl(\{N+1\}\cup
  \{w_t:p<t<q,\ w_t>\max(a,b)\}\bigr).
\end{aligned}
\]
Define $G^-_{pq},G^+_{pq}\in\{0,1\}$ by
\[
\begin{aligned}
 G^-_{pq}=1&\ \Longleftrightarrow\
    \exists i<p:\ \alpha_{pq}<w_i<\min(a,b),\\
 G^+_{pq}=1&\ \Longleftrightarrow\
    \exists i<p:\ \max(a,b)<w_i<\beta_{pq}.
\end{aligned}
\]
These indicators record the presence of an earlier lower or upper
anchor in the indicated gap. The arc is upward for $a<b$ and
downward for $a>b$. Visible arcs correspond to the strong Bruhat covers incident
to $w$, since for $a<b$
\[ \ell(wt_{pq})-\ell(w) =1+2\#\{t:p<t<q,\ a<w_t<b\}, \]
with the analogous negative formula when $a>b$. Thus the visible
arc graph records Bruhat covers, whereas $\mathcal F(r(w))$ records
cover pairs of interval modules.

\begin{theorem}\label{solrothe:thm:main}
Membership in $\mathcal C_N$ is equivalent to the following two
conditions:
\begin{enumerate}[label=\textup{(\alph*)},leftmargin=2em]
\item An upward visible arc with $\alpha_{pq}>0$ has
      $G^+_{pq}=0$.
\item A downward visible arc with $G^-_{pq}=G^+_{pq}=1$ satisfies
\begin{equation}\label{solrothe:eq:height}
 \kappa_q=\kappa_p+1,
 \qquad\text{equivalently}\qquad
 h_{w_q}-h_{w_p}=q-p-1.
\end{equation}
\end{enumerate}
\end{theorem}
\begin{proof}
An upward arc supports a type I configuration exactly when it has an
interior value below its left endpoint and an earlier upper anchor
below its upper wall. These are $\alpha_{pq}>0$ and
$G^+_{pq}=1$. Indeed, the anchor and the two endpoints bound the
empty rectangle, and any interior smaller value supplies the fourth
point.

A downward arc supports a type II configuration exactly when both
anchor indicators are one. Visibility and the definitions of the
walls make its rectangle empty. For every choice of anchors,
\[ L(p,q)-M(p,q)=r_q-r_p=\kappa_q-\kappa_p. \]
Thus balance is precisely \eqref{solrothe:eq:height}. Apply
Theorem~\ref{thm:record-rectangles}.
\end{proof}

For example, the arc joining positions $2,4$ in $4213$ is
upward, has an interior lower value, and has an earlier upper anchor.
It violates condition~\textup{(a)}. In $1432$ and $4132$,
the final downward arc has both anchors and equal ranks at its two ends,
violating~\textup{(b)}. Indeed, these three permutations are exactly the members of
$S_4\setminus\mathcal C_4$.

\subsection{Insertion of a largest entry}
For $w\in\mathcal C_N$, set
\[ T_c(w)=(w_1,\ldots,w_c,N+1,w_{c+1},\ldots,w_N), \qquad 0\le c\le N. \]
Let $\mathscr U(w)$ consist of the visible arcs with
$\beta_{pq}=N+1$ satisfying either
\[ a<b,\quad \alpha_{pq}>0, \qquad\text{or}\qquad a>b,\quad G^-_{pq}=1,\quad\kappa_q\ne\kappa_p+1. \]
Define
\[ b(w)=\max\bigl(\{0\}\cup\{p:(p,q)\in\mathscr U(w)\}\bigr). \]

\begin{corollary}\label{solrothe:cor:threshold}
One has
\[ T_c(w)\in\mathcal C_{N+1}\quad\Longleftrightarrow\quad c\ge b(w). \]
\end{corollary}
\begin{proof}
Inserting the largest entry does not change the visibility, the
ranks at the ends, or the lower-anchor indicators of an existing arc.
Insertion inside an arc neither obstructs it nor provides an earlier
anchor. Insertion after the arc leaves it unchanged; insertion before
its initial position creates a new upper anchor exactly when the old upper
wall was $N+1$. Since $w$ already satisfies the two arc
conditions, the newly forbidden arcs are exactly $\mathscr U(w)$
with $c<p$. An arc incident to the new largest entry has no upper
anchor and introduces no further condition.
\end{proof}

The arc criterion, and hence $b(w)$, can be evaluated in quadratic
time with linear auxiliary storage. For each initial position, scan
the positions to its right, recording the nearest intervening values
below and above the initial value. Prefix
counts determine whether an anchor lies in either gap. The scan tests
visibility before updating these two interior values. Thus neither a
Schubert expansion nor repeated maximum insertions are needed.

\subsection{Pattern descriptions}
Write $\mathcal C=\bigcup_N\mathcal C_N$.
\begin{proposition}\label{thm:record-classical-subclass}
The unique largest classical permutation class contained in
$\mathcal C$ (ordered by inclusion) is
\[ \Av(1432,4132,4213). \]
The inclusion is strict: $25413\in\mathcal C_5$, although it
contains $1432\notin\mathcal C_4$. In particular,
$\mathcal C$ is not a classical permutation class.
\end{proposition}
\begin{proof}
A type I configuration has pattern $4213$; a type II configuration
has pattern $1432$ or $4132$, according to the order of its
anchors. Avoiding these three patterns therefore implies membership
by the rectangle criterion. Each of the three patterns fails that
criterion. A classical class contained in $\mathcal C$ must avoid
them. The stated example follows directly from the record definition.
\end{proof}

A mesh pattern consists of a permutation together with shaded open
cells between its selected positions and value ranks; an occurrence
requires these cells to contain no additional permutation points.
The obstruction to a finite mesh-pattern description can be seen in
an explicit family: increasing the length of an unselected block can
change the record criterion without affecting any bounded-size mesh
pattern. We give the construction, since it is the essential content
of the negative result.

\begin{lemma}\label{lem:mesh-indistinguishability}
For every $k\ge1$, there are permutations $w^+\in\mathcal C$
and $w^-\notin\mathcal C$ that contain exactly the same ordinary
mesh patterns of length at most $k$.
\end{lemma}
\begin{proof}
For $a\ge1$, $b\ge0$, and $c\ge1$, set
$h=a+c+1$, $n=a+b+c+3$, and
\[ w_{a,b,c}=(c+1,\ldots,c+a,\ h+1,\ldots,h+b,\ n,n-1, \ 1,\ldots,c,\ h). \]
The second interval is empty for $b=0$. This permutation has no
type I configuration. Indeed, an entry admitting a smaller later
entry and a larger earlier entry must be $n-1$, or lie in the
final increasing low block; the latter has no smaller later entry,
and no possible fourth value lies between $n-1$ and its earlier
larger entry $n$.
For a type II configuration, the third entry must be $n-1$,
and the fourth must be the terminal $h$: a low entry has no
smaller preceding anchor. The upper anchor is $n$, and the
lower anchor may be any of $c+1,\ldots,c+a$. Its rectangle
is empty, with counts $L=c$ and $M=b$. Thus
\[ w_{a,b,c}\in\mathcal C\quad\Longleftrightarrow\quad c=b+1. \]
Every proper prefix has no type II configuration and belongs to
$\mathcal C$.

Take $w^+=w_{k+1,k+1,k+2}$ and
$w^-=w_{k+1,k+1,k+3}$. Consider an occurrence of a mesh pattern
of length at most $k$ in $w^-$. At least one point of its
consecutive increasing low block is unselected. Delete such a point.
No selected point is removed and no shaded cell acquires a point,
so the occurrence survives. Standardization gives $w^+$.
Conversely, in an occurrence in $w^+$, choose an unselected low
point. Replace it by two consecutive increasing points. Both lie in
the same cell relative to the selected points. That cell is unshaded,
since it already contained the unselected point. Hence the occurrence
survives, and standardization gives $w^-$. This proves the
claim, including mesh patterns with outer shaded cells.
\end{proof}

\begin{theorem}\label{solsw:thm:main}\label{thm:record-infinite-mesh}
No finite family of ordinary mesh patterns defines $\mathcal C$.
\end{theorem}
\begin{proof}
For each $k$, Lemma~\ref{lem:mesh-indistinguishability} constructs a
member and a nonmember of $\mathcal C$ with exactly the same mesh
patterns of length at most $k$. Taking $k$ to be the greatest length
in a proposed finite family gives a contradiction.
\end{proof}
Thus the balance condition in the rectangle criterion cannot be
replaced by finitely many mesh-emptiness conditions. This does not
exclude descriptions by decorated patterns that distinguish the two counts.

\subsection{Cover components}
Let $\mathcal F_w=\mathcal F(r(w))$ be the cover forest of
$w\in\mathcal C_N$, as in Part~I. Its vertices are the nonempty
slots $j$, with interval $[r_j,j-1]$. An edge is directed from
an earlier slot to its parent. Every vertex has at most one parent.

\begin{lemma}[Insertion of covers]\label{cor:record-insertion}
\label{thm:record-insertion}
When a new rank is appended to a member of $\mathcal C$, subject
to the record criterion, its new incoming cover edges are exactly
those selected by the internal R1 and R2 records. There is at most one
of each type, and their sources, when present, are roots of distinct
old components. Existing covers do not change.
\end{lemma}
\begin{proof}
An intermediate slot factoring an extension pair lies strictly between
its endpoints, so appending a slot does not change old covers. By
Proposition~\ref{thm:gap-cover-minor}, a new extension pair $(q,N)$,
with $r_q<h\le q$, is a cover precisely when no later slot before
$N$ has rank in $[r_q,h)$. If the selected parent value at
$q$ is below $h$, this condition is exactly the R1 record
condition; internality gives $r_q=x$. If it is at least $h$,
it is exactly the R2 condition; internality gives $r_q=h-1$.
The possible equality at the upper end of the interval is excluded by
the exact-predecessor consequence of the gap condition. We give the
remaining details in Appendix~\ref{app:cover-insertion}.

An internal R1 record is the rightmost position among the values
$1,\ldots,h$, so is unique. Two internal R2 records, in increasing
value order, would place the first selected value before the second,
in addition to its $h-2$ smaller predecessors. This contradicts
rank $h-1$ at the second. Finally, a new edge must leave an old
root, since no vertex has two parents; distinct roots belong to
distinct components.
\end{proof}

\begin{lemma}\label{solcomp:lem:boundary}
If $A$ is a nonempty connected component of $\mathcal F_w$ and
$m=\max A$, then the first $m$ values of $w$ form a
consecutive interval. No cover crosses the cut after $m$.
\end{lemma}
\begin{proof}
In the standardized prefix of length $m$, the first $m$ values are
$1,\ldots,m$. Follow this block through the subsequent insertions.
Suppose a later insertion first breaks this property. In its
standardized parent, the initial values form $[a+1,a+m]$, and
the new rank satisfies $a+1<h\le a+m$. The value in position
$m$ is $x=a+r_m<a+m$. If $x<h$, it is an R1 record:
all values $x+1,\ldots,h$ occur before it. If $x\ge h$,
it is an R2 record with upper comparison value $x+1$; all the
other comparison values again precede it. The record is internal,
so Lemma~\ref{cor:record-insertion} supplies a cover from $m$
to the new slot. This cover persists, contradicting $m=\max A$.

At any subsequent insertion outside the initial consecutive interval,
no record selects a value inside it. An R1 comparison contains the
later outside upper pivot, unless R1 is absent; an R2 comparison
contains the later outside lower pivot, unless R2 is absent. Thus no
new cover can cross the cut, again by
Lemma~\ref{cor:record-insertion}.
\end{proof}

\begin{theorem}\label{solcomp:thm:closure}
For a component $A$ of $\mathcal F_w$, let $v_A\in S_N$
be defined by the rank word
\[
 r_i(v_A)=\begin{cases}r_i(w),&i\in A,\\i,&i\notin A.\end{cases}
\]
Then $v_A\in\mathcal C_N$, and
\begin{equation}\label{eq:component-Schubert-factorization}
 \Sch_w(c)=\prod_{A\in\pi_0(\mathcal F_w)}\Sch_{v_A}(c).
\end{equation}
The factors are distinct variables of a common extended cluster
and are irreducible in $\CC[U_N]$. Consequently,
\eqref{eq:component-Schubert-factorization} is the squarefree
irreducible factorization. The equality uses the original universal Chern-class variables; no
rescaling by frozen monomials is involved. For
$w=\mathrm{id}_N$, the forest is empty and the product is one.
\end{theorem}
\begin{proof}
Order the components by their maximum slots. Let $m$ be the least
such maximum. By Lemma~\ref{solcomp:lem:boundary}, the first $m$
values form a consecutive interval, and their only nonempty component
is the one ending at $m$. Sorting this initial interval increasingly
replaces its ranks by $r_i=i$ and leaves every later rank unchanged.
It also preserves the later record tests. Indeed, a later selected
value lies outside the interval, and its comparison range contains
either the whole interval or none of it; every position in the
interval already precedes the selected position. The internality
condition at that position depends only on its unchanged rank.
The sorted permutation therefore remains in $\mathcal C_N$.

The standardized initial interval is a prefix in $\mathcal C_m$.
Padding it by terminal fixed maxima gives exactly $v_A$. Repeat
on the sorted permutation. The later cover relations are unchanged,
since their possible intermediate slots lie after the deleted initial
interval. This proves $v_A\in\mathcal C_N$ for every component.

The normalized cover representation is the direct sum of the
representations on its connected components:
\[ X_{r(w)}\simeq\bigoplus_A X_{r(v_A)}. \]
Cluster characters are multiplicative on direct sums. Applying
Theorem~\ref{thm:schubert-characterization} to $w$ and to each
$v_A$ gives \eqref{eq:component-Schubert-factorization}.

Proposition~\ref{prop:cover-components-indecomposable} says that these
are precisely the pairwise nonisomorphic indecomposable summands.
The containing seed of Theorem~\ref{solseed:thm:main} contains all of
them, so each character is an extended cluster variable. Lemma~\ref{lem:extended-variable-irreducible}
makes every factor irreducible. Variables in this one seed are algebraically independent, so no
two of them are scalar multiples of one another. They are compatible
because they occur in the same seed. This proves squarefreeness and
uniqueness in the polynomial ring $\CC[U_N]$.
\end{proof}

\begin{corollary}[The single-variable criterion]
\label{cor:Schubert-single-variable}
Let $w\ne\mathrm{id}_N$. Then $\Sch_w(c)$ is an extended
cluster variable if and only if $w\in\mathcal C_N$ and
$\mathcal F_w$ is connected. For $w\in\mathcal C_N\setminus\{\mathrm{id}_N\}$,
these conditions are also equivalent to irreducibility in $\CC[U_N]$.
Moreover, for every $w\in\mathcal C_N$,
\[ \Des(w)=\bigsqcup_{A\in\pi_0(\mathcal F_w)}\Des(v_A),\qquad |\pi_0(\mathcal F_w)|\leq|\Des(w)|. \]
\end{corollary}
\begin{proof}
An extended cluster variable is root homogeneous. Thus, if
$\Sch_w(c)$ is such a variable, the homogeneous-locus theorem
first gives $w\in\mathcal C_N$. The component factorization and its
irreducibility assertion then give the equivalences. The descent
partition follows by applying Proposition~\ref{prop:factor-descent}
and Lemma~\ref{lem:descent-statistic} to its component factors.
\end{proof}

\begin{example}\label{ex:component-variables}
For $w=25413$ from Example~\ref{ex:merging-cover-tree}, the forest is
connected. Consequently
\[
 \Sch_{25413}(c)=c_1(2)c_3(3)c_2(4)-c_1(2)c_1(3)c_4(4)+c_2(3)c_4(4)-c_3(3)c_3(4)
\]
is one cluster variable, despite the negative coefficients in these
Chern-class coordinates. Its initial Laurent expansion instead has the
five positive terms specified by its $F$-polynomial in Part~I.
For contrast, $w=2143$ has two isolated nonempty slots, $2$ and $4$.
Its component permutations are $2134$ and $1243$, and
\[ \Sch_{2143}(c)=\Sch_{2134}(c)\Sch_{1243}(c)=c_1(1)c_1(3). \]
This is a product of two distinct cluster variables, not one variable.
\end{example}

\section{Compatibility and containing seeds}
\label{sec:compatibility-seeds}\label{sec:forest-compatibility}

The results of this section hold for every gap word. We first compute
extension spaces between the normalized cover representations and obtain
a compatibility criterion for their generic functions. We then remove
the least nonempty slot to construct containing seeds, with an explicit
bound on the number of mutations. On $\mathcal C_N$, these become
statements about the universal Schubert polynomials themselves.

\subsection{Two mixed forests}
Let $u,v\in S_N$ have gap rank words $r=r(u)$ and
$s=r(v)$. We compute Hom and extension spaces over $\Lambda_Q$ between the
normalized representatives $X_r$ and $X_s$ of
Theorem~\ref{thm:gap-cover-representation}. Set $I=\{i:r_i<i\}$ and $J=\{j:s_j<j\}$.
Define
\[
\begin{aligned}
 \HH(u,v)&=\{(i,j)\in I\times J:r_i\le s_j<i\le j\},\\
 \EE(u,v)&=\{(i,j)\in I\times J:r_i<s_j\le i<j\}.
\end{aligned}
\]
Define a bipartite graph $\Gamma(u,v)$ with vertex set
$\HH(u,v)\sqcup\EE(u,v)$. Write $\HH_{i,j}$ and $\EE_{i,j}$
for its two types of vertices; a pair belonging to both indexing sets
therefore gives two distinct vertices.
Its edges, whenever both endpoints exist, are
\begin{equation}\label{eq:mixed-forest-edges}
\begin{aligned}
 \HH_{i,c}&\longrightarrow\EE_{i,j}
       &&\text{for a cover }c\to j\text{ of }v,\\
 \EE_{i,j}&\longrightarrow\HH_{p,j}
       &&\text{for a cover }i\to p\text{ of }u.
\end{aligned}
\end{equation}
The orientation is used only in the proof; matchings are undirected.
Let $\nu(u,v)$ be the maximum matching number and set
\[ \delta(u,v)=|\EE(u,v)|-\nu(u,v). \]

\begin{theorem}[Mixed-forest formula]\label{solmatch:thm:main}
Both $\Gamma(u,v)$ and $\Gamma(v,u)$ are forests, and
\begin{align}
 \dim\Hom_{\Lambda_Q}(X_r,X_s)
     &=|\HH(u,v)|-\nu(u,v),\label{eq:mixed-Hom}\\
 \dim\Ext^1_{\Lambda_Q}(X_r,X_s)
     &=\delta(u,v)+\delta(v,u).\label{eq:mixed-Ext}
\end{align}
Consequently $\Gen_u$ and $\Gen_v$ are monomials in a common
seed if and only if each mixed forest has a matching covering every
$\EE$-vertex. For $u,v\in\mathcal C_N$, this is the corresponding
criterion for $\Sch_u$ and $\Sch_v$.
\end{theorem}
\begin{proof}
Use label-preserving duality to the rightward quiver $Q^{\rm op}$,
and write $A=X_r^\vee$, $B=X_s^\vee$. The usual interval
calculation gives
\[ \dim\Hom_{Q^{\rm op}}([s_j,j-1],[r_i,i-1])=\one_{\HH(u,v)}(i,j), \]
\[ \dim\Ext_{Q^{\rm op}}^1([r_i,i-1],[s_j,j-1])=\one_{\EE(u,v)}(i,j). \]
A forward-quiver morphism is a preprojective morphism precisely when
it commutes with the reverse maps. In the normalized cover bases, the
commutator matrix has columns $\HH(u,v)$, rows $\EE(u,v)$,
and entries
\[ D_{(i,j),(i,c)}=1\quad(c\to j),\qquad D_{(i,j),(p,j)}=-1\quad(i\to p), \]
with all other entries zero. The two rules give distinct matrix
positions. To verify the first, compose the Hom map supported on
$[s_c,i-1]$ with the cover map supported on
$[s_j-1,c-1]$; their common support is
$[s_j-1,i-1]$, the basis vector indexed by $(i,j)$.
The other composition gives the second entry with the opposite sign.
The interval inequalities are precisely those defining $\HH$ and
$\EE$. Thus $\Hom_{\Lambda_{Q^{\rm op}}}(B,A)=\ker D$.

In \eqref{eq:mixed-forest-edges}, each vertex has at most one outgoing
edge, because each slot has at most one parent. The sum of the two
indices strictly increases along every arrow. An undirected cycle
would have a vertex with two outgoing edges, so the graph is a forest.
A square submatrix of its signed adjacency matrix has at most one
nonzero determinant term: two different perfect matchings would give
an alternating cycle. Every matching of size $k$ therefore gives
a nonzero $k$-minor, and a nonzero minor gives a matching.
Hence $\rank D=\nu(u,v)$, proving \eqref{eq:mixed-Hom}.

For preprojective modules, the relation complex gives
\[ \dim\Ext^1(A,B)=\dim\Hom(A,B)+\dim\Hom(B,A) -(\dim A,\dim B), \]
where the last expression is the symmetric Cartan pairing. Indeed,
the zeroth and second cohomology of this three-term complex are
$\Hom(A,B)$ and $\Hom(B,A)^*$, while its middle cohomology
is $\Ext^1(A,B)$. The ordinary quiver Euler form gives
\[ (\dim A,\dim B)=|\HH(u,v)|+|\HH(v,u)| -|\EE(u,v)|-|\EE(v,u)|. \]
Substitution proves \eqref{eq:mixed-Ext}. Reachability and the
completion argument below identify extension-orthogonality with
membership in a common seed, proving the last assertion.
\end{proof}

\begin{corollary}[Leaf removal]\label{solmatch:cor:leaf}
Repeatedly remove an isolated vertex, counting it when it belongs to
$\EE$, or match a leaf to its neighbor and remove both. The number
of counted vertices is $\delta(u,v)$, independently of the
choices. The two counts sum to the extension dimension.
\end{corollary}
\begin{proof}
Some maximum matching contains any prescribed leaf edge. If that edge
is absent, either it can be added or it can replace the edge already
incident to its neighbor. Induction on the number of vertices proves
the assertion.
\end{proof}
The mixed forests have at most a constant multiple of $|I||J|$
vertices and edges, so construction and leaf removal are quadratic in
the number of slots.

\subsection{Completion and prescribed mutation paths}
We use the stable category of the GLS Frobenius category. For a rigid
object $R$ and a cluster-tilting object $T$, choose minimal
approximations in triangles
\[
 T[-1]\longrightarrow R'\longrightarrow X\longrightarrow T,
 \qquad
 T\longrightarrow Y\longrightarrow R''\longrightarrow T[1],
\]
where the first map is a left $\op{add}R$-approximation and the
last map in the second triangle is a right
$\op{add}R$-approximation. Set
\[ \mathcal P_R(T)=(R\oplus X)^{\mathrm{basic}},\qquad \mathcal Q_R(T)=(R\oplus Y)^{\mathrm{basic}}. \]
These are the left and right co-Bongartz completions used below. Projective-injective
summands are included on returning to the Frobenius category.

\begin{lemma}\label{solseed:lem:projections}
Both completions are cluster-tilting. As maps on the exchange
graph, they fix every vertex containing $R$ and send each edge to an
edge or a vertex. Every summand of $T$ that is extension-orthogonal to $R$ is also a
summand of its completion.
For $R\oplus R'$ rigid,
\[ \mathcal P_R\mathcal P_{R'}=\mathcal P_{R\oplus R'} =\mathcal P_{R'}\mathcal P_R, \]
and the analogous identity holds for $\mathcal Q$.
\end{lemma}
\begin{proof}
For $\mathcal P$, these are the completion and mutation theorems
of Cao~\cite[Theorems~4.1.2, 4.1.4, and 4.2.2]{SeedCao}.
The opposite category gives the assertions for $\mathcal Q$.
For an orthogonal summand, the approximation is zero and the summand
is unchanged by completion.
\end{proof}

Suppose that paths of lengths $a,b$ from the initial seed end at
seeds containing $R$ and $U$, respectively. Starting at the
first seed, traverse the first path backwards and then the second path.
Project this walk by either completion and prepend the first path.
We obtain a path to the corresponding completion of the second seed
of length at most
\begin{equation}\label{solseed:lem:path}
 2a+b.
\end{equation}
Appendix~\ref{app:effective-completions} gives integer $G$-matrix
tests that determine the projection of each edge. Thus the construction
uses the two given paths and does not require enumerating seeds.

\subsection{Removing the least slot}
We prove reachability by induction on the nonempty slots. For the
rightward quiver $Q^{\rm op}$, let $X$ be a normalized cover representation
whose slots have at most one parent and such that every extension pair
$i\prec j$ has a cover $c\to j$ with $i\vdash c$. The gap condition
implies both properties by Proposition~\ref{thm:gap-cover-minor}.
Here $\prec$ and $\vdash$ are the relations of
\eqref{eq:record-Ext-relation} and \eqref{eq:record-Hom-relation};
duality changes the direction of module maps, not the slot labels.

\begin{lemma}\label{solseed:lem:earliest}
Let $p$ be the least nonempty slot, $P=[r_p,p-1]$, and
$U=X/P$. The quotient has the same two properties, with its
original covers, and
\[ \Hom_{\Lambda_{Q^{\rm op}}}(P,U)=0,\qquad \dim\Ext^1_{\Lambda_{Q^{\rm op}}}(U,P)=\#\{j:p\to j\}\in\{0,1\}. \]
The natural exact sequence $0\to P\to X\to U\to0$ splits when
this number is zero and represents a nonzero class when it is one.
\end{lemma}
\begin{proof}
All intermediate factors of a pair lie between its endpoints.
Deleting the least slot changes no remaining cover. A cover
$c\to j$ dominating $i\prec j$ has $c\ge i>p$, so the second property still holds. Reverse maps run from parents
to children;
the least slot has no child and is a submodule.

The order of interval endpoints gives $\Hom_{Q^{\rm op}}(P,U)=0$ and
$\Ext_{Q^{\rm op}}^1(U,P)=0$. Thus every extension of $U$ by $P$
splits over the forward quiver. Its remaining reverse-arrow data give
\[ \Ext^1_{\Lambda_{Q^{\rm op}}}(U,P)=\coker\bigl( \Hom_{Q^{\rm op}}(U,P)\longrightarrow\Ext_{Q^{\rm op}}^1(P,U)^*\bigr). \]
Rows are indexed by $p\prec j$, and columns by $p\vdash c$ with
$c\in U$. A column has a single entry, in the row indexed by
$(p,\parn(c))$, when that row exists. There is no negative
commutator term, since $P$ has no reverse map. Every noncover row
is hit by a dominating cover, and no cover row is hit, since that
would factor the cover. The cokernel therefore has the indicated
basis. The extension defined by $X$ has coefficient one at its unique
cover row, if present. If there is no such row, the slot is isolated
and the sequence splits.
\end{proof}

Dualizing a nonsplit sequence from the lemma gives, in the
leftward convention,
\[ 0\longrightarrow D(U)\longrightarrow D(X) \longrightarrow D(P)\longrightarrow0. \]
The connecting map from the brick $D(P)$ to $D(U)[1]$ spans
its one-dimensional Hom space. It is therefore a minimal right
$\op{add}D(P)$-approximation. The completion
$\mathcal Q_{D(P)}$ of a seed containing $D(U)$ contains
$D(X)$. In the split case, the completion adds $D(P)$ without changing
$D(U)$. This remains true with multiplicities, because
approximations are additive. A projective-injective $P$ belongs
to the split case.

\begin{theorem}\label{solseed:thm:main} \label{thm:homogeneous-reachability}
For every gap word $r$, its normalized cover representation is
reachable. A containing seed can be constructed in at most
\begin{equation}\label{solseed:eq:global-bound}
 L_N=4\binom{N+1}{5}+2\binom{N+1}{4}
\end{equation}
mutations. More precisely, its path length is at most
\begin{equation}\label{solseed:eq:profile-bound}
 2\sum_{j:r_j<j}\ell_{\mathrm{int}}(r_j,j-1),\qquad
 \ell_{\mathrm{int}}(a,b)=2\left(\binom{b+1}{4}-\binom{a+1}{4}\right)
                     +\binom{b+1}{3}.
\end{equation}
For $w\in\mathcal C_N$, the resulting seed $\Sigma(w)$ and a
vector $a(w)\in\{0,1\}^{\mathcal I_N}$ satisfy
\[ G_{\Sigma_N^U}(\Sigma(w))a(w)=\g_w, \qquad \Sch_w(c)=\prod_{v\in\mathcal I_N}x_{v;\Sigma(w)}^{a_v(w)}. \]
In particular, every frozen exponent is zero or one.
\end{theorem}
\begin{proof}
A thin interval $[a,b]$ has a containing path of length at most
$\ell_{\mathrm{int}}(a,b)$. We construct these paths in
Appendix~\ref{app:interval-paths}: the triangular maximal green word
reaches the simple at $i$ in $\binom{i+1}{3}$ mutations,
and the sequence
\[ 0\to S_a\to[a,b]\to[a+1,b]\to0 \]
allows successive completion by the simple at $a$.

Now remove the least nonempty slot. The preceding lemma and
completion argument construct a containing seed from one for the
quotient and one for its interval. Equation~\eqref{solseed:lem:path}
gives
\[ \ell(X)\le2\ell_{\mathrm{int}}(r_p,p-1)+\ell(U). \]
Induction proves \eqref{solseed:eq:profile-bound}. Since
$\ell_{\mathrm{int}}(a,b)\le2\binom{b+1}{4}+\binom{b+1}{3}$, summing over
$b\le N-1$ proves \eqref{solseed:eq:global-bound}.
The induction uses only the two properties preceding the lemma, not
membership in $\mathcal C_N$. The cover representative is a
point of the dense orbit by Part~I, so its character is the generic
element at the prescribed full vector.

For $w\in\mathcal C_N$, Corollary~\ref{thm:homogeneous-canonical}
and the canonical comparison for cluster monomials identify this
character with $\Sch_w(c)$. The containing module determines the
exponents, including those of the projective-injective summands.
Proposition~\ref{prop:cover-components-indecomposable} shows that its
indecomposable summands are pairwise nonisomorphic, so every exponent
is zero or one. Equivalently, their exponents are determined by the full $G$-matrix
equation. The construction of the mutation path
and this full-vector calculation are detailed in
Appendix~\ref{app:effective-completions}.
\end{proof}

\begin{corollary}[Squarefree generic cluster monomials]
\label{cor:gap-squarefree-factors}
For every gap word $r=r(w)$, one has
\[ \Gen_w=\prod_{A\in\pi_0(\mathcal F(r))}\varphi_{X_r[A]}, \]
where $\varphi$ denotes the GLS cluster character. These factors are
distinct extended cluster variables in a common seed, and the product
is the squarefree irreducible factorization in $\CC[U_N]$.
If $\Gen_w\ne1$, this function is an extended cluster variable
if and only if $\mathcal F(r)$ is connected.
\end{corollary}
\begin{proof}
Combine Proposition~\ref{prop:cover-components-indecomposable}, the
containing-seed theorem, multiplicativity of cluster characters, and
Lemma~\ref{lem:extended-variable-irreducible}.
\end{proof}

\begin{corollary}\label{solmatch:cor:families}
Let $w_1,\ldots,w_k\in\mathcal C_N$. The polynomials
$\Sch_{w_1}(c),\ldots,\Sch_{w_k}(c)$ are monomials in a common
seed if and only if every pair satisfies the two matching
conditions of Theorem~\ref{solmatch:thm:main}. Given individual paths
of lengths $\ell_1,\ldots,\ell_k$, a common seed is reached in
at most $2\sum_i\ell_i\le2kL_N$ mutations.
\end{corollary}
\begin{proof}
The pairwise matching conditions make the direct sum of the
corresponding rigid modules rigid. Fix one completion throughout. For the basic nonprojective
support $R_i$ of the $i$th module, follow its individual path
and then the projection of its reverse. This reaches
$\mathcal P_{R_i}(T_0)$ in at most $2\ell_i$ steps.
Project these paths successively by the completions already used.
Commutation in Lemma~\ref{solseed:lem:projections} makes the final
cluster-tilting object $\mathcal P_{\oplus_iR_i}(T_0)$, and no projected path
is longer than the original one. The converse follows from rigidity
of a cluster-tilting module. Frozen summands remain unchanged.
\end{proof}

The matching test determines whether the two generic functions are
monomials in a common extended cluster. When the test succeeds, the proof
constructs a mutation path to such a seed.

\section{Pipe dreams and the homogeneous Newton polytope}
\label{sec:homogeneous-pipe-dreams}\label{app:pipe-dream-bijection}

For $w\in\mathcal C_N$, Part~I gives the $F$-polynomial and
Laurent Newton polytope of $\Sch_w(c)$. Reduced pipe dreams index
the lattice points of the full coefficient polytope, not just their
images under row collapse. The signature map is defined combinatorially and is injective.
Its surjectivity for $w\in\mathcal C_N$ follows from the
Schubert--generic coincidence theorem.

\begin{corollary}\label{cor:homogeneous-full-F}
If $w\in\mathcal C_N$, then
\[ \Sch_w(c)=\mathbf x^{\g_w} \sum_{q\in\mathcal P(r(w))\cap\ZZ^{\mu_N}} \mathbf x^{B^{\mathsf T}q}. \]
Every nonzero coefficient in this initial Laurent expansion is one,
and
\[ \op{supp}_{\mathbf x}\Sch_w(c) =\bigl(\g_w+B^{\mathsf T}\mathcal P(r(w))\bigr) \cap\ZZ^{\mathcal I_N}. \]
Both $\mathcal P(r(w))$ and its affine image under this map are integral
and have the integer decomposition property.
\end{corollary}
\begin{proof}
Apply Theorem~\ref{thm:schubert-characterization},
Theorem~\ref{F:thm:gap}, and Corollary~\ref{F:cor:Laurent}.
\end{proof}

\subsection{The signature map}
Let $\RP(w)$ be the reduced pipe dreams (or RC-graphs) of $w$~\cite{BergeronBilley,BJS}. Read a pipe
dream by increasing row and from right to left within each row. Its
compatible pair is $a=(a_1,\ldots,a_L)$,
$b=(b_1,\ldots,b_L)$, where a crossing in row $b_t$ and
column $c_t$ gives $a_t=b_t+c_t-1$. The word $a$ is
reduced, $b$ is weakly increasing, $b_t\le a_t$, and the
letters of $a$ strictly decrease among crossings in the same row.

For a word $v$, pair each occurrence of $i$ with the most
recent unpaired $i+1$ to its left. The remaining subword is
$i^s(i+1)^t$. Let $\sigma_i$ replace these unpaired letters,
in their original positions, by $i^t(i+1)^s$, leaving paired
letters fixed. Define
\begin{equation}\label{solpipe:eq:psi}
 \psi(\varnothing)=\varnothing,\qquad
 \psi(a_1,\ldots,a_t)=
 \sigma_{a_t}\bigl(\psi(a_1,\ldots,a_{t-1})\bigr)(a_t+1).
\end{equation}
Place $b_t$ in column $\psi(a)_t$, and sort each column
increasingly. Denote the resulting filling by $C(D)$. Column
$j$ has length $e_j=j-r_j$; index its entries by
$r_j\le k<j$, writing $C_j(k)$.

A \emph{cover tableau} has a strictly increasing column $j$ of
length $e_j$, with entries in $[1,j-1]$, and satisfies
\begin{equation}\label{solpipe:eq:cover-tableau}
 C_j(k+1)\le C_i(k)
 \qquad(i\to j,\quad r_j-1\le k\le i-1).
\end{equation}
Write $\CT(r)$ for the set of these tableaux. Empty columns are
omitted. The sets $J_{a,s}$ of \eqref{F:eq:J} are ordered by
increasing slot. For a pipe dream define
\begin{equation}\label{solpipe:eq:map}
 \Psi_w(D)_{a,s}
    =\#\{j\in J_{a,s}:C_j(a+s)>s\}.
\end{equation}

\begin{theorem}[Pipe-dream formula]\label{solpipe:thm:main}
For $w\in\mathcal C_N$, the maps above give bijections
\[ \RP(w)\longrightarrow\CT(r(w)) \longrightarrow\mathcal P(r(w))\cap\ZZ^{\mu_N}. \]
Writing $\op{row}_h(D)$ for the number of crossings in row $h$,
their composite satisfies
\begin{equation}\label{solpipe:eq:main-weights}
 F_w(\mathbf y)=\sum_{D\in\RP(w)}\mathbf y^{\Psi_w(D)},\qquad
 \sum_a\Psi_w(D)_{a,s}
   =\sum_{h\le s}\bigl(\mu_h(w)-\op{row}_h(D)\bigr).
\end{equation}
It identifies $D$ with the unique quotient of dimension
$\Psi_w(D)$ of $\Ext^1_{\Lambda_Q}(T,X_{r(w)})$.
Equivalently, the corresponding quiver Grassmannian has a unique
geometric point.
For a gap word $r(w)$ outside $\mathcal C_N$, the same construction
gives an injection $\RP(w)\hookrightarrow\CT(r(w))$.
\end{theorem}

\subsection{Reversibility and the cover inequalities}
\begin{lemma}\label{solpipe:prop:injection}
For every permutation, the signature construction gives an injection
from its reduced pipe dreams to strict column fillings of lengths
$e_j$, with entries in $[1,j-1]$. It preserves the row weight.
\end{lemma}
\begin{proof}
The involution $\sigma_i$ exchanges the multiplicities of $i$
and $i+1$. If there are at least as many $i$'s as $i+1$'s,
only some $i$'s increase.

At an ascent $u_i<u_{i+1}$, the inversion-end counts at positions
$i,i+1$, say $p,q$, satisfy $p\ge q$, and right
multiplication by $s_i$ changes them to $q,p+1$. This is
exactly the content change in \eqref{solpipe:eq:psi}. Induction
therefore gives column lengths $e_j$. The same induction shows
that an output letter never decreases from its initial value
$a_t+1$, so $b_t<\psi(a)_t$.
The reflection $\sigma_i$ preserves the descent set: an adjacent
pair $i+1,i$ is paired, and the other comparisons cannot be reversed.
For the last position, the preceding output letter is initially
$a_{L-1}+1$. Its comparison with the new letter $i+1$ is the
comparison of $a_{L-1}$ with $i$; when $a_{L-1}=i-1$,
a possible change from $i$ to $i+1$ still creates no descent.
Since a reduced word has no consecutive equal letters, this covers all
cases. Induction shows that $a$ and $\psi(a)$ have the same descents.
Thus crossings in the same row have distinct, decreasing column labels, and
the columns are strict.

To reconstruct the pipe dream on the image of the signature map,
order tableau entries by increasing value, breaking ties by decreasing
column index. This recovers $b$ and $v=\psi(a)$. Its last letter gives
$a_L=v_L-1$. Delete that last letter and apply
$\sigma_{a_L}$ to the remaining word. Repeating recovers $a$,
and hence every crossing $(b_t,a_t-b_t+1)$.
\end{proof}

\begin{lemma}\label{solpipe:thm:landing}
Let $r=r(w)$ satisfy the gap condition, let $a$ be a reduced word
for $w$, and let $i\to j$ be a cover. Every prefix of $\psi(a)$
satisfies
\begin{equation}\label{solpipe:eq:prefix}
 \#i-\#j\le r_j-r_i-1.
\end{equation}
Consequently $C(D)$ satisfies \eqref{solpipe:eq:cover-tableau}.
\end{lemma}
\begin{proof}
Every nonempty intermediate slot $k\in(i,j)$ has $r_k<r_i$.
An intermediate rank between $r_i$ and $r_j$ would factor the
cover; a larger rank would, by the gap condition, force such an
intermediate rank later. In the standardized prefix of length $i$,
the values $r_j,\ldots,i$ therefore remain above every nonempty
insertion through slot $j$. Their $i-r_j+1$ positions are
common inversion predecessors of all those slots.

The signature construction is a reversed and complemented form of
plactification. Reiner--Shimozono's peelability theorem, followed by
Greene's theorem, turns these common inversion predecessors into
$i-r_j+1$ disjoint decreasing subsequences in $\psi(a)$,
each with a $j$ before an $i$. We prove this implication,
including the convention change, in
Appendix~\ref{app:pipe-landing}. At most
$e_i-(i-r_j+1)=r_j-r_i-1$ occurrences of $i$ are unpaired.
This proves \eqref{solpipe:eq:prefix}. Entries with row label at most
$h$ form a prefix, so
\[ |C_i\cap[1,h]|-|C_j\cap[1,h]|\le r_j-r_i-1. \]
For strict sorted columns, these inequalities are equivalent to
\eqref{solpipe:eq:cover-tableau}.
\end{proof}

\subsection{The quotient corresponding to a tableau}
Fix a gap word $r=r(w)$. In this subsection, $M=X_r$, with the
graded spaces and normalized cover maps of
Theorem~\ref{thm:gap-cover-representation}. In particular, every quotient
below is a quotient of the Ext-module
$E=\Ext^1_{\Lambda_Q}(T,X_r)$.

\begin{proposition}\label{solpipe:prop:geometry}
Cover tableaux are in bijection with the integral points of
$\mathcal P(r)$. The map is \eqref{solpipe:eq:map}, and its
inverse is
\begin{equation}\label{solpipe:eq:tableau-inverse}
 C_j(k)=k-r_j+1+
 \#\{(a,s):a+s=k,\ j\text{ is among the first }q_{a,s}
                         \text{ members of }J_{a,s}\}.
\end{equation}
\end{proposition}
\begin{proof}
For a cover tableau define, inside $M_{a+s}$,
\[ L_{a,s}=\Span\{v_{j,a+s}:C_j(a+s)\le s\}. \]
Column strictness gives $C_j(k)\ge k-r_j+1$, hence
$L_{a,s}\subseteq\ker y^s$. The column and cover inequalities
imply
\[ yL_{a,s}\subseteq L_{a,s-1},\qquad xL_{a,s}\subseteq L_{a+1,s},\qquad L_{a,s}\subseteq L_{a-1,s+1}. \]
A strict column of length $j-r_j$, bounded by $j-1$, also
has $C_j(k)\le k$. Thus the boundary is $L_{0,s}=M_s$,
and iteration gives $x^aM_s\subseteq L_{a,s}$. It follows that
$L_{a,s}/x^aM_s$ is the kernel of a quotient of
$E_{a,s}=\ker y^s/x^aM_s$. Its quotient dimension is
\eqref{solpipe:eq:map}.

Conversely, the kernel at each coefficient space is the unique final
segment of its central Jordan chain. Its preimage in $M_{a+s}$
is spanned by interval basis vectors. The identity maps make membership
of $v_{j,k}$ an increasing threshold in $s$. Denote that
threshold by $C_j(k)$. The $y$-maps make the columns strict,
and the $x$-maps give the cover inequalities. The least threshold
is $k-r_j+1$; its increments are exactly the initial quotient
segments. This proves \eqref{solpipe:eq:tableau-inverse} and the
bijection.
\end{proof}

\begin{proof}[Proof of Theorem~\ref{solpipe:thm:main}]
The preceding lemmas give an injection into cover tableaux on every
gap word, and the proposition identifies all cover tableaux with
quotient points. The inverse signature procedure is defined on the image of this
injection. For $w\in\mathcal C_N$, cardinality will show that
every cover tableau lies in that image.
The already established Schubert--generic equality
and classical specialization give
\[ |\RP(w)|=\Sch_w(1,\ldots,1)=F_w(\mathbf1) =|\mathcal P(r(w))\cap\ZZ^{\mu_N}|. \]
The finite injection is consequently surjective. On the homogeneous
locus, the inverse signature procedure therefore recovers a reduced
compatible word from any cover tableau. 

The minimal columns $(1,\ldots,e_j)$ have content $\mu(w)$.
An entry in grade $k$ contributes to the deficit below $s$
exactly when its minimal value is at most $s$ and its actual
value exceeds $s$. It then contributes once to $q_{k-s,s}$.
Summing gives the weight identity in
\eqref{solpipe:eq:main-weights}. The lattice-point formula of Part~I
proves the $F$-polynomial identity.
\end{proof}

\begin{corollary}\label{solpipe:cor:product}
The signature bijection induces a canonical weight-preserving bijection
\[ \RP(w)\longleftrightarrow \prod_{A\in\pi_0(\mathcal F_w)}\RP(v_A), \qquad \Psi_w(D)=\sum_A\Psi_{v_A}(D_A). \]
It commutes with adjoining terminal fixed maxima.
\end{corollary}
\begin{proof}
The column and cover inequalities separate over components. Restrict
a tableau to each component and use the inverse signature construction;
conversely, unite the component tableaux. Each nonzero $J_{a,s}$
is a single cover path, so belongs to one component. The quotient
coordinates therefore add. A terminal fixed maximum adds an empty
column and leaves the construction unchanged.
\end{proof}

\begin{example}\label{ex:merging-cover-pipes}
We continue Example~\ref{ex:merging-cover-tree}. The record criterion
gives $25413\in\mathcal C_5$. Here $C_4=(1,2,3)$ is fixed, and
all other nonempty columns and quotient coordinates are as follows:
\[
\begin{array}{c|c|c|c}
 C_3&C_5&(q_{1,1},q_{2,1},q_{2,2})&\text{ordinary weight}\\\hline
 (1)&(1,2)&(0,0,0)&x_1^3x_2^2x_3\\
 (1)&(1,3)&(0,0,1)&x_1^3x_2x_3^2\\
 (2)&(1,2)&(1,0,0)&x_1^2x_2^3x_3\\
 (2)&(1,3)&(1,0,1)&x_1^2x_2^2x_3^2\\
 (2)&(2,3)&(1,1,1)&x_1x_2^3x_3^2
\end{array}
\]
For the last tableau the reverse signature procedure gives
\[ a=(1,4,3,2,4,3),\qquad b=(1,2,2,2,3,3). \]
The corresponding crossings are
$(1,1),(2,3),(2,2),(2,1),(3,2),(3,1)$.
Thus the five quotient points from Part~I are exactly the five reduced
pipe dreams of this homogeneous universal Schubert polynomial.
\end{example}

Distinct coefficient monomials may have the same ordinary Schubert
weight; see Example~\ref{ex:two-merge-pipe-dreams} in
Appendix~\ref{app:diagram-enumeration}.

\section{Enumeration of the homogeneous locus}
\label{sec:record-enumeration}

Let $c_N=|\mathcal C_N|$, with $c_0=1$, and set
$C(z)=\sum_{N\ge0}c_Nz^N$. The universal Schubert polynomials indexed by $S_N$ are linearly
independent. Hence Theorem~\ref{thm:schubert-characterization}
identifies $c_N$ with both the number of torus-homogeneous universal
Schubert polynomials and the number that are extended cluster
monomials.

\begin{theorem}\label{solenum:thm:count}
Let $\omega=\omega(z)\in z\QQ[[z]]$ be determined by
\[ \omega=z(1+\omega)^2. \]
Define
\[
\begin{aligned}
 A_0(\omega)&=1+3\omega+2\omega^2+2\omega^3,\\
 B_0(\omega)&=1+\omega-2\omega^2-\omega^3-\omega^4,\\
 D_0(\omega)&=1+3\omega+\omega^2-3\omega^3-8\omega^4-9\omega^5-7\omega^6-3\omega^7-\omega^8.
\end{aligned}
\]
Then
\begin{equation}\label{solenum:eq:ordinary}
 C(z)=\frac{(1+\omega)^2A_0(\omega)B_0(\omega)}{(1+2\omega)D_0(\omega)}.
\end{equation}
In particular, $C(z)$ is algebraic of degree two over $\QQ(z)$.
\end{theorem}
\begin{proof}
We remove terminal fixed maxima and classify the permitted next ranks
using the terminal record test. The permitted ranks consist of rank
one, a distinguished rank immediately above the last nonempty rank,
at most one further isolated rank, and a terminal interval. Their
updates are determined by six integer parameters. The recurrence and
its proof are given in Appendix~\ref{app:enumeration-state}.

Two types of insertion states suffice to derive the generating function.
Their functional equations have two linear kernels. Eliminating the
first leaves the quadratic kernel $1-X+zX^2$; its power-series
root is $X=1+\omega$. Substituting this root determines the total
series and gives \eqref{solenum:eq:ordinary}. We carry out the
unweighted elimination in
Appendix~\ref{app:enumeration-kernels}. The other root replaces $\omega$ by $1/\omega$ and changes
the rational function in \eqref{solenum:eq:ordinary}. Thus the degree
of this algebraic function is two.
\end{proof}

The first values are
\[
\begin{array}{c|rrrrrrrr}
 N&0&1&2&3&4&5&6&7\\\hline
 c_N&1&1&2&6&21&80&317&1285
\end{array}
\]
\[
\begin{array}{c|rrrrrr}
 N&8&9&10&11&12&13\\\hline
 c_N&5277&21847&90928&379831&1590829&6675826.
\end{array}
\]

The terminal-insertion recurrence is given in
Proposition~\ref{solenum:prop:state}. It also admits an augmentation
that records the maximum-insertion threshold $b(w)$ from
Section~\ref{sec:intrinsic-record-locus}. If
\[ B_n(t)=\sum_{w\in\mathcal C_n}t^{b(w)}, \]
then deletion of the maximum gives
\[ c_{n+1}=(n+1)B_n(1)-B_n'(1). \]
Indeed, a permutation of size $n$ admits a new maximum in exactly
$n+1-b(w)$ positions. The augmented recurrence is proved in
Appendix~\ref{app:threshold-recurrence}; the threshold alone does not
determine its successors.

\subsection{Growth}
Let $\omega_0$ be the positive root of $D_0(\omega)$, and set
\[
 \rho=\frac{\omega_0}{(1+\omega_0)^2},\qquad
 K=-\frac{(1-\omega_0)(1+\omega_0)A_0(\omega_0)B_0(\omega_0)}
            {\omega_0(1+2\omega_0)D_0'(\omega_0)}.
\]
The signs of the coefficients of $D_0$, read in descending degree,
have one change, so this positive root is unique.

\begin{theorem}\label{thm:record-growth}
For every $\rho<R<1/4$,
\begin{equation}\label{eq:record-growth}
 c_N=K\rho^{-N}+O(R^{-N}).
\end{equation}
Numerically,
\[ \rho=0.23623356874885\ldots,\qquad \rho^{-1}=4.23309864595553\ldots,\qquad K=0.04507479992549\ldots. \]
\end{theorem}
The asymptotic formula follows by subtracting the simple pole at
$\rho$ from $C(z)$; the root count and residue calculation are given
in Appendix~\ref{app:enumeration-poles}.

\appendix
\begingroup
\small
\setlength{\abovedisplayskip}{7pt plus 2pt minus 2pt}
\setlength{\belowdisplayskip}{7pt plus 2pt minus 2pt}
\setlength{\abovedisplayshortskip}{3pt plus 1pt minus 1pt}
\setlength{\belowdisplayshortskip}{4pt plus 1pt minus 1pt}
\setlength{\jot}{2pt}
\setlist{topsep=3pt, partopsep=0pt, itemsep=1pt, parsep=0pt}
\section{PBW comparisons on the right-regular set}
\label{app:classical-transitions}
\subsection{Restriction to the right-regular set}

Let $\kappa_{\mathfrak m}$ be the PBW basis and $f_{\mathfrak n}$
the semicanonical basis functions. For a fixed dimension vector
$\mathbf d$, identify the PBW basis with the orbit-indicator functions
on the equioriented representation space $E_{\mathbf d}$, embedded in
the preprojective variety by setting the reverse arrows to zero. If
$M_{\mathfrak m}$ represents its orbit, then
\[ \left.f_{\mathfrak n}\right|_{E_{\mathbf d}} =\sum_{\mathfrak m}f_{\mathfrak n}(M_{\mathfrak m},0)\kappa_{\mathfrak m}. \]
We use this zero-section embedding whenever evaluating at an
equioriented module. As before, $X_w$ denotes a general preprojective
point.
Dualizing gives
\begin{equation}\label{eq:PBW-semicanonical-dual}
 \kappa_{\mathfrak m}^{*}
 =\sum_{\mathfrak n}
 {f_{\mathfrak n}(M_{\mathfrak m},0)}\rho_{Z_{\mathfrak n}}.
\end{equation}
Its support is contained in the orbit-closure order and the diagonal
coefficient is one \cite{YinZhang}. For the dual canonical basis one has
\begin{equation}\label{eq:PBW-canonical-dual}
 \kappa_{\mathfrak m}^{*}
 =\sum_{\mathfrak n}
 IC_{\mathfrak m,\mathfrak n}(1)b^*_{\mathfrak n},
\end{equation}
where $IC_{\mathfrak m,\mathfrak n}(q)$ is the local intersection
cohomology polynomial of $\overline{\cO_{\mathfrak n}}$ along
$\cO_{\mathfrak m}$. For linear quivers these are type-$A$
Kazhdan--Lusztig polynomials \cite{Henderson}.

For $1\leq a\leq j\leq N$, set
\[ Z_{a,j}:=Z(c)_{a,j}=c_{j-a}(j-1). \]
The dual PBW element attached to \eqref{eq:schubert-multisegment} is
Fulton's code monomial:
\[ \kappa_{\mathfrak m_w}^{*} =\prod_{j=2}^{N}Z_{r_j(w),j}=\Code_w. \]
Let $m_N^{\rm fus}(w,z)$ denote the multiplicity of the stable MV cycle
indexed by $\mathfrak m_z$ in the iterated Beilinson--Drinfeld fusion of
the flag-minor cycles occurring in $\Code_w$.

The following proposition collects the PBW--semicanonical,
PBW--canonical, and PBW--MV transition formulas in the present
right-regular indexing.

\begin{proposition}\label{prop:classical-right-regular-transitions}
For each right-regular multisegment, the three dual-PBW expansions
are supported on right-regular multisegments.
Explicitly,
\begin{align}
 \Code_w&=\sum_{z\preceq_{\rm Z}w}
 D_N^{\rm gen}(w,z)\Gen_z,
 &D_N^{\rm gen}(w,z)&={f_{\mathfrak m_z}(M_w,0)},
 \label{eq:Dgen}\\
 \Code_w&=\sum_{z\preceq_{\rm Z}w}
 D_N^{\rm can}(w,z)\Can_z,
 &D_N^{\rm can}(w,z)&=IC_{\mathfrak m_w,\mathfrak m_z}(1),
 \label{eq:Dcan}\\
 \Code_w&=\sum_{z\preceq_{\rm Z}w}
 D_N^{\rm MV}(w,z)\MV_z,
 &D_N^{\rm MV}(w,z)&=[\MV_z]\Code_w=m_N^{\rm fus}(w,z).
 \label{eq:DMV}
\end{align}
The three matrices are integral and unitriangular. Moreover,
\[ \cM_N =\bigoplus_{w\in S_N}\ZZ\Gen_w =\bigoplus_{w\in S_N}\ZZ\Can_w =\bigoplus_{w\in S_N}\ZZ\MV_w, \]
\[ D_N^{\rm can}(w,z)>0 \quad\Longleftrightarrow\quad z\preceq_{\rm Z}w, \qquad D_N^{\rm MV}(w,z)\in\NN. \]
\end{proposition}

\begin{proof}
Apply \eqref{eq:PBW-semicanonical-dual} to
$\mathfrak m=\mathfrak m_w$. A nonzero term satisfies
$\cO_{\mathfrak m_w}\subseteq\overline{\cO_{\mathfrak n}}$.
Lemma~\ref{lem:right-regular-generization} shows that $\mathfrak n$ is
right-regular, hence $\mathfrak n=\mathfrak m_z$ for a unique $z\in S_N$.
This gives \eqref{eq:Dgen}; unitriangularity follows from \cite{YinZhang}.
The same argument applied to \eqref{eq:PBW-canonical-dual} gives
\eqref{eq:Dcan}. For an orbit in the closure, the relevant intersection-cohomology
transition polynomial has constant term one; for an orbit outside the
closure, it vanishes.

For crystal elements $b,b'$, write
\[ b'\preceq_{\rm pol}b \quad\Longleftrightarrow\quad \op{Pol}(b')\subseteq\op{Pol}(b). \]
The transition from the dual semicanonical basis to the MV basis is
unitriangular with respect to inclusion of MV polytopes
\cite[Proposition~3.11 and Theorem~4.8]{BaumannMV}.
With the present indexing convention, an off-diagonal term has a smaller MV
polytope for $\preceq_{\rm pol}$
\cite[Proposition~3.9]{BaumannMV}. In the equioriented type-$A$
orientation, this polytope inclusion implies that the corresponding
orbit is a generization of the original equioriented orbit
\cite[Proposition~4.1]{Baumann}. By
Lemma~\ref{lem:right-regular-generization}, every datum in the transition
is therefore right-regular. Restricting the unitriangular transition to these finitely many
labels then proves that the MV elements form a basis of $\cM_N$.

Each factor $Z_{a,j}$ of $\Code_w$ is a Levi flag minor and hence an MV
basis element \cite[Remark~2.10]{BKKMV}. Products of MV basis elements
expand with nonnegative fusion multiplicities \cite{BKKMV}. This gives
\eqref{eq:DMV}; the support restriction and diagonal statement follow from
the two unitriangular transitions through the semicanonical basis.
\end{proof}

\section{Finite-seed minors and conormal representatives}
\label{app:finite-strict-minors}

We give the finite-seed comparison used in
Theorem~\ref{prop:strict-minor-cluster}, with polynomial frozen coefficients.

\begin{lemma}\label{lem:cluster-monomial-freezing}
Let $J$ be a subset of the mutable vertices of a finite-rank
skew-symmetrizable seed $\Sigma$. If a cluster monomial has no negative
Laurent exponent at any vertex of $J$, it is a cluster monomial after
freezing $J$, allowing nonnegative powers of the newly frozen variables.
The same holds for a locally finite infinite-rank seed if the monomial is
obtained by finitely many mutations and only finitely many vertices are
left mutable.
\end{lemma}

\begin{proof}
Write the monomial as $\prod_b f_b^{m_b}$ in distinct cluster variables.
For $f_b\ne x_j$, the negative of the minimum Laurent exponent at
$x_j$ is nonnegative, and is zero precisely when $f_b$ and $x_j$
are compatible \cite[Theorem~6.3]{CaoLiEnough}.
Additivity forces all these numbers to vanish if $x_j$ is not a factor;
if it is a factor, compatibility already holds. Thus all factors are
compatible with every $x_j$, $j\in J$, and extend to a common cluster
reachable without mutating $J$, by
\cite[Theorems~7.4 and~6.2]{CaoLiEnough}.
In the locally finite case, first choose a finite set containing every
vertex mutated in a sequence producing the monomial, all vertices to
remain mutable and their initial neighbors, and all vertices occurring
in the exchange relations along that sequence.
Freeze the complement and apply the finite-rank result. No coefficients
are specialized.
\end{proof}

Set $n=N-1$. For a type-$A_n$ height function $h$, set
\[
 \mathbf W_{r,p}^{(a)}=[L(Y_{r,p}Y_{r,p+2}\cdots Y_{r,p+2a-2})],
 \qquad z^h_{a,r}=\mathbf W_{r,h_r-2a+2}^{(a)},\qquad
 \mathbf W_{r,p}^{(0)}=1.
\]
Here $p\equiv r\pmod2$. The Kirillov--Reshetikhin seed has arrows
$(r,p)\to(r,p+2)$ and $(r,p)\to(r\pm1,p-1)$, when both endpoints
satisfy $p\le h_r$; see \cite[Section~2.4]{CasbiLi}.

\begin{lemma}\label{lem:finite-HL-seed-comparison}
The standard seed $\Sigma_N^U$, including its frozen variables, occurs
in a seed obtained by finitely many mutations from the usual
Hernandez--Leclerc seed for $\mathcal C^-$. The correspondence is
\[ x_{a,r}\longmapsto z^\xi_{a,r},\qquad \xi_r=-r,\qquad a+r\le N. \]
After freezing $a+r=N$, no vertex with $a+r<N$ is adjacent to the
exterior of this finite region.
\end{lemma}

\begin{proof}
For $\xi_r=-r$, the root correspondence is
$[a,a+r-1]\leftrightarrow Y_{r,2-r-2a}$: the root at $(r,\xi_r)$
is $[1,r]$, and positive Coxeter translates give the remaining roots.
The Hernandez--Leclerc isomorphism sends dual PBW elements to standard
module classes and dual canonical elements to simple classes
\cite[Theorem~6.1]{HLQuantumGrothendieck}. Since $x_{a,r}$ has datum
$[1,r]+[2,r+1]+\cdots+[a,a+r-1]$, the flag-minor/Kirillov--Reshetikhin
identification \cite[Section~2.4]{CasbiLi} gives its asserted image.
Substituting $p=2-r-2a$ gives the arrows
\[ (a,r)\to(a-1,r),\qquad (a,r)\to(a,r+1),\qquad (a,r)\to(a+1,r-1), \]
which identify the triangular seed and its separating boundary.
These are identities of module classes, not specializations of
highest-weight monomials.

To obtain the stated finite seed by mutation, start with the
bipartite height function
$\eta_r=0$ for even $r$, $-1$ for odd $r$.
One reaches $\xi$ by finitely many source operations
$h\mapsto h-2\mathbf e_r$, with $h_r=h_s+1$ at every neighbor.
Indeed, from any $h_r>\xi_r$, following a neighbor of larger height
stays above $\xi$ and terminates at an eligible source.
At such a source, mutate successively at
$(1,r),(2,r),\ldots,(L,r)$. The exchange relation at vertex $(a,r)$
is the T-system \cite[Equation~(2.8)]{CasbiLi}:
\[ z^h_{a,r}z^{h-2\mathbf e_r}_{a,r} =z^h_{a+1,r}z^{h-2\mathbf e_r}_{a-1,r} +\prod_{s\sim r}z^h_{a,s}. \]
The factor with first index zero is one. Inductively, mutation cancels
the arrows to the preceding vertices in neighboring columns and gives
the same exchange relation at $a+1$. After these mutations, the cluster
variables and exchange-matrix rows agree with those of the new height seed
at every vertex whose first coordinate is at most $L-1$; any difference
involves the truncation boundary. If $T$ source operations are needed,
begin with a bound $K>N+2T$ on the first coordinate. At each operation,
perform the successive mutations through one less than the bound for the
region of agreement, then reduce that bound by two. The final region
contains the finite triangle and all its neighbors. This gives the
required finite mutation sequence.
\end{proof}
The height orientation here is auxiliary and does not change the fixed
quiver orientation used for preprojective components.

\subsection{The representative supported on cover pairs}
\label{app:record-rigidity}
In~\cite[Sections~3.1--3.3]{LapidGreedy}, Lapid uses the opposite orientation $Q^{\rm op}$. Let
$M=\bigoplus_iV_i$ and $M^\vee$ be its vector-space dual, viewed as
a representation of $Q^{\rm op}$. Duality gives
\[ \op{Aut}_Q(M)\longrightarrow\op{Aut}_{Q^{\rm op}}(M^\vee), \qquad g\longmapsto(g^{-1})^\vee, \]
and an equivariant identification of the dual extension spaces. Its
differential is $h\mapsto-h^\vee$.
All the bases and extension coordinates used below are transported
through these duality identifications. In particular, a column $(i,j)$ now corresponds to a
map $V_i\to V_j$, whereas the row indexed by the extension pair
$(i,j)$ corresponds to $\Ext_Q^1(V_j,V_i)^*$. With these bases, the differential is the matrix
\eqref{eq:record-commutator-matrix} introduced in
Section~\ref{sec:full-generic-F}. In particular, its rows are indexed
by extension pairs and its columns by Hom pairs.
Using the ordinary dual Hom bases instead would multiply all columns
by $-1$, and would not change the rank of any submatrix. Duality also induces an
anti-equivalence for the preprojective algebras, preserving self-extension
spaces and the corresponding conormal components. Thus Lapid's rigidity
criterion applies to $Z_{\mathfrak m_w}$ for the orientation fixed in
this paper.
The matrix represents the differential of the
$\operatorname{Aut}_Q(M)$-action on the conormal fiber
$\Ext_Q^1(M,M)^*$. Surjectivity gives an open orbit in this affine
space. Its saturation under the base-change group is dense in the
corresponding conormal component; this is how the representative
constructed below is shown to be generic.

\subsection{Insertion of covers}\label{app:cover-insertion}
We complete the proof of the cover-insertion assertion in
Lemma~\ref{cor:record-insertion}. Work in the standardized parent
$u\in S_{N-1}$, set $h=r_N$, and use $i\vdash j$ for the Hom relation of Part~I.

For the record description, a new extension pair is $(q,N)$ with
$r_q<h\le q$. By \eqref{eq:record-compatible-pivot}, it is a cover
exactly when no $k>q$ has $r_q\le r_k<h$. Write $x=u_q$,
where $u=\tau_{\bar\alpha}$. If $x<h$ and some $x<a\le h$
has $p_a>q$, then $r_q<r_{p_a}\le h$ and
$r_{p_a}<p_a$: the first $q$ positions contain a value greater
than $h$, because $q\ge h$ and $a$ has not yet occurred. If
$r_{p_a}<h$, the cover criterion is contradicted directly. If
$r_{p_a}=h$, \eqref{eq:gap-exact-predecessor} applied to
$(p_a,N)$ gives a later rank $h-1$, with the same contradiction.
Thus $q$ is an $R1$ record and internality gives $r_q=x$.
Conversely, at an internal $R1$ record all values
$1,\ldots,h$ occur by position $q$, so every later rank exceeds $h$ and
$(q,N)$ is a cover.

Suppose $x\ge h$. Necessarily $x<N-1$, since the largest value
has rank equal to its position. Let $z$ maximize $p_z$ on
$[h-1,x+1]$. If $p_z>q$ and $z<x$, this is an $R1$ record
when $z=h-1$, or an $R2$ record with label $z+1$ otherwise.
Internality gives $r_{p_z}=h-1$, contradicting the cover criterion.
If $z=x+1$ and $p_z>q$, then
$r_q<r_{p_z}<p_z$: a value greater than $x+1$ already
occurs before $q$, since $r_q<q$ and $x+1$ occurs later.
If $r_{p_z}<h$ this again contradicts the criterion; otherwise
\eqref{eq:gap-exact-predecessor} supplies a later rank $h-1$, with
the same consequence. Hence $p_z=q$, giving the internal $R2$
record with label $x+1$. Conversely, at such a record
$r_q=h-1<q$, and every later value is either at most $h-2$ or at
least $x+2$. In the latter case the $h-2$ smaller values before
$q$, together with $x,x+1$, give rank greater than $h$.
Thus no later rank equals $h-1$, and $(q,N)$ is a cover.

An internal $R1$ record is the unique rightmost position among
$p_1,\ldots,p_h$. If $y<y'$ were two internal $R2$ records,
the record condition for $y'$ implies $p_{y-1}<p_{y'-1}$.
The $h-2$ smaller values before $p_{y-1}$, together with $y-1$,
then contradict $r_{p_{y'-1}}=h-1$. This proves Lemma~\ref{cor:record-insertion}.

\subsection{Paths for intervals}\label{app:interval-paths}
Set $h(i)=\binom{i+1}{3}$, with $h(1)=0$. We first
construct a seed containing the simple module $S_i$. In the
triangle of rank $i+1$, read the reduced word
\[ (1,\ldots,i)(1,\ldots,i-1)\cdots(1). \]
For each occurrence of $r$, let $b$ be the number of its later
occurrences and mutate successively at $(h,r)$ for
$1\le h\le b$. Concatenating these sequences gives the triangular
maximal green sequence of \cite[Theorem~4.1]{SeedSW}, of length
$\sum_{r=1}^i\binom{i-r+1}{2}=\binom{i+1}{3}$.
The resulting seed has principal $G$-matrix a negative permutation
matrix and represents a shift of the initial cluster-tilting object
in the stable category.

We check the simple-module identification in our leftward convention.
The summands $T_{1,i-1}$ and $T_{i-1,1}$ are the thin
forward and reverse intervals on $[1,i-1]$. The projective-injective summands in this seed $I_1,I_i$, at $x_{1,i}$ and
$x_{i,1}$, fit into
\[ 0\to T_{1,i-1}\to I_1\to S_i\to0, \qquad 0\to S_i\to I_i\to T_{i-1,1}\to0. \]
The maps are identities on the common thin support. Thus both shifts
of the initial object contain $S_i$. If the maximal green word
is $(k_1,\ldots,k_t)$, with final columns
$g_j=-\mathbf e_{\sigma(j)}$, the word
$(\sigma(k_t),\ldots,\sigma(k_1))$ reaches the inverse shift.
Indeed, apply the inverse suspension to the reversed path and relabel
its initial summands. Either word has length $h(i)$.
The case $i=1$ is already in the initial seed.

The smaller triangle embeds in every ambient rank $N\ge i+1$.
Only its interior is mutated. Entries joining this interior to rows
outside the smaller triangle are initially zero and remain zero;
its boundary is not mutated. The exchange relations, and hence the
simple-module identification, are unchanged in the ambient ice seed.

\begin{lemma}\label{solseed:lem:interval}
The thin forward interval $[a,b]$, for
$1\le a\le b\le N-1$, is contained in a seed reached in at most
\[ \ell_{\mathrm{int}}(a,b)=2\left(\binom{b+1}{4}-\binom{a+1}{4}\right) +\binom{b+1}{3} \]
mutations.
\end{lemma}
\begin{proof}
For $a=b$, use the mutation path for the simple module described
above. Otherwise, consider
\[ 0\longrightarrow S_a\longrightarrow[a,b] \longrightarrow[a+1,b]\longrightarrow0. \]
The end terms have disjoint supports, adjacent along one edge of the
underlying quiver. Their mixed Hom
spaces vanish and their symmetric Cartan pairing is $-1$; hence
 their preprojective extension space has dimension one. Directly, its
only coordinate is the joining arrow, and neither a relation nor a
change of splitting eliminates it. Since $\End(S_a)=\CC$, the
nonzero map $[a+1,b][-1]\to S_a$ is a minimal left
$\op{add}S_a$-approximation. Thus $\mathcal P_{S_a}$ of a
seed containing $[a+1,b]$ contains $[a,b]$. The prescribed-path
construction in \eqref{solseed:lem:path} gives
\[ \ell[a,b]\le2h(a)+\ell[a+1,b],\qquad \ell[b,b]\le h(b). \]
Summing proves the stated formula. The character of this interval is
$z_{a,b+1}$, with full vector
$\mathbf e_{a,b-a+1}-\mathbf e_{a-1,b-a+1}$, as in the multisegment correspondence
of Part~I.
\end{proof}

\subsection{Determining the projected path}
\label{app:effective-completions}
We now describe how to compute the path projections used in
Section~\ref{sec:compatibility-seeds}. Fix a seed $K$ containing a reachable rigid object
$R$. Let $J$ be the mutable positions occupied by its basic
nonprojective summands. In this paragraph, $G_S^K$ denotes the
principal index matrix of a seed $S$ relative to $K$.
The completion $V=\mathcal P_R(S)$ is characterized by
\begin{equation}\label{solseed:eq:projection-test}
 R\subseteq V,\qquad
 (G_V^K)^{-1}G_S^K\text{ has nonnegative rows outside }J.
\end{equation}
This is the completion condition for a $G$-system;
existence and uniqueness follow from
\cite[Proposition~6.1.1 and Theorem~6.1.2]{SeedCao}.
The summands at $J$ are never mutated. After one mutation of
$S$, its possible images are the preceding image and its mutations
outside $J$. Exactly one satisfies
\eqref{solseed:eq:projection-test}. Thus a path is projected by at
most one mutation at each step, using at most one more than the
number of mutable vertices as candidates.

For $\mathcal Q_R$, apply the same criterion to the opposite
index system. Its initial principal exchange matrix is the negative
of the one used for $\mathcal P_R$. Both systems are obtained by
principal-coefficient mutation from $K$. In the skew-symmetric
case, tropical duality gives
\[ (G_V^K)^{-1}=(C_V^K)^{\mathsf T}; \]
see \cite{SeedNZ}. We invert the unimodular $G$-matrix, not the
principal exchange matrix, which may be singular.

The following calculation includes frozen coordinates in the
$G$-vectors. Write $\widetilde B=B^{\mathsf T}$, so that
all vertices index rows and mutable vertices index columns. At the
reference seed, initialize the principal matrix $C=I$ and the full
matrix $G=I$.
Let $c_j,g_j$ denote their columns. For mutation at $k$, let
$\varepsilon$ be the common sign of the nonzero entries of
$c_k$. Then
\[
\begin{aligned}
 c'_k&=-c_k,&
 c'_j&=c_j+[\varepsilon\widetilde b_{kj}]_+c_k\quad(j\ne k),\\
 g'_k&=-g_k+\sum_j[-\varepsilon\widetilde b_{jk}]_+g_j,&
 g'_j&=g_j\quad(j\ne k).
\end{aligned}
\]
Frozen columns remain the corresponding coordinate vectors.
In the first line, $k,j$ are mutable, so that
$\widetilde b_{kj}$ is an entry of the principal part of
$\widetilde B$. The second line uses all rows of the same matrix.
Both formulas use the column convention $\widetilde B=B^{\mathsf T}$,
not the mutable-row convention for $B$ elsewhere in the paper.
Sign-coherence makes the choice of $\varepsilon$ unambiguous.

To remove any dependence on opposite-category conventions, one may
perform the two principal calculations, with anchor matrices
$B_K$ and $-B_K$, and choose a path whose final seed contains the required vector.
Each candidate path is
mutated in the original ice seed; the opposite matrix is used only
for the index test. The approximation arguments in
Section~\ref{sec:compatibility-seeds} prove that at least one of the
two paths succeeds. Both satisfy the same length bound.

More explicitly, split the target $g$ into mutable and frozen
coordinates. If $G_{\mathrm{fr},\mathrm{mut}}$ is the frozen
part of the mutable columns of the final matrix, the unique
coordinates of $g$ are
\[
 a_{\mathrm{mut}}=C^{\mathsf T}g_{\mathrm{mut}},\qquad
 a_{\mathrm{fr}}=g_{\mathrm{fr}}-
       G_{\mathrm{fr},\mathrm{mut}}a_{\mathrm{mut}}.
\]
We test nonnegativity of both vectors, including the frozen
coordinates; equality of dimension vectors alone is not sufficient.
For a compatible family, the anchor can be changed to the seed already
containing the preceding members: reverse its known path to obtain the
argument matrices, and then apply the same edge test. The commuting
completion identity gives the final seed described in
Corollary~\ref{solmatch:cor:families}.

\section{Coefficient varieties and Newton polytopes}
\label{app:F-coefficients}
\subsection{Maps between rectangles}
\label{app:rectangle-maps}
We complete the endomorphism-algebra calculation in
Lemma~\ref{F:lem:kernel}. A homogeneous map
$T_{a,r}\to T_{b,s}$ is determined by a linear combination of
\begin{equation}\label{F:eq:Hom-basis}
 x^iy^{s+i-r},\qquad
 \max(0,b-a,r-s)\leq i\leq\min(b-1,r-1).
\end{equation}
Indeed, the target vector must have degree $r$, must be annihilated by
$x^a$ and $y^r$, and must be a nonzero monomial in the target
rectangle. These are precisely the inequalities above. Each map
factors by first quotienting to $T_{b-i,r}$, then applying $i$ corner
maps to $T_{b,r-i}$, and finally multiplying by a power of $y$ to map
into $T_{b,s}$. The inequalities ensure that every intermediate
rectangle belongs to the same finite triangle. Hence the three
adjacent maps generate every Hom space, and stability under
\eqref{F:eq:arrows} is stability under the full endomorphism algebra.

\subsection{The exchange-matrix sign and the integral lattice}
\label{app:F-normalization}
Let $H_{u,v}=\dim\Hom_{\Lambda_Q}(T_u,T_v)$, with both indices
running through $\mathcal I_N$. For a fixed target $(b,s)$, write
$h(a,r)=H_{(a,r),(b,s)}$. Formula~\eqref{F:eq:Hom-basis} gives
\[ h(a,r)-h(a-1,r) =\one_{\{a\leq b\}}\one_{\{b-a+1\leq r\leq b+s-a\}}. \]
Taking successive differences yields
\begin{equation}\label{F:eq:Hom-difference}
\begin{split}
 &h(a-1,r+1)+h(a+1,r)+h(a,r-1)\\
 &\qquad-h(a+1,r-1)-h(a-1,r)-h(a,r+1)
   =\delta_{a,b}\delta_{r,s}.
\end{split}
\end{equation}
We set $h(0,r)=h(a,0)=0$. For a mutable $(a,r)$, every other neighbor
in this formula belongs to the finite triangle. If
$\iota_{\mu_N}:\ZZ^{\mu_N}\hookrightarrow\ZZ^{\mathcal I_N}$ is
the coordinate inclusion, the identity is
\[ H^{\mathsf T}(-B^{\mathsf T})=\iota_{\mu_N}. \]
Let $R_{\mu_N}$ be coordinate restriction to the mutable indices.
Then $-R_{\mu_N}H^{\mathsf T}$ is an integer left inverse of
$B^{\mathsf T}$.

In the GLS character formula, the mutable columns of $H^{-\mathsf T}$
are the exponents associated with the submodule dimensions of
$E=\Ext^1(T,M)$~\cite[Equation~(1.3) and Theorem~4]{FKernelGLS}.
They are $-B^{\mathsf T}$ by \eqref{F:eq:Hom-difference}. Writing
$F_E(\mathbf y)=\sum_e\chi(\Gr_e(E))\mathbf y^e$, we therefore have
\[
 \varphi_M=\mathbf x^h F_E(\mathbf x^{-B^{\mathsf T}})
 =\mathbf x^{h-B^{\mathsf T}b}
   \bigl(\mathbf y^bF_E(\mathbf y^{-1})\bigr)
                         (\mathbf x^{B^{\mathsf T}}),
 \qquad b=\dim E.
\]
The last polynomial has constant term one and nonnegative exponents.
The integer left inverse implies that the cone
$B^{\mathsf T}\mathbb R_{\geq0}^{\mu_N}$ is pointed. Two finite
pointed expressions for the same function must have the same initial
exponent: each initial exponent belongs to the support of the other
expression, so their difference and its negative lie in this cone.
For a general point of $Z_{\mathfrak m_w}$, the character is
$\Gen_w$, whose initial exponent is $\g_w$. This proves the
normalization in Lemma~\ref{F:lem:normalization}.

\subsection{Equations in Grassmannian charts}
\label{app:F-effective}
We derive equations from the kernel--image description.
Work over $\QQ(\lambda_{ij})$ and choose bases for the spaces in
\eqref{F:eq:kernel}. In a Grassmannian chart, let the columns of $U_v$
span a prospective submodule, with a fixed identity minor, and let
$R_v$ be the projection onto the pivot rows, so that $R_vU_v=I$.
For an arrow matrix $A_{uv}$, stability is
expressed by
\[ (I-U_vR_v)A_{uv}U_u=0. \]
Clearing the nonzero denominators gives polynomial equations in each
of finitely many charts. These equations include the incidence and dimension conditions
defining the varieties $\mathcal V_q(r)$.

These equations have two consequences. First, constructibility gives
one dense open parameter set on which all coefficient Euler
characteristics are constant. Second, they make each coefficient
computable by finite algebraic procedures: a Gr\"obner basis tests
nonemptiness, and semialgebraic triangulation with parameter
stratification computes the generic Euler characteristic
\cite{FKernelBPR}. We use disjoint Schubert cells, or additivity, when
combining charts.

For the Newton polytope, only nonemptiness is needed. There are finitely
many possible quotient dimensions, and Lemma~\ref{F:lem:vertices} gives
\[ \operatorname{Newt}(F_w) =\operatorname{conv}\{q:\mathcal V_q(r(w))\ne\varnothing\}. \]
The collapsed polytope is its image under $\mathsf R$. Possible
cancellation at nonvertices therefore does not affect either convex
hull. Under the gap condition, \eqref{F:eq:P} gives explicit defining
inequalities.

\subsection{Central cyclicity under the gap condition}
\label{app:F-cyclic}
We work with the cover representative of
Proposition~\ref{thm:gap-cover-minor}, so cover parameters are one and all
other parameters are zero.

\begin{lemma}\label{F:lem:barrier}
Let $p<q$ index nonempty summands and suppose that
$r_p\geq r_q=b>1$. There is a cover path of length $b-1$ ending at
$q$, all of whose other indices are strictly greater than $p$.
\end{lemma}
\begin{proof}
The gap condition supplies an index strictly between $p$ and $q$ of
rank at least $b-1$. The corresponding summand is nonempty: an empty
one would have rank equal to its index, which is greater than $r_p$.
Choose the last nonempty index $i\in(p,q)$ with rank at least $b-1$.
If $r_i\geq b$, applying the gap condition to $(i,q)$ gives a later
such index, a contradiction. Hence $r_i=b-1$, and all intervening
nonempty ranks after $i$ are smaller than $b-1$.

Since $i>p>r_p\geq b$, the pair $(i,q)$ is an extension pair. An
intermediary in either forbidden factorization would have rank at
least $b-1$, so no such intermediary exists. Thus $i\to q$ is a
cover. For $b>2$, apply induction to $(p,i)$ and add this edge; for
$b=2$, the single edge suffices.
\end{proof}

\begin{lemma}\label{F:lem:survivors}
The space $E_{a,r}$ has a basis indexed by $J_{a,r}$ in
\eqref{F:eq:J}. In the bases indexed by $J_{a,r}$, the three maps in
\eqref{F:eq:arrows} act, respectively, by the parent map and by
two identity-on-indices maps. Any basis vector whose target index is
absent is sent to zero.
\end{lemma}
\begin{proof}
Set $k=a+r$. The numerator of \eqref{F:eq:kernel} has the coordinate
basis $v_{j,k}$ with $a<r_j\leq k<j$. A power of $x$ follows the
unique parent path, or is zero. A basis vector with target index $j$
lies in $x^aM_r$ exactly when $j$ has a descendant $i>r$ at distance
$a$.

Indeed, along such a path the starts increase strictly, so
$r_i\leq r_j-a\leq r$. The indices also increase strictly. At step
$t$, the start is at most $r+t$, and the preceding endpoint is greater
than $r+t-1$. Thus each required reverse-arrow map is nonzero and has
coefficient one. Conversely, a nonzero contribution to $x^a$ gives
such a path. The criterion is exactly $\delta_a(j)>r$. Removing
these image coordinates gives \eqref{F:eq:J}. The formulas for the
three maps follow from \eqref{F:eq:kernel} and the same support
conditions.
\end{proof}

Together with the preceding two lemmas, Lemma~\ref{F:lem:chain}
of Section~\ref{sec:full-generic-F} establishes the cyclicity used in
Theorem~\ref{F:thm:gap}. The resulting uniqueness of quotient
subspaces is set-theoretic: the underlying reduced quiver Grassmannians
are points when nonempty, but the original schemes need not be reduced.

\section{The rank-ten comparison of geometric bases}
\label{app:rank-ten-geometric-comparison}
\subsection{Lusztig data and polytope inequalities}
\label{app:S10-computation}
We prove Proposition~\ref{prop:S10-deformation} by comparing the bases
in type $A_5$ and transporting the comparison coefficients to type
$A_9$. We first specify the Lusztig-data and MV-polytope conventions
used in the calculation.

For type $A_{n-1}$, use the reduced word
\[ \mathbf i_0=(1)(2,1)\cdots(n-1,n-2,\ldots,1). \]
The corresponding positive roots are the intervals $[a,b]$, ordered
first by $b$ and then by $a$. A commutation exchanges the two
corresponding Lusztig coordinates. A braid move replaces three
coordinates by
\begin{equation}\label{eq:tropical-braid-comparison}
 (a,b,c)\longmapsto
 (b+c-\min(a,c),\ \min(a,c),\ a+b-\min(a,c)).
\end{equation}
These changes determine the data in any reduced word;
see~\cite[Proposition~5.2]{KamnitzerMVPolytopes}.

We normalize the initial vertex of the MV polytope to
$-\operatorname{wt}(\mathfrak m)$, where the weight here is the positive
sum of its interval roots. Along a reduced word
$\mathbf i=(i_1,\ldots,i_L)$, set
\[
 w_k=s_{i_1}\cdots s_{i_k},\qquad
 \beta_k=w_{k-1}\alpha_{i_k},\qquad
 \nu_k=-\operatorname{wt}(\mathfrak m)+\sum_{j=1}^k n_j\beta_j.
\]
For $I=w_k\{1,\ldots,i_k\}$, the Berenstein--Zelevinsky value
is $M_I(\mathfrak m)=\sum_{a\in I}(\nu_k)_a$. The initial chambers
are evaluated at $\nu_0$. Every nonempty proper subset occurs as a
chamber in some reduced word; the values do not depend on that word. Thus
\begin{equation}\label{eq:MV-containment-BZ}
 \operatorname{Pol}(\mathfrak m')\subseteq
 \operatorname{Pol}(\mathfrak m)
 \quad\Longleftrightarrow\quad
 M_I(\mathfrak m')\geq M_I(\mathfrak m)
 \quad\text{for every nonempty proper }I.
\end{equation}
Equations~\eqref{eq:tropical-braid-comparison} and
\eqref{eq:MV-containment-BZ} specify all entries in the tables below.
For any chosen $I$, one may use a reduced word with a prefix sending
an initial interval to $I$, extend it to $w_0$, and apply braid moves
from $\mathbf i_0$.

Leclerc's type-$A_5$ example~\cite{Leclerc} involves the data
\[
\begin{split}
 X_0&=2([1,2]+[2,4]+[3,3]+[4,5]),\\
 Y_0&=[1,2]+[1,4]+[2,3]+[2,5]+[3,4]+[4,5].
\end{split}
\]
Their root degree is $(2,4,4,4,2)$. The $1138$ multisegments of this
root degree are the nonnegative integer solutions of
\[ \sum_{a\leq i\leq b}n_{ab}=d_i\qquad(1\leq i\leq5). \]
They can be listed by choosing successively the multiplicities of the
intervals beginning at $1$, then at $2$, and so on, subtracting the
chosen interval dimensions at each stage. Table~\ref{tab:A5-exclusion}
applies the indicated inequalities in order. In its last column we
count the previously unexcluded multisegments for which
$M_I(\mathfrak m)<M_I(X_0)$. The counts sum to $1136$; the two
remaining data are $X_0$ and $Y_0$.
\begin{table}[htbp]
	\centering \label{tab:A5-exclusion}
	\caption{Successive exclusions in root degree $(2,4,4,4,2)$.}	
	\begingroup \small
	\setlength{\tabcolsep}{3.5pt}
	\renewcommand{\arraystretch}{.96}
	\begin{tabular}{@{}lrrr@{}}
		\toprule
		$I$ & $M_I(X_0)$ & $M_I(Y_0)$ & Excluded \\
		\midrule
		$\{2\}$ & -2 & -2 & 612\\
		$\{3\}$ & -2 & -1 & 112\\
		$\{4\}$ & -2 & -1 & 93\\
		$\{1,2,4\}$ & -4 & -4 & 167\\
		$\{3,4\}$ & -2 & -2 & 38\\
		$\{5\}$ & 0 & 0 & 57\\
		$\{1,3,4,5\}$ & -2 & -2 & 19\\
		$\{1,3,4,6\}$ & -2 & -2 & 28\\
		$\{1,2,3,4,6\}$ & -2 & -2 & 5\\
		$\{1,2,5,6\}$ & -2 & -2 & 3\\
		$\{3,5,6\}$ & 0 & 0 & 2\\
		\bottomrule
	\end{tabular} \endgroup
\end{table}
Consequently any basis-comparison term with polytope contained in
$\operatorname{Pol}(X_0)$ has label $X_0$ or $Y_0$.
Let $\overline D$ be the equivariant-multiplicity map of
\cite[Section~1.4]{BKKMV}. The preceding calculation supplies the
support restriction needed to use \cite[Theorem~A.13]{BKKMV}.
Unitriangularity fixes the coefficient at
$X_0$ to be one, while that theorem gives
\[ \overline D(\Gen_{X_0}) =\overline D(\MV_{X_0})+2\overline D(\MV_{Y_0}), \qquad \overline D(\MV_{Y_0})\ne0. \]
Although $\overline D$ is not injective, the preceding support
calculation and the nonvanishing of $\overline D(\MV_{Y_0})$ determine
the remaining coefficient. Hence
$\Gen_{X_0}=\MV_{X_0}+2\MV_{Y_0}$.
The semicanonical--canonical comparison is
$\Gen_{X_0}=\Can_{X_0}+\Can_{Y_0}$
\cite[p.~196, equation~(2)]{GLSsemicanonical}. The table also shows
that $\operatorname{Pol}(Y_0)$ is a proper subset of
$\operatorname{Pol}(X_0)$. Thus no other label has its polytope
contained in $\operatorname{Pol}(Y_0)$, and unitriangularity gives
$\Can_{Y_0}=\MV_{Y_0}$. Therefore
\[ \Can_{X_0}=\MV_{X_0}+\MV_{Y_0}. \]

\subsection{The crystal transport}
We use the interval crystal of~\cite[Definition~3.2]{ClaxtonTingley},
with the right-root convention of Section~\ref{sec:root-crystal-actions}.
To apply $f_i$, order the intervals by increasing length, and among
equal lengths by decreasing left endpoint. Place a right parenthesis
above an interval ending at $i-1$ and a left parenthesis above an
interval ending at $i$. Cancel matched pairs. The number of remaining
left parentheses is $\epsilon_i^{\mathrm R}$. The operator $f_i$ extends the interval
under the rightmost remaining right parenthesis; if there is none, it
adds $[i,i]$.

Embed the two type-$A_5$ data in type $A_9$. Table~\ref{tab:crystal-transport}
applies the sequence
\[(6,5,6,7,8,7,6,5,4,3,4,9,8,7,6,5)\]
from left to right. It records the one interval changed in each datum;
all other intervals remain unchanged. Before each operation, both
$\epsilon_i^{\mathrm R}$ values are zero. The symbol $\varnothing$ in the source position means
that a simple interval is added. The resulting data are $X$ and $Y$ in
Proposition~\ref{prop:S10-deformation}.

\begin{table}[htbp]
	\centering
	\caption{Simultaneous crystal operations from $X_0$ and $Y_0$.}
	\label{tab:crystal-transport}
	\begingroup
	\small
	\setlength{\tabcolsep}{3.5pt}
	\renewcommand{\arraystretch}{.96}
	\begin{tabular}{@{}rrll@{}}
		\toprule
		Step & $i$ & Starting at $X_0$ & Starting at $Y_0$ \\
		\midrule
		1 & 6 & $[4,5]\longmapsto[4,6]$ & $[2,5]\longmapsto[2,6]$\\
		2 & 5 & $[2,4]\longmapsto[2,5]$ & $[1,4]\longmapsto[1,5]$\\
		3 & 6 & $[4,5]\longmapsto[4,6]$ & $[4,5]\longmapsto[4,6]$\\
		4 & 7 & $[4,6]\longmapsto[4,7]$ & $[2,6]\longmapsto[2,7]$\\
		5 & 8 & $[4,7]\longmapsto[4,8]$ & $[2,7]\longmapsto[2,8]$\\
		6 & 7 & $[4,6]\longmapsto[4,7]$ & $[4,6]\longmapsto[4,7]$\\
		7 & 6 & $[2,5]\longmapsto[2,6]$ & $[1,5]\longmapsto[1,6]$\\
		8 & 5 & $[2,4]\longmapsto[2,5]$ & $[3,4]\longmapsto[3,5]$\\
		9 & 4 & $[3,3]\longmapsto[3,4]$ & $[2,3]\longmapsto[2,4]$\\
		10 & 3 & $[1,2]\longmapsto[1,3]$ & $[1,2]\longmapsto[1,3]$\\
		11 & 4 & $[3,3]\longmapsto[3,4]$ & $\varnothing\longmapsto[4,4]$\\
		12 & 9 & $[4,8]\longmapsto[4,9]$ & $[2,8]\longmapsto[2,9]$\\
		13 & 8 & $[4,7]\longmapsto[4,8]$ & $[4,7]\longmapsto[4,8]$\\
		14 & 7 & $[2,6]\longmapsto[2,7]$ & $[1,6]\longmapsto[1,7]$\\
		15 & 6 & $[2,5]\longmapsto[2,6]$ & $[3,5]\longmapsto[3,6]$\\
		16 & 5 & $[3,4]\longmapsto[3,5]$ & $[2,4]\longmapsto[2,5]$\\
		\bottomrule
	\end{tabular}
	\endgroup
\end{table}

\begin{proof}[Proof of Proposition~\ref{prop:S10-deformation}]
The calculations above give the three comparison coefficients
$1,1,2$ for the pair $(X_0,Y_0)$. Table~\ref{tab:crystal-transport}
applies the same crystal word to both data, with both relevant
string lengths zero at every step, and yields the pair $(X,Y)$.
Baumann's coefficient-invariance result~\cite[Proposition~2.6 and
its proof]{Baumann} therefore gives
\eqref{eq:S10-basis-coefficients}. The inversion-end codes recover
the permutations $w$ and $v$ in the proposition.
\end{proof}

\FloatBarrier
\section{Schubert coefficients and the record criterion}
\label{app:coefficient-record-proofs}\label{app:homogeneous-technical}
\label{app:record-sufficiency}

We prove the coefficient statements used in
Section~\ref{sec:Schubert-cluster-problem}. The first calculation
detects external records; the multilinear construction proves
sufficiency of the record criterion independently of reachability. We then establish the two M\"obius identities and
the negative Laurent coefficient.

\subsection{Complementary codes and deletion records}
\label{app:record-obstructions}

\begin{proof}[Proof of Lemma~\ref{lem:complementary-code}]
Let
\[ K(a,u)=[x^a]\Sch_u(x), \]
where $a$ ranges over staircase exponents. Lemma~\ref{prop:Cauchy-transpose-Pieri}
says
\[ \Pi_N(v,w)=K(\theta(v),ww_0). \]
For homogeneous degree $\ell(u)$, the nil-Hecke pairing gives
$K^{-1}(u,a)=\partial_u x^a$
\cite{FominStanley,BjornerBrenti}. Inverting the preceding reindexed
transpose yields
\[ \Pi_N^{-1}(w,v)=\partial_{ww_0}x^{\theta(v)}. \]
This is the required formula.
\end{proof}

For a Lehmer code $a=(a_1,\ldots,a_N)$, let $\mathbf i(a)$ be the
concatenation, for $i=1,\ldots,N-1$, of the strings $(i+a_i-1,i+a_i-2,\ldots,i)$,
with the empty string when $a_i=0$. We use the usual left action of divided
differences. Thus, if $\mathbf i=(i_1,\ldots,i_\ell)$ is a reduced word,
then the operators in $\partial_{\mathbf i}$ act in the order
$i_\ell,\ldots,i_1$. The monomial expansion is governed by
\begin{equation}\label{app:eq:two-variable-box}
 \partial_i x^\nu=
 \begin{cases}
  \displaystyle\sum_{a'=b}^{a-1}
  x^{\nu+(a'-a)\eps_i+(a-a'-1)\eps_{i+1}},&a>b,\\[2mm]
  \displaystyle-\sum_{a'=a}^{b-1}
  x^{\nu+(a'-a)\eps_i+(a-a'-1)\eps_{i+1}},&a<b,\\[2mm]
  0,&a=b,
 \end{cases}
\end{equation}
where $a=\nu_i$ and $b=\nu_{i+1}$.

For $e\ge1$ set $T_e=\partial_{e-1}\cdots\partial_1$, with
$\partial_1$ acting first, and let $P^{+1}$ denote the shift of a polynomial
from $x_1,x_2,\ldots$ to $x_2,x_3,\ldots$. Write $\op{CT}$ for
constant term.

\begin{lemma}\label{lem:insertion-functional}
Let $1\le s\le h$, and let $P$ be homogeneous of degree $h-s$. Then
\begin{equation}\label{eq:insertion-functional}
 \op{CT}\,T_h\bigl(x_1^{s-1}P^{+1}\bigr)
 =\sum_{s=j_0<j_1<\cdots<j_m=h}
 (-1)^m
 \left[
 x_{j_0}^{j_1-j_0}\cdots
 x_{j_{m-1}}^{j_m-j_{m-1}}
 \right]P.
\end{equation}
For $s=h$ the sum consists of the empty chain and equals the constant
term of $P$.
\end{lemma}

\begin{proof}
It is enough to take $P=x^\kappa$. At each application of $\partial_i$, take the term independent
of $x_i$, since later operators do not
change that exponent. If the exponent of $x_i$ before this application
is positive, \eqref{app:eq:two-variable-box} forces $\kappa_i=0$ and
transfers the exponent, decreased by one, to $x_{i+1}$, with coefficient
one. If it is zero, a term occurs precisely when $\kappa_i>0$; its
coefficient is minus one and its exponent at $x_{i+1}$ is
$\kappa_i-1$. Variables beyond $x_h$ are untouched and must be absent.
Thus the nonzero entries of $\kappa$ begin at $s=j_0$ and partition
$[s,h)$ into intervals $[j_u,j_{u+1})$. Conversely, each such partition
determines exactly one sequence of terms contributing to the constant
coefficient. Its sign is $(-1)^m$, one minus sign for each interval.
\end{proof}

For $\alpha=(\bar\alpha,r)$ set $h=N-r$ and
$\bar y=\tau_{\bar\alpha}w_0^{(N-1)}$. The Lehmer reduced word gives
\begin{equation}\label{eq:peel-operators}
 \partial_{\tau_\alpha w_0^{(N)}}
 =T_h\circ\partial_{\bar y}^{+1}.
\end{equation}

Call a monomial $c$-initial if it contains no variable $x_i$ with $i<c$,
and write $\op{Lead}_c(P)$ for the sum of the $c$-initial terms
of $P$.

\begin{lemma}[Initial terms at a record]\label{lem:leading-support}
Let $u\in S_M$ and write $p_z=u^{-1}(z)$.
\begin{enumerate}[(i)]
 \item Suppose $x<h$ and $p_x>p_z$ for every $x<z\le h$. Set
 $q=p_x$, $t=h-x$, and
 \[ \mu=\theta(u)+t\eps_{M+1-q}. \]
 Then
 \[ \op{Lead}_x \bigl(\partial_{uw_0^{(M)}}x^\mu\bigr)=x_x^t. \]
 \item Suppose $h<y\le M$ and
 $p_{y-1}>p_z$ for every
 $z\in[h-1,y]\setminus\{y-1\}$. Set $q=p_{y-1}$ and
 \[ \mu=\theta(u)+\eps_{M+1-q}. \]
 Then
 \[ \op{Lead}_{h-1} \bigl(\partial_{uw_0^{(M)}}x^\mu\bigr) =x_{h-1}+x_h+\cdots+x_{y-1}. \]
\end{enumerate}
\end{lemma}

\begin{proof}
We argue by induction on $M-q$. If $q=M$, the truncation
$\bar u=\op{std}(u_1,\ldots,u_{M-1})\in S_{M-1}$ satisfies
$\partial_{\bar u w_0^{(M-1)}}x^{\theta(\bar u)}=1$, and
\eqref{eq:peel-operators} reduces the assertion to one application of
$T_{u_M}$. Formula~\eqref{app:eq:two-variable-box} gives respectively
$x_x^t$ and $x_{h-1}+\cdots+x_{y-1}$.

Assume $q<M$, set $\bar h=u_M$, and let $\bar u=\op{std}(u_1,\ldots,u_{M-1})$ be the truncation of $u$.
The record condition excludes $\bar h$ from the comparison interval.
Hence either $\bar h<c$, where $c=x$ in~\textup{(i)} and $c=h-1$ in
~\textup{(ii)}, or $\bar h$ lies above that interval. The truncation $\bar u$ then
carries the same record after shifting all relevant values down by one in
the first case, and carries the unchanged record in the second.

Apply \eqref{eq:peel-operators}. To obtain a $c$-initial monomial,
\eqref{app:eq:two-variable-box} requires the exponents at indices below
$\min(c,\bar h)$ in the rank-$(M-1)$ expression to vanish. If $\bar h<c$,
the contributing terms are obtained from its $(c-1)$-leading terms by
shifting the variable indices up by one. The assertion follows by
induction.

Suppose now that $\bar h$ lies above the comparison interval. Only the
$c$-initial terms supplied by induction can contribute.
In case~\textup{(i)}, the induction hypothesis leaves $x_x^t$.
At the application of
$\partial_x$, the exponent of $x_x$ in a contributing monomial
must be $t$; each remaining application decreases the nonzero exponent
by one, with coefficient one. Thus the only $c$-initial term is $x_x^t$.

In case~\textup{(ii)}, the terms supplied by induction are $x_j$ for
$h-1\le j\le y-1$. A direct use of
\eqref{app:eq:two-variable-box} gives
\[
 [x_k]T_{\bar h}\bigl(x_1^{\bar h-1}x_{j+1}\bigr)
 =\begin{cases}
 0,&k<j,\\
 1,&k=j,\\
 0,&k>j,
 \end{cases}
 \qquad h-1\le j,k\le y-1.
\]
Indeed, for $k<j$ the application of $\partial_j$ is the first to
exclude a contribution to $x_k$. For $k>j$, the same application
would require a positive exponent where $x_k$ has exponent zero.
For $k=j$, each required term occurs with coefficient one. The stated
initial part follows.
\end{proof}

\begin{proof}[Proof of Proposition~\ref{prop:record-obstructions}]
If the lengths in \eqref{eq:Q-rank-reduction} differ, both sides vanish
by grading. Otherwise, \eqref{eq:peel-operators} and
Lemma~\ref{lem:insertion-functional}, with $s=h$, give
\eqref{eq:Q-rank-reduction}.

Let $\gamma$ be a profile obtained from a record, set $u=\tau_{\bar\alpha}$,
and set $M=N-1$. For an $R1$ record, set $q=p_x$ and $\delta=h-x$;
for an $R2$ record, set $q=p_{y-1}$ and $\delta=1$. The
result lies in $\op{Prof}_N$: in the $R1$ case, the record inequality gives
$\alpha_{q-1}\ge\delta$ and
$\alpha_{N-1}+\delta=N-x\le N-1$; in the $R2$ case it gives
$\alpha_{q-1}\ge1$ and
$\alpha_{N-1}+1=N-h+1\le N-1$.

Deleting the last entry of $\tau_\gamma$ and standardizing gives
complementary code $\theta(u)+\delta\eps_{M+1-q}$. Only the initial
part computed in Lemma~\ref{lem:leading-support} contributes to
\eqref{eq:insertion-functional}. The sole contributing monomial is
$x_x^{\delta}$ for $R1$ and $x_{h-1}$ for $R2$. Each corresponds to a
one-part chain and hence has sign $-1$. This proves
\eqref{eq:record-obstruction-coefficient}.

Only the rank letters in positions $q$ and $N$ change,
replacing them by $\rho_q(\alpha)+\delta$ and $h-\delta$. Their multiset is
unchanged if and only if $\rho_q(\alpha)=h-\delta$. By the dimension-vector
formula, this holds precisely when $\gamma$ and $\alpha$ have
the same dimension vector. It is also the internality condition; when
it holds, the two rank letters are exchanged.
\end{proof}

\begin{lemma}\label{lem:sealed-gap}
Let $\alpha=(\bar\alpha,r)\in\op{Prof}_N$, set $h=N-r$, and
suppose that every record of $\alpha$ is internal. There is no $q<N$ with
\[ 1<h\le\rho_q(\alpha)<q \quad \text{ and } \quad \rho_t(\alpha)\notin[h-1,\rho_q(\alpha)] \qquad(q<t<N). \]
\end{lemma}

\begin{proof}
Write $u=\tau_{\bar\alpha}$ and $p_z=u^{-1}(z)$. Call
$z\in[h,N-1]$ dominant if
$p_z>p_{z'}$ for every $z'\in[h-1,z)$, and choose $m$ minimal such that
every value in $(m,N-1]$ is dominant. If $z>m$, all smaller values in
$[h-1,z)$ lie to the left of $p_z$, all larger values lie to its right,
and values below $h-1$ are smaller than $z$. Hence
$\rho_{p_z}=p_z$.

Suppose that $q$ violates the assertion. Since
$u_q\ge\rho_q\ge h$, the preceding observation gives $u_q\le m$.
Choose $z^*\in[h-1,m]$ with $p_{z^*}$ maximal. Then $z^*<m$ and
$p_{z^*}\ge q$. If $z^*=h-1$, this is an $R1$ record; if $z^*\ge h$, it
is the $R2$ record with $y=z^*+1$. In either case internality gives
$\rho_{p_{z^*}}=h-1$. If $p_{z^*}=q$, this contradicts
$\rho_q\ge h$; if $p_{z^*}>q$, it gives an index between $q$ and $N$ whose rank lies in the excluded
interval. This contradiction proves the lemma.
\end{proof}

\subsection{Multilinear Schubert polynomials}

Set $X_i=(x_1,\ldots,x_i)$ and
$T^{(r)}_{ij}=e_{i-j}(X_{i-1})$. For the empty alphabet only
$e_0=1$ is nonzero. For a matrix of independent indeterminates
$A=(A_{ij})_{1\le i,j\le r}$, define
\begin{equation}\label{eq:rs-matrix-polynomial}
 \mathscr P_\sigma(A)=
 \sum_{1\le j_i\le r}
 \left(\operatorname{CT}\partial_{\sigma w_0}
 x_1^{j_r-1}\cdots x_r^{j_1-1}\right)
 \prod_{i=1}^r A_{i,j_i},\qquad \sigma\in S_r.
\end{equation}
Here the divided differences have the usual right-descent convention.
Set $\mathscr P_\varnothing=1$. We call a polynomial \emph{row-multilinear} if it is homogeneous
of degree one in the variables of each row. Give $A_{ij}$ column weight
$\varepsilon_j$, and call a nonzero polynomial of one column weight
\emph{column-homogeneous}. For $0\le a\le r$, let $\pi_a$ set $A_{ij}=0$ for
$i\le a$ and $j>i$.

\begin{lemma}\label{lem:rs-completion}
The following statements hold.
\begin{enumerate}[(i)]
\item Evaluation at $T^{(r)}$ is injective on the row-multilinear
polynomials invariant under right upper-unitriangular Toeplitz
multiplication. The polynomial $\mathscr P_\sigma$ is the unique such
polynomial evaluating to $\Sch_\sigma(x)$.
\item An ascent $\sigma_i<\sigma_{i+1}$ makes
$\mathscr P_\sigma$ alternating in rows $i,i+1$.
If $\sigma=\alpha\ominus\beta$, with $p=|\alpha|$ and
$q=|\beta|$, then
\[ \mathscr P_\sigma(A)= \mathscr P_\alpha(A_{[p],[p]}) \mathscr P_\beta(A_{\{p+1,\ldots,p+q\},[q]}). \]
\item Suppose $W$ shuffles the low values $1,\ldots,r$, in order
$\sigma$, with increasing higher values. If its low positions are
$f_i+1$, then
\begin{equation}\label{eq:rs-transport}
 \Sch_W(c)=\mathscr P_\sigma
 \bigl(c_{f_i-j+1}(f_i)\bigr)_{1\le i,j\le r}.
\end{equation}
Distinct surviving column words give distinct standard-elementary
monomials. In particular, the right side is homogeneous precisely when
restricting $\mathscr P_\sigma$ by $j_i\le f_i+1$ gives a column-homogeneous polynomial.
\end{enumerate}
\end{lemma}

\begin{proof}
Pair matrix column words with monomials in
$\mathbb Q[z_1,\ldots,z_r]/(z_1^r,\ldots,z_r^r)$, pairing
$\prod_iA_{i,j_i}$ with $z^{\mathbf j-\mathbf1}$.
The transpose of the $k$th Toeplitz Lie operator is multiplication
by $p_k(z)$, for $1\le k<r$. Newton identities and
$\prod_i(t-z_i)=t^r$ modulo the elementary symmetric ideal give
\[ (z_1^r,\ldots,z_r^r,p_1,\ldots,p_{r-1})=(e_1,\ldots,e_r). \]
Finite logarithms identify Toeplitz invariance with the common Lie
kernels, so the invariant space is the dual coinvariant algebra. Its Artin basis
has exponents $0\le b_i<i$. At $T^{(r)}$, the corresponding matrix
words evaluate to the independent standard-elementary products
$\prod_i e_{i-j_i}(X_{i-1})$, whereas words with some $j_i>i$
vanish. This proves injectivity. The functional in
\eqref{eq:rs-matrix-polynomial} kills the positive symmetric ideal:
divided differences commute with symmetric factors and their positive
degree prevents a constant term. Grouping the finite Schubert Cauchy
identity used in Lemma~\ref{prop:Cauchy-transpose-Pieri} by the
second-alphabet variables proves
$\mathscr P_\sigma(T^{(r)})=\Sch_\sigma(x)$, proving (i).

For (ii), $y=\sigma w_0$ has a right descent at $r-i$, so
$\partial_y s_{r-i}=-\partial_y$. Formula
\eqref{eq:rs-matrix-polynomial} gives row alternation.
Reversing a skew sum gives a direct sum in the opposite order.
Its divided differences act on disjoint variable sets and their
constant-term functionals factor. Powers at least the size of a factor
vanish in its coinvariant algebra. Reversing back proves the product
formula, including the asserted initial column ranges.

For (iii), first place $M$ increasing entries, all larger than $r$,
before the entries of $\sigma$. Let
$E_M=x_1\cdots x_M$, and let $z_i=x_{M+i}$. The dominant case
$\sigma=w_0$, followed by divided differences in the low positions,
gives
\[ \Sch_{(r+1,\ldots,r+M,\sigma)}(x)=E_M^r\Sch_\sigma(z). \]
The elementary matrix with entries $e_{M+i-j}(X_{M+i-1})$ is
$E_MT^{(r)}(z)U$, where
\[ U=I+\sum_{k\ge1}e_k(x_1^{-1},\ldots,x_M^{-1})J^k, \]
and
$J$ is the upper shift matrix. Toeplitz invariance gives the same
evaluation. For $M=0$, take $E_M=1$ and $U=I$.
Move the larger entries to the right by adjacent descent swaps.
If the smaller entry in a swap has flag $f$, then
\[ \partial_f e_{f-j+1}(X_f)=e_{f-j}(X_{f-1}). \]
No other low has flag $f-1$; all other elementary factors are
symmetric in $x_f,x_{f+1}$. Therefore each swap lowers just this
one flag. The distinct flags make every surviving term standard-elementary,
so Fulton's injectivity proves the universal identity.
Each flag recovers its column, even when its degree is zero. A word has
torus weight $\sum_i\varepsilon_{j_i}-\sum_i\varepsilon_{f_i+1}$,
so the homogeneity assertion follows. Nonvanishing follows by ordinary
Schubert specialization.
\end{proof}

\subsection{Alternating tails and exterior multiplication}

Let $\mathcal V_{a,r}$ be the space of row-multilinear polynomials
supported on $j_i\le i$ for $i\le a$ and alternating in their
last $r-a$ rows. Its coefficient space is
\[ \left(\bigotimes_{i=1}^a\mathbb Q^i\right) \otimes\bigwedge^{r-a}\mathbb Q^r. \]
In particular it has the basis
\begin{equation}\label{eq:rs-tail-basis}
 M_{\mathbf j,J}=
 \left(\prod_{i=1}^a A_{i,j_i}\right)
 \det A_{\{a+1,\ldots,r\},J},
 \qquad 1\le j_i\le i,\quad J\subseteq[r],\quad |J|=r-a.
\end{equation}

\begin{lemma}[Alternating-tail stability]\label{lem:rs-tail-stability}
Evaluation at $T^{(r)}$ is injective on $\mathcal V_{a,r}$.
If $\sigma=\eta\oplus1$ and $\sigma$ is increasing after position
$a<r$, then $\pi_a\mathscr P_\sigma$ is obtained from the expansion
of $\pi_a\mathscr P_\eta$ in \eqref{eq:rs-tail-basis} by replacing
$J$ with $J\cup\{r\}$, without changing any coefficient.
It is column-homogeneous if and only if $\pi_a\mathscr P_\eta$ is column-homogeneous.
\end{lemma}

\begin{proof}
The dimension is $a!\binom r{r-a}=r!/(r-a)!$. Exactly that many
permutations have increasing last $r-a$ entries. Their restricted
matrix polynomials belong to $\mathcal V_{a,r}$ by
Lemma~\ref{lem:rs-completion}(ii) and evaluate to linearly independent
ordinary Schubert polynomials. This proves injectivity.

In $T^{(r)}$, column $r$ is zero above its last entry, which is
one. Expanding a minor with column set $J\cup\{r\}$ along this
column gives exactly the old minor in $T^{(r-1)}$, with positive
sign. The proposed expansion therefore evaluates to
$\Sch_\eta=\Sch_{\eta\oplus1}$. Injectivity proves the identity.
Each weight gains the same $\varepsilon_r$. This also covers an
empty old tail and the empty permutation.
\end{proof}

Let $P$ be a row-multilinear polynomial on $n\times n$ matrices,
where $n=m+L$, alternating in its last $L$ rows. For $t\ge0$, we define its
\emph{exterior insertion} on $N\times N$ matrices, where $N=n+t$.
Set $Q=\{m+1,\ldots,N\}$ and
\begin{equation}\label{eq:rs-exterior}
 \mathcal F_{L,t}(P)(A)=
 \sum_{\substack{I\subseteq Q\\|I|=t}}
 (-1)^{\sum_{i\in I}(i-m)-t(t+1)/2}
 \det A_{I,[t]}\,
 P\bigl(A_{[m]\,\Vert\,(Q\setminus I),[n]}\bigr).
\end{equation}
All subsets are listed increasingly, and $\Vert$ concatenates row
lists. Thus $\mathcal F_{L,t}$ is exterior multiplication by the
initial-column $t$-form, with no factorial normalization. It maps a
column weight $\mu$ to $\mu+\varepsilon_1+\cdots+\varepsilon_t$
or to zero. It commutes with $\pi_a$ when $a\le m$.

\begin{proposition}[Exterior insertion]\label{prop:rs-exterior}
Let $m\ge1$, $\eta\in S_m$, $1\le h\le m+1$, $L\ge1$, and
$0\le t\le h-1$. Suppose that the subword of $\eta$ on
$[h-1]$ ends in $h-t,\ldots,h-1$, with no condition when $t=0$.
Set $H=h+L-1$, $n=m+L$, $N=n+t$, and
$\iota_s(x)=x+L\mathbf1_{x\ge h-s}$. Define
\[
 \begin{aligned}
 \tau&=(\iota_0(\eta),h,\ldots,H),\\
 v&=(\iota_t(\eta),h-t,\ldots,H-t),\\
 \chi&=(v_1+t,\ldots,v_m+t,1,\ldots,t,v_{m+1}+t,\ldots,v_n+t).
 \end{aligned}
\]
Then $\mathscr P_\chi=\mathcal F_{L,t}(\mathscr P_\tau)$.
For $a\le m$, column homogeneity of $\pi_a\mathscr P_\tau$ implies column homogeneity of
$\pi_a\mathscr P_\chi$.
\end{proposition}

\begin{proof}
Both sides are row-multilinear Toeplitz invariants. Indeed, the initial
$t$-column minor is invariant, and the first $n$ transformed
columns use only the first $n$ old columns. The source polynomial is
alternating in its last $L$ rows; hence its exterior insertion is
alternating in all $L+t$ new rows. It suffices, by
Lemma~\ref{lem:rs-completion}, to compare evaluations at $T^{(N)}$.
Set $E=e_m(X_m)$. Elementary-alphabet addition writes the new rows as a
unit-lower-triangular transform of
\[ D_{p,j}=e_{m+p-j}(X_m),\qquad 1\le p\le L+t. \]
Alternation removes that transform. In the first $t$ columns of
$D$, rows $p>t$ vanish, and the first $t$ rows have determinant
$E^t$. Therefore only the first $t$ rows contribute the selected
minor in \eqref{eq:rs-exterior}, with positive sign. Applying
Lemma~\ref{lem:rs-completion}(iii) gives
\begin{equation}\label{eq:rs-exterior-evaluation}
 \mathcal F_{L,t}(\mathscr P_\tau)(T^{(N)})
 =E^t\Sch_W(X_m,0),\quad
 W=(\iota_0(\eta),n+1,\ldots,N,h,\ldots,H).
\end{equation}
For completeness, the unit-lower matrix just used has entries
$e_{p-q}(x_{m+1},\ldots,x_{m+p-1})$, $q\le p$.

Appending $t$ increasing minima gives, for any $v\in S_n$,
\[ \Sch_{(v+t,1,\ldots,t)}=(x_1\cdots x_n)^t\Sch_v. \]
Prove this first at $v=w_0$ by the dominant formula and then descend
within the first $n$ positions. In the present case $v$ has
increasing last $L$ entries, so $\Sch_v$ uses only $X_m$.
Move the minima left across that tail by its $Lt$ descent swaps.
On $(x_{m+1}\cdots x_{m+L})^t$, these are the divided differences
taking a rectangular dominant Schubert polynomial to one. They leave
the other factors unchanged. Consequently
\begin{equation}\label{eq:rs-child-evaluation}
 \Sch_\chi=E^t\Sch_v(X_m).
\end{equation}

The hypothesis on the order of the low values leaves a single
additional summand in each last-descent transition. For
$0\le s\le t$, set
\[ W_s=(\iota_s(\eta),n+1,\ldots,N-s, h-s,\ldots,H-s,N-s+1,\ldots,N). \]
For $s<t$, its last descent is at $r_s=m+t-s>m$; its value is
$N-s$, and the rightmost smaller value is $H-s$ at position
$q_s=r_s+L=N-s$. Set $p_s=W_st_{r_s,q_s}$.
The classical last-descent transition, in the convention of
Macdonald, \emph{Notes on Schubert Polynomials}, (4.16), is
\[ \Sch_{W_s}=x_{r_s}\Sch_{p_s} +\sum_{i<r_s:\,\ell(p_st_{i,r_s})=\ell(p_s)+1} \Sch_{p_st_{i,r_s}}. \]
Before $r_s$, the only values below $H-s$ are original lows
$x<h-s$. By the position condition, $k_s=h-s-1$ is after all smaller values.
It is an eligible cover partner and blocks every smaller candidate.
The unique child is $W_{s+1}$: the swap replaces $k_s$ by
$k_s+L$ in the old prefix and fixes $N-s$ at its own position.
Thus
\[ \Sch_{W_s}=x_{m+t-s}\Sch_{p_s}+\Sch_{W_{s+1}}. \]
Setting variables beyond $x_m$ to zero and using
$W_t=v\oplus\operatorname{id}_t$ proves
$\Sch_W(X_m,0)=\Sch_v(X_m)$. Equations
\eqref{eq:rs-exterior-evaluation}--\eqref{eq:rs-child-evaluation}
and invariant evaluation injectivity prove the matrix identity.

If $a\le m$, restriction commutes with exterior insertion, so its
output has the prescribed shifted weight or is zero. It is not zero:
$(\pi_a\mathscr P_\chi)(T^{(N)})=\Sch_\chi\ne0$.
This proves column homogeneity, including when some blocks are empty.
\end{proof}

\subsection{Decomposition at a marked maximum}

We separate the permutation reductions used for homogeneity from
the matrix-polynomial identities.
For $\sigma\in S_r$ and $0\le a\le r$, set
\[ T_a(\sigma)=(\sigma_1,\ldots,\sigma_a,r+1, \sigma_{a+1},\ldots,\sigma_r). \]
We call $(a,\sigma)$ \emph{admissible} when
$T_a(\sigma)\in\mathcal C_{r+1}$, and refer to the inserted maximum
as the marked entry. We include $T_0(\varnothing)=(1)$.
Throughout this subsection, insertion ranks are those of $\sigma$,
with the marked entry omitted. They are unchanged by its presence,
since it is larger than every entry of $\sigma$.

For a permutation $u$, write $p_u(x)$ for the position of $x$ and
\[ q_u(x)=1+\#\{y<x:p_u(y)<p_u(x)\}. \]
Appending rank $h$ selects the R1 values $x<h$ with
$p(x)>\max_{x<z\le h}p(z)$ and requires $q(x)=x$.
The R2 values are $h\le x<|u|$ with
$p(x)>\max\{p(h-1),\ldots,p(x-1),p(x+1)\}$,
and their required prefix rank is $h-1$.
R1 is absent at a terminal maximum, and R2 is absent at $h=1$.

\begin{lemma}\label{lem:rs-record-deletions}
Deleting the maximum preserves membership in $\mathcal C$.
Reordering an initial consecutive interval of values changes none of
the later record tests or their internality. In particular, if
$\sigma=\alpha\ominus\beta$ and $a<|\alpha|$, then
\[ (a,\sigma)\text{ admissible} \quad\Longrightarrow\quad (a,\alpha)\text{ and }(0,\beta)\text{ admissible}. \]
\end{lemma}
\begin{proof}
Every record that survives deletion of the maximum was already selected,
with the same comparisons and prefix rank. The deleted entry contributes
no earlier smaller value.

Consider a record test after an initial consecutive interval of values.
A selected record cannot be in that interval: the R1 comparison value
$h$, or the R2 comparison value $h-1$, lies outside and later, unless
the corresponding test is absent. Later selections see the whole
interval before them, so their tests depend on its set of values, not
its order. The internal tests within the interval are those of its
standardized permutation.

For the skew sum, $\alpha$ together with the marked maximum forms an
initial top interval. Its standardized prefix gives $T_a(\alpha)$.
Sort this interval increasingly and delete its current maxima in
decreasing order, keeping only its smallest value. The result is
$T_0(\beta)$.
\end{proof}

We need one explicit form of the terminal record test. Let $u^\circ$
be obtained by deleting all terminal fixed maxima, and let $u^-$
mean deletion of the last entry followed by standardization. For
position comparisons, extend $p_u(j)=j$ when $j>|u|$.

\begin{lemma}[Terminal record test]\label{lem:rs-terminal-test}
Appending rank one creates no defective record. For $h\ge2$, set
$b=h-1$ and $U=u^\circ$. All new records are internal if $|U|\le b$.
Otherwise let $s$ be the last value of $U$. The condition fails for
$s>b$ and holds for $s=b$.

When $s<b$, start with $(v,A,\ell)=(U^-,b-1,s)$ and apply the following
reductions, keeping $b$ fixed.
\begin{enumerate}[(i)]
\item Replace $v$ by $v^\circ$. If $|v|\le A$, the condition holds.
\item Otherwise let $z$ be its last value. If $z=b$, the condition
holds exactly when
\begin{equation}\label{eq:rs-order-test}
 \max_{\ell\le x\le A}p_v(x)<p_v(A+1)<\cdots<p_v(b).
\end{equation}
\item If $z<\ell$, replace $(v,A,\ell)$ by
$(v^-,A-1,\ell-1)$ and continue. In every other case the condition fails.
\end{enumerate}
This criterion does not assume that $u\in\mathcal C$. It deletes at
most $h-2$ entries that are not terminal maxima.
\end{lemma}
\begin{proof}
A terminal fixed maximum blocks every newly available selection, so
its removal preserves the tests. If the remaining last value exceeds
$b$, it is a defective R2 selection. At $b$, it blocks the other
selections and its R1 test is internal. If the last value is $s<b$,
R1 below $s$ is blocked and the selection at $s$ is internal.
Absence of defective R1 selections above $s$ is equivalent to
$\max_{s<x<h}p(x)<p(h)$.
After deleting $s$, this becomes
$\max_{\ell\le x\le A}p(x)<p(A+1)$.

Together with this inequality, the remaining requirement is that every
$x>A$ satisfying
\[ p(x)>\max\{p(A),\ldots,p(x-1),p(x+1)\} \]
has prefix rank $b$. Deleting a later smaller entry does not change
that prefix rank. Nor does removal of terminal fixed maxima affect the
requirement. A remaining permutation of size at most $A$ satisfies it
automatically. Otherwise a last value in $[\ell,A]$ violates the
position inequality, and a last value above $A$ other than $b$ is a
defective selection. At $b$, excluding selections $A<x<b$ successively
forces $p(A)<p(A+1)<\cdots<p(b)$. This is
\eqref{eq:rs-order-test}. A last value below $\ell$ can be removed;
standardization lowers $A$ and $\ell$ by one and leaves the target
prefix rank $b$ unchanged.

After $D$ deletions of nonmaximum entries, the parameters are
$A=b-D$ and $\ell=s-D+1$. A further deletion requires $\ell\ge2$,
so $D\le s\le h-2$. This proves termination as well as the criterion.
\end{proof}

For a rank word $\rho$, let $\sigma(\rho)$ denote its permutation.
We use \emph{skew-indecomposable} in its usual sense: the permutation
cannot be written as $\alpha\ominus\beta$ with both factors nonempty.
The next lemma gives the reduction needed in the homogeneity proof.

\begin{lemma}[Insertion decomposition]\label{lem:rs-normal-form}
Suppose that $(a,\sigma)$ is admissible and that $\sigma$ has no
proper skew decomposition. Assume that its maximum occurs after
position $a$ and that its entries after $a$ are not increasing. Then its rank word has the form
\begin{equation}\label{eq:rs-guarded-deletion}
 (\alpha,1,\ldots,t,h,\ldots,H),
 \qquad t>0,
\end{equation}
where $m=|\alpha|>a$, the last rank of $\alpha$ is $h-1$,
$t\le h-2$, and the restriction of $\sigma(\alpha)$ to $[h-1]$
ends in $h-t,\ldots,h-1$. Moreover,
$\big(a,\sigma(\alpha,h,\ldots,H)\big)$ is admissible. 
Thus deletion of $1,\ldots,t$ gives a shorter admissible pair at the
same marked position.
\end{lemma}
\begin{proof}
We analyze the terminal test for an accepted rank $h\ge2$ at a position
$i>a$ of $\sigma$. As long as an entry after the mark remains, the
global maximum is not terminal. After $D$ deletions the remaining size
exceeds the size bound in the test by $i-h+1\ge1$. Hence, before
reaching the marked entry, the test either ends at an earlier insertion
of rank $h-1$ or deletes every entry after the mark.

\emph{The test ends at rank $h-1$.}
Suppose it has deleted $t$ intervening ranks, forming a word $\beta$.
If $t=0$, the two ranks are consecutive. If $t>0$, every deleted rank
is smaller than $h-1$. For any rank $k\ge2$ within $\beta$, its own
accepted test must end at a preceding rank $k-1$ within $\beta$:
it cannot pass the larger rank $h-1$, reach the marked entry, or
satisfy a size bound first. Following these preceding ranks gives
$\beta_j\le j$. The deletion bound for the final test gives
$t\le\beta_t$, so $\beta_t=t$. Its preceding rank $t-1$ must
then be penultimate, and repetition proves $\beta=(1,2,\ldots,t)$.
The remaining position condition is precisely that the restriction of
the preceding permutation to $[h-1]$ ends in
$h-t,\ldots,h-1$; also $t\le h-2$.
Conversely, these conditions permit the insertions: rank one is
unrestricted, ranks $2,\ldots,t$ have their immediately preceding
ranks, and the final test at $h$ satisfies the position condition.
The choice $t=0$ is permitted as well.

\emph{The test reaches the marked entry.}
Write $t=i-a-1$ for the number of intervening entries. When $t>0$,
let their ranks be $\beta$, set $s=\beta_t$, and set $e=s-t$.
The terminal test gives
\[ t\le s\le h-2,\qquad 0\le e<a,\qquad \beta_j\le e+j,\qquad\beta_t=e+t. \]
Indeed, when $\beta_j$ is deleted, the lower threshold is
$s-(t-j)+1=e+j+1$. Let $v$ be the standardized prefix before the
mark. Its original values $e+1,\ldots,a$ remain above every inserted
$\beta$-value. Delete the marked maximum, then the current maxima
corresponding to the original values $a,\ldots,e+2$, keeping
$e+1$ as the new marked maximum. By
Lemma~\ref{lem:rs-record-deletions}, this gives an admissible pair at
a position $c\le e<a$, whose rank word is
\[ \delta=(\operatorname{ranks}(v|_{[e]}),\beta), \qquad |\delta|=e+t=\delta_{e+t}. \]

We claim that an admissible pair whose rank word ends in its full
length has increasing entries after the mark. We prove this by
induction on $(a,r)$ in lexicographic order. At $a=0$, reaching the
mark would require $r-1$ deletions, exceeding the bound $r-2$.
The test therefore ends at rank $r-1$, which must be penultimate;
induction gives the identity permutation. At a positive marked
position, a suffix of length at most one needs no proof. Ending at
rank $r-1$ again reduces the length. If the test reaches the mark,
the word $\delta$ above has a smaller marked position and ends in
its full length. Induction makes the appended $\beta$-values
increasing. The final rank is larger, so the entire suffix is
increasing. This proves the claim. Applying it to $\delta$ shows
that every accepted test which reaches the mark has increasing suffix;
when $t=0$, that suffix has only one entry.

\emph{Choosing the deletion.}
Return to the permutation in the statement. Its last rank is not one,
since a last minimum would give a proper skew decomposition. Its
terminal test cannot reach the marked entry, because the suffix is not
increasing. It therefore ends at a preceding rank $h-1$. The
permutation ending there is still skew-indecomposable: a proper final
skew factor would have size at least $h-1$, and all subsequent ranks
$1,\ldots,t,h$ would extend that same factor, preserving a proper
skew decomposition of $\sigma$.

Continuing backward in this way strictly decreases the final rank.
The process ends with an increasing suffix. Between that suffix and
the original word, every step inserts either no intervening ranks or
a block $1,\ldots,t$ as described above. At least one such block is
nonempty. Choose the last nonempty block. All steps following it have
no intervening ranks, so their final ranks are consecutive
$h,h+1,\ldots,H$. This gives \eqref{eq:rs-guarded-deletion}.
At the chosen step, choose $t=0$ and leave the subsequent
consecutive ranks unchanged. Their tests remain internal by the
first case. The resulting shorter pair is admissible, with the stated
position condition. This proves the required decomposition.
\end{proof}

\begin{example}\label{ex:marked-decomposition}
The permutation from Example~\ref{ex:merging-cover-tree} is
$25413=T_1(2413)$. The rank word of $2413$ is
$(1,2,1,3)$; deleting the intervening rank $1$ gives $(1,2,3)$.
Thus $\alpha=(1,2)$, $t=1$, and $h=H=3$ in
Lemma~\ref{lem:rs-normal-form}. Exterior insertion gives
\[ \mathscr P_{2413}(A) =A_{31}\det A_{\{1,2,4\},[3]} -A_{41}\det A_{\{1,2,3\},[3]}. \]
Every term has column weight $2\varepsilon_1+\varepsilon_2+
\varepsilon_3$. Restriction by $\pi_1$ and transport give
\[
 \Sch_{25413}(c)=c_3(3)\bigl(c_1(2)c_2(4)-c_3(4)\bigr)-c_4(4)\bigl(c_1(2)c_1(3)-c_2(3)\bigr).
\]
This calculation explains homogeneity directly; it does not use the
later comparison with geometric functions.
\end{example}

\subsection{Homogeneity from the decomposition}

The preceding lemmas provide four reductions for an admissible pair. A maximum before
the mark lowers the marked position; a skew decomposition separates two
smaller pairs; an increasing suffix permits deletion of its last
maximum; and Lemma~\ref{lem:rs-normal-form} removes an intervening block
of ranks. The polynomial operations corresponding to the last three
reductions are multiplication, alternating-tail extension, and exterior
insertion. Each preserves a single column weight.

\begin{theorem}[Homogeneity from the record criterion]\label{thm:rs-saturation}
For every $r\ge0$, $0\le a\le r$, and $\sigma\in S_r$,
\[ T_a(\sigma)\in\mathcal C_{r+1} \quad\Longrightarrow\quad \pi_a\mathscr P_\sigma\text{ is column-homogeneous}. \]
Consequently $\mathcal C_N\subseteq\mathcal H_N$ for every $N$.
\end{theorem}
\begin{proof}
We induct on $(a,r)$ in lexicographic order, with
$\mathscr P_\varnothing=1$. Thus all ranks are available at a smaller
marked position, and shorter words are available at the same position.

\emph{The maximum occurs before the mark.}
Suppose $r$ occurs in position $p\le a$. Delete the marked maximum
from $T_a(\sigma)$, and write $\sigma=T_c(\zeta)$ with
$c=p-1<a$. By induction, $\pi_c\mathscr P_\zeta$ has one column
weight. Transport gives homogeneity for every shuffle of $\zeta$
with increasing higher values whose first higher value follows
position $c$: its row restrictions are stronger than $\pi_c$.
Apply this to
\[ W_M=(\sigma_1,\ldots,\sigma_a,r+1,\ldots,r+M, \sigma_{a+1},\ldots,\sigma_r). \]
Relative to $\zeta$, the higher values are $r,r+1,\ldots,r+M$
in increasing order. Relative to $\sigma$, the row bounds are $i$
for $i\le a$ and $\min(r,i+M)$ afterward. For
$M\ge\max(1,r-a-1)$, these impose exactly $\pi_a$.
The injectivity in Lemma~\ref{lem:rs-completion}(iii) proves the claim.
This case includes $a=r$.

\emph{There is a proper skew decomposition.}
Let $\alpha$ be the first skew-indecomposable factor. If
$|\alpha|\le a$, its maximum is before the mark, so the preceding
case applies. Otherwise write $\sigma=\alpha\ominus\beta$.
Lemma~\ref{lem:rs-record-deletions} gives the smaller admissible pairs
$(a,\alpha)$ and $(0,\beta)$. By induction their restricted
polynomials have one column weight. Lemma~\ref{lem:rs-completion}(ii)
expresses $\pi_a\mathscr P_\sigma$ as their product on disjoint
row sets, which is nonzero and has one column weight.

\emph{The suffix after the mark is increasing.}
We may assume that $\sigma$ is skew-indecomposable and its maximum
occurs after $a$. The last entry is then $r$. Deleting it leaves an
admissible pair at the same marked position, because its marked word
is a standardized prefix of the original. Apply induction and
Lemma~\ref{lem:rs-tail-stability}.

\emph{The suffix is not increasing.}
Lemma~\ref{lem:rs-normal-form} gives a shorter admissible rank word
$(\alpha,h,\ldots,H)$. Set $\eta=\sigma(\alpha)$ and
$L=H-h+1$. The permutation of this shorter word is the $\tau$
of Proposition~\ref{prop:rs-exterior}; inserting the deleted ranks
$1,\ldots,t$ gives its $\chi$. Indeed, terminal insertion of
$h,\ldots,H$ shifts the entries of $\eta$ at least $h$ by $L$,
whereas first inserting $1,\ldots,t$ shifts the threshold to $h-t$.
These are exactly the two maps $\iota_0$ and $\iota_t$ in that
proposition. The position condition in the lemma supplies its remaining
hypothesis, and $|\alpha|>a$ lets exterior insertion commute with
$\pi_a$. Induction and that proposition prove column homogeneity.

Finally, remove the maximum of a nonempty member of $\mathcal C$
to write it as $T_a(\sigma)$. Its row bounds are $i$ for
$i\le a$ and $i+1$ afterward. These further restrict the
column-homogeneous polynomial just proved. By transport the result is
$\Sch_{T_a(\sigma)}(c)$, which is nonzero. Hence it is homogeneous.
\end{proof}

\subsection{Canonical coincidence and parabolic cancellation}
\label{app:canonical-coincidence-proofs}
The inverse transition identities use different hypotheses. Support
in one root-degree fiber makes the inverse Pieri row a M\"obius row;
the gap condition identifies the inverse canonical
transition with the corresponding M\"obius row.

\subsubsection{Inverse Pieri coefficients and M\"obius inversion}
\label{app:single-fiber-rows}

Set $\mathcal A_N=\{a\in\ZZ_{\ge0}^N:a_i\le N-i\}$. For a finite weak
composition $a$, write $a^+$ for its decreasing rearrangement and
$y_a\in S_\infty$ for the permutation with Lehmer code $a$, with
trailing zeros understood. Dominance and strict dominance are denoted by
$\unrhd$ and $\triangleright$. Lemma~\ref{lem:d-content} gives
\[ d(u)=d(v)\quad\Longleftrightarrow\quad\theta(u)^+=\theta(v)^+. \]
The \emph{key polynomials} have conventions
\[ \pi_i f=\partial_i(x_if),\qquad \kappa_\lambda=x^\lambda,\qquad \kappa_{s_i a}=\pi_i\kappa_a\quad(a_i>a_{i+1}), \]
where $\lambda$ is a partition and the rightmost operator acts first;
see \cite{ReinerShimozono,MarbergScrimshaw}.

\begin{lemma}\label{lem:key-Schubert-comparison}
For every finite weak composition $a$, there are integers $t_{a,b}$ such that
\begin{equation}\label{eq:key-Schubert-comparison}
 \kappa_a=\Sch_{y_a}+\sum_{b^+\triangleright a^+}t_{a,b}\Sch_{y_b}.
\end{equation}
Consequently, if $a^+=b^+$, then
\begin{equation}\label{eq:key-Schubert-pairing}
 \partial_{y_a}\kappa_b=\delta_{a,b}.
\end{equation}
\end{lemma}

\begin{proof}
Work stably, and let $F_w^{\rm Stan}=\lim_m\Sch_{1^m\times w}$.
The Schubert-to-key and Stanley-to-Schur tableau formulas give
\[ \Sch_w=\sum_b c_{w,b}\kappa_b,\qquad F_w^{\rm Stan}=\sum_b c_{w,b}s_{b^+},\qquad c_{w,b}\in\ZZ_{\ge0}; \]
see \cite{ReinerShimozono} and \cite[Theorems~2.11--2.12]{ReinerYong}.
The two sums use the same tableaux: the content of the left key sorts to
the tableau shape. Set $a=L(w)$ and $\lambda=a^+$. Stanley's
extremal-shape theorem \cite[Theorem~4.1]{StanleyReduced}, with partitions
conjugated from his convention to the stable-Schubert convention above, gives
\[ c_{w,b}\ne0\ \Longrightarrow\ b^+\unrhd\lambda, \qquad \sum_{b^+=\lambda}c_{w,b}=1. \]
The lexicographically least monomials of $\kappa_b$ and $\Sch_w$
are respectively $x^b$ and $x^a$, each with coefficient one. For
keys this is \cite[Equation~(2.3)]{MarbergScrimshaw}; for Schubert
polynomials it also follows from the KP realization: its cyclic vector
has weight $a$, and upper-triangular root operators strictly increase
lexicographic weight \cite[Section~2.2]{WatanabeRingel}.
Positivity of the key expansion therefore forces $c_{w,a}=1$: a key
contributing to $x^a$ has $b\le_{\rm lex}a$, whereas its least
monomial occurring in $\Sch_w$ gives $b\ge_{\rm lex}a$.
The extremal coefficient excludes every other summand of shape
$\lambda$. Thus $\Sch_w=\kappa_a+\sum_{b^+\triangleright a^+}c_{w,b}\kappa_b$.
Inversion in partition dominance proves \eqref{eq:key-Schubert-comparison};
it terminates because the degree is fixed. The equal-degree nil-Hecke
pairing $\partial_u\Sch_v=\delta_{u,v}$ proves
\eqref{eq:key-Schubert-pairing}.
\end{proof}

For a partition $\lambda$, let $W_J$ be its stabilizer and give its
rearrangements the multiset Bruhat order: $b\le a$ when
$\sigma_b\le\sigma_a$, where $\sigma_a\in W^J$ is minimal with
$\sigma_a\lambda=a$. Thus $\lambda$ is minimal; increases interchange
$a_p>a_q$ at positions $p<q$, not necessarily adjacent.

\begin{lemma}\label{lem:key-orbit-and-staircase}
If $a^+=\lambda$, then
\begin{equation}\label{eq:key-orbit-terms}
 \kappa_a=\sum_{\substack{b^+=\lambda\\b\le a}}x^b
       +\sum_{b^+\lhd\lambda}k_{a,b}x^b.
\end{equation}
Every exponent $b$ occurring in $\kappa_a$ satisfies
\begin{equation}\label{eq:key-suffix-bound}
 b_i\le\max_{j\ge i}a_j.
\end{equation}
In particular, $a\in\mathcal A_N$ implies
$\operatorname{supp}(\kappa_a)\subseteq\mathcal A_N$.
\end{lemma}

\begin{proof}
Modulo monomials of strictly smaller sorted shape, the two-variable
formula gives
\[
 \pi_i x^b\equiv
 \begin{cases}
 x^b+x^{s_i b},&b_i>b_{i+1},\\
 x^b,&b_i=b_{i+1},\\
 0,&b_i<b_{i+1}.
 \end{cases}
\]
The omitted terms strictly balance the two exponents, so smaller shapes
remain smaller under $\pi_i$ and $\bar\pi_i=\pi_i-1$.
Use $\pi_\sigma=\sum_{\tau\le\sigma}\bar\pi_\tau$
\cite[Section~2.1]{MarbergScrimshaw}. If $\tau\notin W^J$, its
parabolic factorization has a nontrivial right factor in $W_J$, whose
rightmost operator annihilates $x^\lambda$. If $\tau\in W^J$, each
suffix of a reduced expression is also in $W^J$, since a right descent
of that suffix in $J$ would be a right descent of $\tau$. Moreover,
$\ell(u)=\#\{p<q:(u\lambda)_p<(u\lambda)_q\}$ for $u\in W^J$.
Thus each step in operator-action order swaps strictly descending
entries, and $\bar\pi_\tau x^\lambda\equiv x^{\tau\lambda}$.
Summing over $\tau\le\sigma_a$ proves \eqref{eq:key-orbit-terms},
including coefficient one for each term of shape $\lambda$.

For \eqref{eq:key-suffix-bound}, induct through the key recursion. If
$a_i<a_{i+1}$ and $a'=s_i a$, the suffix maxima outside $i,i+1$
are unchanged, while
$M_i(a)=M_{i+1}(a)=M_i(a')\ge M_{i+1}(a')$, where
$M_k(a)=\max_{j\ge k}a_j$.
The two relevant exponents before applying $\pi_i$ are at most
$M_i(a)$, and in each term of $\pi_i x^b$ neither exponent exceeds
$\max(b_i,b_{i+1})$.
This proves the induction; for $a\in\mathcal A_N$, one has
$M_i(a)\le N-i$.
\end{proof}

Reversal and addition of one identify these multiset Bruhat orders with
those of rank words: swapping $a_p>a_q$, $p<q$, swaps an ascending
pair at positions $N+1-q<N+1-p$. Passing to a smaller word
in Bruhat order preserves the rank-word bounds, by
Proposition~\ref{prop:bruhat-interval}. Consequently, for
$z\in I_d$
and $\lambda=\theta(z)^+$,
\begin{equation}\label{eq:key-root-fiber-expansion}
 \kappa_{\theta(z)}=\sum_{\substack{v\in I_d\\v\le z}}x^{\theta(v)}
       +\sum_{\substack{a\in\mathcal A_N\\a^+\lhd\lambda}}k_{z,a}x^a.
\end{equation}
The staircase bound ensures that every exponent in the second sum is
$\theta(u)$ for $u\in S_N$.

\begin{proof}[Proof of Proposition~\ref{prop:single-fiber-Mobius-row}]
Apply $\partial_{ww_0}$ to \eqref{eq:key-root-fiber-expansion}.
By \eqref{eq:key-Schubert-pairing}, the left side is $\delta_{w,z}$.
Every term in the second sum has the same total degree and lies outside
$I_d$, so Lemma~\ref{lem:complementary-code} and the support
hypothesis annihilate it. Therefore
\[
 \sum_{\substack{v\in I_d\\v\le z}}Q_N(w,v)=\delta_{w,z}\qquad(z\in I_d).
\]
Multiplication on the right by the inverse of the Bruhat zeta matrix
proves $Q_N(w,v)=\mu_{I_d}(w,v)$.
\end{proof}

\begin{remark}[The support hypothesis is essential]
The proof only requires vanishing on strictly dominated
complementary-code shapes. It does not identify an unrestricted
root-degree submatrix with a M\"obius matrix: for $w=53142$ and
$v=34521$ in $S_5$, the root degrees agree, but
$Q_5(w,v)=-1$ and $\mu_{I_{d(w)}}(w,v)=0$.
\end{remark}

\subsubsection{Cancellation of parabolic Kazhdan--Lusztig sums}
\label{app:parabolic-KL}

\paragraph{Words and parabolic fibers.}
Let $c_1<\cdots<c_n$ be positive integer capacities, and let
$r,b$ be positive integer words of the same content, with
$r_i,b_i\le c_i$. Assume
\begin{equation}\label{pc-eq:delta}
 \Delta_p(a):=\#\{i\le p:r_i\le a\}-\#\{i\le p:b_i\le a\}\ge0
 \qquad(p,a),
\end{equation}
and the gap condition
\begin{equation}\label{pc-eq:sgc}
 q<j,\quad 1<r_j\le r_q<c_q
 \quad\Longrightarrow\quad
 \exists t\ (q<t<j),\quad r_j-1\le r_t\le r_q.
\end{equation}
A position is \emph{saturated} if $r_i=c_i$. A common position
$p$, of value $d=r_p=b_p$, is \emph{cancellable} when
\begin{equation}\label{pc-eq:bw-word}
 \Delta_{p-1}(d-1)=\Delta_{p-1}(d)=0.
\end{equation}
Deletion removes the same position from $(c,r,b)$, leaving the
capacities at the remaining positions unchanged.

Let $W=S_n$, and let $W_J$ have consecutive position blocks of sizes
the nonzero letter multiplicities, in increasing letter order. Write
$y(a)\in W^J$ for the list of positions occupied by each letter,
increasing within each letter, and set
\[ \phi(u)=w_0uw_J,\qquad x=\phi(y(b)),\qquad\tau=\phi(y(r)),\qquad F(x,\tau;J)=\{t\in W_J:xt\le\tau\}. \]

\begin{lemma}[Standardization and deletion]\label{pc-lem:dictionary}
The inequalities \eqref{pc-eq:delta} are equivalent to
$y(r)\le y(b)$, or $x\le\tau$. The consecutive value components of
$D_L(\tau)$ correspond, in reverse order, to the maximal strict-ascent
runs of $r$; their incidence with a position block of $W_J$ counts occurrences
of its lower-word letter in that run. Replacing the lower permutation by the maximum of its
$W_{D_L(\tau)}$-orbit arranges the entries of the lower word
in increasing order on each run.

A cancellable position in the words gives a common cancellable dot $(k,v)$ of
$[x,\tau]$. Every $t\in F(x,\tau;J)$ fixes $k$. Deleting the
dot and shortening its position block by one gives $J'$ and a bijection
\[ F(x,\tau;J)\simeq F(\widehat x,\widehat\tau;J') \]
which preserves $(-1)^{\ell(t)}$ and every polynomial $P_{xt,\tau}$.
The two deleted permutations belong to $W^{J'}$.
\end{lemma}

\begin{proof}
The first assertion is the rank-matrix criterion for parabolic Bruhat
order \cite[Chapter~2]{BjornerBrenti}. The left descents of
$w_0y(r)w_J$ are the dual values at the strict ascents of $r$, giving
the correspondence with ascent runs and the effect of taking the
maximum in the left orbit.
The two equalities in \eqref{pc-eq:bw-word} identify the same occurrence
of $d$ and its left inversion count; equal total content gives the
right count. These are the cancellability conditions of
\cite[Section~2]{Henderson}. By \cite[Proposition~2.2]{Henderson}, every
element of the interval has the common dot, and deleting it preserves
the interval and its Kazhdan--Lusztig polynomials. Since $x\in W^J$,
all fiber elements $xt$ belong to this interval, so $t(k)=k$.
Deletion and insertion give the stated bijection; signs agree because
\[ \ell(t)-\ell(t')=2\#\{i<k:t(i)>k\}. \]
Deleting a common entry preserves increase within the remaining
position blocks.
\end{proof}

\paragraph{Normalization.}
Set $D=D_L(\tau)$. Let $C_\alpha$ be its consecutive value
components, $B_\beta$ the position blocks of $W_J$, and
$M_{\alpha\beta}=\#\{i\in B_\beta:x(i)\in C_\alpha\}$.
The ideal below $\tau$ is $W_D$-stable, and $P_{u,\tau}$ is
constant on its left $W_D$-orbits, by lifting and left-descent
invariance \cite[Chapters~2 and~5]{BjornerBrenti}.

\begin{proposition}\label{pc-prop:normalization-dichotomy}
If some $M_{\alpha\beta}\ge2$, the terms indexed by $F(x,\tau;J)$
have a fixed-point-free sign-reversing involution preserving
$P_{xt,\tau}$. Otherwise there is one $h\in W_D$, independent of
$t$, such that $hxt$ is the maximum of $W_Dxt$ for every
$t\in W_J$, and
\[ hx\in W^J,\qquad F(hx,\tau;J)=F(x,\tau;J),\qquad P_{hxt,\tau}=P_{xt,\tau}\quad(t\in F). \]
In this second case, the normalized lower word is strictly increasing on
each ascent run of $r$, still satisfies the capacity bounds, and
satisfies \eqref{pc-eq:delta}. The upper word and the signs of all
summands are unchanged.
\end{proposition}

\begin{proof}
If $M_{\alpha\beta}\ge2$, two adjacent positions $p,p+1$ of
$B_\beta$ have values in $C_\alpha$, because $x$ is increasing
on $B_\beta$. For $a=s_p\in W_J$, one has
$xax^{-1}\in W_D$. Thus $t\mapsto at$ preserves the fiber and the
Kazhdan--Lusztig polynomial, and reverses parity without fixed points.

If $M$ is zero-one, each $D$-component occurs at most once in any
position block $B_\beta$. Permuting positions within each $B_\beta$
cannot change the order of these occurrences between distinct position
blocks. Hence the same $h\in W_D$
places the values of each component decreasingly in their occupied
positions for every $xt$, proving simultaneous maximization.
It preserves the relative order of distinct components, so $hx\in W^J$.
Invariance under the left action gives the equalities of fibers and
polynomials without changing $t$. Lemma~\ref{pc-lem:dictionary} identifies this operation
with increasing rearrangement of each lower run; the zero-one condition makes that increase strict. Rearranging these
entries preserves the capacity bounds:
among the first $i$ capacities of a run there were already $i$
values at most its $i$-th capacity, so the $i$-th smallest value is
also allowed. Finally, $hx\le\tau$ restores
\eqref{pc-eq:delta} by the same lemma.
\end{proof}

\paragraph{Deletion preserving the gap condition.}
\begin{lemma}[Deletion lemma for bounded words]\label{pc-thm:safe-pivot}
Assume additionally that $b$ is strictly increasing on every maximal
strict-ascent run of $r$. If $b\ne r$, there is a cancellable position whose deletion preserves
\eqref{pc-eq:sgc}. Every saturated position has this property; if none
is saturated, the rightmost cancellable position has this property.
\end{lemma}

\begin{proof}
If $r_p=c_p=C$, every earlier entry of both words is less than $C$.
Were $b_p<C$, the difference $\Delta_p(C-1)$ would be $-1$.
Thus $b_p=C$ and \eqref{pc-eq:bw-word} holds. This position cannot
be the only intermediate position satisfying the gap condition: if
$q<p$ and $r_q<c_q$, then $r_q<c_q<c_p=r_p$. Its deletion
therefore preserves \eqref{pc-eq:sgc}.

Now suppose no position is saturated, and write $R_1,\ldots,R_m$ for
the maximal strict-ascent runs of $r$. At a run boundary the next
entry is weakly smaller, so \eqref{pc-eq:sgc} forces each later run to
start at $1$. Let $d_k$ be the number of occurrences of $1$ in the upper word
through $R_k$, minus the corresponding number in the lower word.
The lower word is strictly increasing on each run and therefore contains
at most one occurrence of $1$ there. Hence
\[ d_k-d_{k-1}=1-\one_{\{1\in b|_{R_k}\}}\ge0\quad(k\ge2), \qquad 0\le d_1\le\cdots\le d_m=0. \]
The lower word also begins with $1$ on each later run; that
position is cancellable. If $m=1$, both words are increasing
with the same content and hence equal. Thus a cancellable position exists when $b\ne r$.

The rightmost such position $p$ lies in the last run. Suppose deleting
it causes the gap condition to fail at $q<p<j$, and write $A=r_q$,
$B=r_j$, $d=r_p=b_p$. The position $p$ was the only one between
$q$ and $j$ with value in $[B-1,A]$, so $B-1\le d\le A$.
If $d\ge B$, the pair $(q,p)$ would violate the gap condition in
the original word, because
$[d-1,A]\subseteq[B-1,A]$. Hence $d=B-1$.
Both $p,j$ lie in the last increasing run, so integrality gives
$j=p+1$ and $r_j=d+1$.

If $e$ ends the preceding run, equal contents and
\eqref{pc-eq:delta} give
\[ \#\{i\in R_m:r_i\le a\}-\#\{i\in R_m:b_i\le a\} =-\Delta_e(a)\le0. \]
The two run subwords are increasing, so $b_i\le r_i$ on $R_m$.
Thus $d<b_j\le d+1$, and $b_j=r_j=d+1$.
Cancellation at $p$ gives $\Delta_{j-1}(d)=0$; all entries after
$j$ exceed $d+1$, so total content equality gives
$\Delta_{j-1}(d+1)=0$. Therefore $j$ is cancellable, contradicting
the choice of $p$.
\end{proof}

\paragraph{Induction and inversion.}
\begin{theorem}\label{pc-thm:signed-fiber}
For words of equal content satisfying the capacity bounds,
\eqref{pc-eq:delta}, and \eqref{pc-eq:sgc}, the associated parabolic
sum satisfies
\[
 A_J(x,\tau;q):=\sum_{t\in F(x,\tau;J)}(-1)^{\ell(t)}P_{xt,\tau}(q)
 =\begin{cases}1,&F(x,\tau;J)=\{e\},\\0,&|F(x,\tau;J)|>1.\end{cases}
\]
The equality is coefficientwise in $q$.
\end{theorem}

\begin{proof}
Induct on the number of positions. Proposition~\ref{pc-prop:normalization-dichotomy}
either cancels all terms in a nonsingleton fiber or normalizes the lower
word while preserving the fiber and every summand. If the normalized
words agree, $x=\tau$, and
$\ell(\tau t)=\ell(\tau)+\ell(t)$ forces $F=\{e\}$.
Otherwise Lemma~\ref{pc-thm:safe-pivot} supplies a cancellable position.
By Lemma~\ref{pc-lem:dictionary}, deletion identifies the sum with the
corresponding sum in rank one less and preserves the cardinality of the
fiber. The shorter words have equal content and satisfy the remaining
capacity bounds. Their initial-segment differences equal the corresponding
differences before deletion, and the chosen deletion preserves the gap
condition. The induction hypothesis therefore applies. It imposes no
order condition on the lower word within the new ascent runs, since each
induction step begins with normalization.
\end{proof}

\begin{example}[Cancellation before specialization]\label{ex:KL-polynomial-cancellation}
Take capacities $c_i=i$ and words
\[ r=(1,1,1,2,1,2),\qquad b=(1,2,1,2,1,1). \]
They satisfy the gap and dominance conditions. Their parabolic subgroup
has block sizes $(4,2)$, so $J=\{1,2,3,5\}$, and the corresponding
permutations are
\[ x=(1,2,4,6,3,5),\qquad \tau=(2,4,5,6,1,3). \]
The fiber has four pairs $\{t,s_1t\}$, with
\[
\begin{array}{c|c|c|c}
 t&xt&x s_1t&P_{xt,\tau}(q)=P_{x s_1t,\tau}(q)\\\hline
 e&(1,2,4,6,3,5)&(2,1,4,6,3,5)&1+q\\
 s_5&(1,2,4,6,5,3)&(2,1,4,6,5,3)&1\\
 s_2&(1,4,2,6,3,5)&(2,4,1,6,3,5)&1\\
 s_2s_5&(1,4,2,6,5,3)&(2,4,1,6,5,3)&1
\end{array}
\]
The simple reflection $s_1$ belongs to both $W_J$ and
$W_{D_L(\tau)}$, and $xs_1x^{-1}=s_1$. Thus
Proposition~\ref{pc-prop:normalization-dichotomy} pairs the two terms
in each row with opposite signs. In particular, the nonconstant terms
cancel as polynomials, not merely after setting $q=1$.
\end{example}

For $u,v\in W^J$, finite spherical--antispherical inversion states
\[ S^J_{u,v}(q)=P_{uw_J,vw_J}(q),\qquad (S^J(q)^{-1})_{u,v} =(-1)^{\ell(v)-\ell(u)}A_J(\phi(v),\phi(u);q); \]
see \cite[Proposition~6.1]{Couillens}. On a root-degree fiber,
$D^{\rm can}(r,v;q)=S^J_{y(r),y(v)}(q)$, whereas Deodhar's relative
M\"obius formula \cite{Deodhar,StembridgeMobius} gives
\[
 \mu_{I_d}(r,v)=
 \begin{cases}
 (-1)^{\ell(y(v))-\ell(y(r))},&F(\phi(y(v)),\phi(y(r));J)=\{e\},\\
 0,&|F(\phi(y(v)),\phi(y(r));J)|>1.
 \end{cases}
\]
Theorem~\ref{pc-thm:signed-fiber}, with $c_i=i$, identifies
the inverse canonical transition coefficients at an
index satisfying the gap condition with the corresponding M\"obius
coefficients. Equivalently,
\[ \sum_{v\in I_d}\mu_{I_d}(w,v)D^{\rm can}(v,z;q)=\delta_{w,z}, \]
which proves Theorem~\ref{thm:canonical-Mobius-row}.

\subsection{The negative Laurent coefficient}
\label{app:laurent-record}

For $\alpha\in\op{Prof}_N$, write
$q_w(\alpha)=Q_N(w,\tau_\alpha)$ and
$\g(\alpha)=\g_{\tau_\alpha}$. Thus
$\Sch_w=\sum_\alpha q_w(\alpha)\mathcal E_\alpha$.
Write $\op{Laur}F$ for the Laurent expansion of $F$ in the initial
triangular seed. As in Appendix~\ref{app:record-sufficiency}, write
$T_a(u)=(u_1,\ldots,u_a,n+1,u_{a+1},\ldots,u_n)$ for
$u\in S_n$.

\subsubsection{Triangularity and transport}

For profiles of the same weight, define $\beta\succeq\alpha$ by
\begin{equation}\label{eq:lp-order}
 \#\{k\le p:\rho_k(\beta)\le b\}\ge
 \#\{k\le p:\rho_k(\alpha)\le b\}\qquad(p,b).
\end{equation}
A supported profile is \emph{maximal in its weight} if no different
supported profile dominates it.

\begin{lemma}\label{lem:lp-triangular}
If $\alpha$ is maximal in its weight in
$F=\sum_\beta b_\beta\mathcal E_\beta$, then
$[\mathbf x^{\g(\alpha)}]\op{Laur} F=b_\alpha$.
\end{lemma}
\begin{proof}
Factor the unipotent matrix using independent parameters:
\[ Z(\mathbf t)=\prod_{s=1}^{N-1} \bigl((I+t_{s,N-1}E_{N-1,N})\cdots(I+t_{s,s}E_{s,s+1})\bigr). \]
Its entries are path sums
\[ z_{p,q}(\mathbf t)= \sum_{1\le s_p<\cdots<s_{q-1},\ s_i\le i} \prod_{i=p}^{q-1}t_{s_i,i}. \]
The earliest path uses stages $1,\ldots,q-p$; every other path delays
at least one step. The path determinant for an initial minor has a unique
nonintersecting family, giving
\begin{equation}\label{eq:lp-monomial-change}
 x_{a,r}=\prod_{s=1}^r\prod_{j=0}^{a-1}t_{s,s+j},\qquad
 t_{s,i}=\frac{x_{a,s}x_{a-1,s-1}}{x_{a-1,s}x_{a,s-1}},\quad a=i-s+1.
\end{equation}
Here $x_{0,r}=x_{a,0}=1$. For uniqueness of the family, the path from
row one must use stages $1,\ldots,r$; disjointness forces every later
path to leave each level before the preceding path arrives, and hence to
use those same stages. The earliest term of $z_{p,p+r}$ is
$x_{p,r}/x_{p-1,r}$, so that of $\mathcal E_\alpha$ is
$m_\alpha=\mathbf x^{\g(\alpha)}$, with coefficient one.

If $\mathcal E_\beta$ contributes $m_\alpha$, compare factors at edge
$i$ in stages at most $s$. Their earliest-path number is
$\#\{q>i:i-s+1\le\rho_q\le i\}$, and delays can only decrease it.
Since total rank-letter contents agree, this yields
\[ \#\{q\le i:\rho_q(\alpha)\ge i-s+1\}\ge \#\{q\le i:\rho_q(\beta)\ge i-s+1\}, \]
which is $\beta\succeq\alpha$. Maximality leaves only $\alpha$.
Any nontrivial delay increases the sum of all stages, proving that its
own contribution is unique. The invertible integral monomial change
\eqref{eq:lp-monomial-change} preserves coefficients.
\end{proof}

\begin{lemma}\label{lem:lp-transport-max}
The following operations preserve maximal supported coefficients. A coefficient of weight different from the code weight continues to
have that property:

\smallskip\noindent
\textup{(i)} Insert a new maximum in position $m$, and insert the rank
letter $m$ in position $m$ of the profile, leaving all old letters unchanged.

\smallskip\noindent
\textup{(ii)} From $T_i(u)$ to $T_{i-1}(u)$, replace the profile
letters $(s,i+1)$ at positions $i,i+1$ by $(i,s)$.
\end{lemma}
\begin{proof}
In (i), transport omits the new maximum's flag, so every supported
profile of the enlarged permutation has letter $m$ at position $m$.
The stated lift
has its old coefficient. If $A\succeq B$, with $B$ a lifted
profile, then for $k>m$ every letter in the first $k$ positions
of $B$ is at most $k-1$. The same is true of $A$, by
\eqref{eq:lp-order}. Deleting the common $m$ therefore leaves legal
old profiles, with their old coefficients and the same order comparison.
Old maximality gives $A=B$.

In (ii), moving the maximum left raises the active low's flag from
$i-1$ to $i$. The coefficients in columns $s\le i$ are unchanged;
only column $i+1$ is newly allowed. Every parent profile has letter
$i$ at position $i$. If $A\succeq B$ and $B$ is the
stated lift, its first $i+1$ letters are at most $i$. Thus
$\rho_{i+1}(A)\le i$, so $A$ also comes from the child.
The child comparison at prefix $i$ is the parent comparison at
$i+1$ after deleting the common $i$; later prefixes replace
that letter by $i+1$. Child maximality gives equality.
The weight changes are respectively $\eps_m-\eps_{n+1}$ and
$\eps_i-\eps_{i+1}$, also the changes of the code weight.
\end{proof}

\subsubsection{The exposed boundary coefficient}

Let $\theta$ be a Lehmer code in rank $n$, and set
$L_\theta(f)=\partial_yf$ on degree $|\theta|$, where $y$
has code $\theta$. Write $a^+$ for decreasing rearrangement. Set
\[ B_j=\{a\in\mathbb Z_{\ge0}^n:a_q\le n-q+\one_{q\le j}\},\qquad 2\le j\le n,\quad i=n+1-j,\quad t=i-\theta_j>0. \]
The functional descends to the coinvariant algebra, so a monomial with
some exponent at least $n$ has value zero. The extra boundary powers
allowed in these boxes are therefore harmless. Call $h<j$ \emph{exposed} if
\[ A:=\theta_h\ge t,\qquad \theta_k\notin[A-t,A-1]\quad(h<k\le j). \]

\begin{proposition}\label{prop:lp-boundary}
Assume the following vanishing condition:
\begin{equation}\label{eq:lp-purity}
 a\in B_{j-1},\quad |a|=|\theta|,\quad a^+\ne\theta^+
 \quad\Longrightarrow\quad L_\theta(x^a)=0.
\end{equation}
For every exposed $h$, the vector
$\gamma=\theta-t\eps_h+t\eps_j$ satisfies $\gamma^+\ne\theta^+$ and
$L_\theta(x^\gamma)=-1$. Moreover,
\begin{equation}\label{eq:lp-extreme}
 \left.
 \begin{gathered}
 a\in B_j,\quad a^+=\gamma^+,\quad L_\theta(x^a)\ne0,\\
 \sum_{q\le p}a_q\ge\sum_{q\le p}\gamma_q\quad(1\le p\le n)
 \end{gathered}\right\}\ \Longrightarrow\ a=\gamma.
\end{equation}
\end{proposition}
\begin{proof}
Reverse the canonical divided-difference word in strings indexed by
$q=1,\ldots,n$, with edges $q+\theta_q-1,\ldots,q$.
A reverse step at a pair $(C,D)$ produces
\[ (C+D+1-m,m)\text{ with sign }+1,\quad (m,C+D+1-m)\text{ with sign }-1,\quad 0\le m\le\min(C,D). \]
A term in the expansion is obtained by choosing one summand at each
reverse step; its sign is the product of the corresponding signs.
We examine these choices before combining like monomials. After the
$q$th string, the coordinates through $q$ are
fixed, and the remaining sum of exponents is
\begin{equation}\label{eq:lp-mass}
 S_q(a)=\sum_{k\le q}(\theta_k-a_k)\ge0.
\end{equation}
For a canonical prefix the remaining mass is zero, and exactly one
choice contributes, with positive sign.

The following one-string observation controls the possible transfer
of exponent mass. Starting
with zero mass, if string $h$, of length $A$, fixes its first
coordinate at $A-s$, $s>0$, then its remaining mass is supported
in positions at least
$P(s)=h+A-s+1$, with a positive amount at $P(s)$.
The monomial corresponding to $s\eps_{P(s)}$ has one
contribution, of sign $-1$;
every other distribution has mass strictly beyond $P(s)$.
This follows by induction on $A$. The last edge starts with left
coordinate zero. A positive fixed coordinate forces the positive step
and reduces to length $A-1$ with the same $s$. A zero fixed
coordinate forces a negative step, with positive mass at $h+1=P(A)$.
All mass stays there only when the preceding shorter string was canonical.
This proves all three assertions, including uniqueness and sign.

The essential consequence of \eqref{eq:lp-purity} is
\begin{equation}\label{eq:lp-unique-reception}
 1\le s\le t,\quad h<k\le j,\quad
 k\le P(s)\le k+\theta_k
 \quad\Longrightarrow\quad(s,k)=(t,j).
\end{equation}
Otherwise choose the least such exceptional $s$, then the first
string $k$ meeting $P(s)$. The unsplit distribution, canonical
intermediate strings, and complete positive absorption at $k$ give
$\xi=\theta-s\eps_h+s\eps_k$. No other choice yields
this monomial: a split would require $s\ge2$ and the $k$th
string to reach $P(s)+1=P(s-1)$, contradicting
minimality of $s$. The intermediate strings cannot meet the donor
mass and have prescribed canonical coordinates, so are forced; after
absorption the remaining mass is zero. Hence $L_\theta(x^\xi)=-1$.

Also $\xi\in B_{j-1}$. For $k=j$, exceptionality gives
$s<t$, so $\theta_j+s\le i-1$. For $k<j$, a violated
bound would give $s\ge2$ and
\[ k+\theta_k\ge n-s+2\ge h+A-s+2=P(s-1), \]
again contradicting minimality. Exposure gives $\xi^+\ne\theta^+$:
unchanged content would require $\theta_k=A-s$. Thus
\eqref{eq:lp-purity} rules out $\xi$, proving \eqref{eq:lp-unique-reception}.

Since $P(t)\le n+1-t=j+\theta_j$, \eqref{eq:lp-unique-reception}
also gives $P(t)\ge j$: otherwise string $P(t)$ is an earlier
meeting. Therefore
\begin{equation}\label{eq:lp-separation}
 j\le P(t)\le j+\theta_j,\qquad
 k+\theta_k<P(t),\quad \theta_k<A-t\quad(h<k<j).
\end{equation}
For $\gamma$, intermediate strings cannot meet the donor mass.
A split would require $j$ to reach $P(t-1)$, excluded by
\eqref{eq:lp-unique-reception}; for $t=1$ no split exists. Exactly
one choice of summands therefore gives $\gamma$, with sign $-1$.
Its decreasing rearrangement differs from $\theta^+$ because $\theta_j\ne A-t$.

Finally, take $a$ as in \eqref{eq:lp-extreme}. Equation~\eqref{eq:lp-purity} forces
$a_j=i$. Combining its prefix inequalities with \eqref{eq:lp-mass}
fixes every coordinate before $h$ and after $j$ at its
canonical value, and gives $A-t\le a_h\le A$. Its remaining
multiset is
$\{A-t,\theta_{h+1},\ldots,\theta_{j-1}\}$. By
\eqref{eq:lp-separation}, only its first member is at least $A-t$,
so $a_h=A-t$. At the first differing intermediate coordinate,
the exponent cannot exceed its canonical value: that string cannot
reach the donor mass. The prefix inequality would require it to be
larger. Hence there is no difference and $a=\gamma$.
\end{proof}

\subsubsection{Record defects and the global conclusion}

\begin{lemma}\label{lem:lp-bad-good}
If $u\in\mathcal C_n$, $T_i(u)\in\mathcal C_{n+1}$, and
$T_{i-1}(u)\notin\mathcal C_{n+1}$, then $\Sch_{T_{i-1}(u)}$
has a maximal profile in a different root degree with coefficient $-1$.
\end{lemma}
\begin{proof}
Compare records when one higher letter moves across low position $i$.
At a low insertion with old low prefix $v\in S_m$, new rank $H$,
and positions $p(x)=v^{-1}(x)$, all ordinary low records and their
prefix ranks are unchanged. The only extra selections are:
\[
\begin{array}{ll}
 H=m+1:&x\text{ an old right-to-left maximum with }p(x)>a;\\
 2\le H\le m:&x=m,\ p(m)>a,\ p(m)>\max_{H-1\le z<m}p(z).
\end{array}
\]
At cut $a$, the first selection is defective exactly when
$p(x)<m$, the second exactly when $p(m)\ne H-1$.
The high does not affect a low prefix rank and cannot itself be selected;
it enters precisely these additional comparison sets. Passing from $T_{i-1}(u)$ to $T_i(u)$ therefore removes a
defective selection at position $i$.
Suppose this happens at low insertion $k=m+1$, and write $r_l$
for the low rank word.

In the first case all values greater than $x$ precede that position,
so $r_i=i-(k-1-x)$ and $t=i+1-r_i=k-x$. Defectiveness gives
$r_i<x$; all intervening values and ranks are also below $x$.
Since $r_k=k$, the interval $[r_k-t,r_k-1]$ has no rank at
positions $i,\ldots,k-1$. In the second case $r_i=i$,
$t=1$, and $r_k=H$; the wall rank is not $H-1$, and
all intervening ranks are below $H-1$. Again that interval is empty.
In both cases $r_k-1\ge t$.

Set $\theta=(r_n-1,\ldots,r_1-1)$, $j=n+1-i$, and
$h=n+1-k$. These are exactly the exposure inequalities.
Transport \eqref{eq:rs-transport} identifies the two coefficient rows
with $L_\theta$ on $B_j$ and $B_{j-1}$. The canonical
rank word survives with coefficient one; the permutation $T_i(u)$ is
homogeneous, so \eqref{eq:lp-purity} holds. Proposition~\ref{prop:lp-boundary}
provides $\gamma$ with coefficient $-1$.

Reinsert the omitted rank letter $i$ at position $i$ to obtain
its parent profile. A dominating same-weight profile would still dominate
after deleting that common letter. Reversal and equality of total contents
then give the prefix-sum inequalities in \eqref{eq:lp-extreme}, forcing
equality. Thus the profile is maximal and remains of a different root degree.
\end{proof}

The global minimal-counterexample deduction, from the exposed
coefficient and transport lemmas proved above, is given in
Theorem~\ref{thm:lp-negative-coefficient} of the main text.

\section{Diagrams, pipe dreams, and enumeration}
\label{app:diagram-enumeration}

\subsection{Plactification and the cover inequalities}\label{app:pipe-landing}\label{solpipe:sec:peeling}
For an ordinary word, use row insertion with weak rows and strict columns;
write $P(v)$ for its insertion tableau. Read a tableau by columns from left
to right and from bottom to top within a column. A word is
$\mathsf D$-\emph{peelable} when its insertion tableau is peelable in the
sense of \cite{PipeRS}: if $V$ is the increasing list in the first nonempty
column of $\mathsf D$, it is an initial segment of the first column of
$P(v)$, and rectifying $P(v)-V$ gives a tableau peelable for the diagram
with that column deleted. For the empty diagram, the peeling recursion terminates at the
empty tableau.

\begin{lemma}[Common columns give full decreasing subsequences]\label{solpipe:lem:common-columns}
Let $v$ be $\mathsf D$-peelable, with $\mathsf D$ northwest. In a fixed
numerical row interval, let $H$ be its nonempty row labels. If $c$ columns
of $\mathsf D$ contain every row of $H$, then $v|_H$ has $c$
pairwise disjoint decreasing subsequences, each containing every member of
$H$ exactly once.
\end{lemma}
\begin{proof}
For a word over an ordered alphabet of $h$ letters, let $\nu$ be the
maximum number of disjoint decreasing subsequences of length $h$.
Greene's theorem gives
\[ \nu(v)=\lambda_h,\qquad \lambda=\operatorname{shape}P(v). \]
Indeed, a union of $k$ decreasing subsequences has maximum size
$\lambda'_1+\cdots+\lambda'_k$. Every such subsequence has length at most
$h$, so the maximum equals $kh$ precisely when the first $k$ columns
all have height $h$. Thus $\nu$ is Knuth-invariant. The word version
follows also by standardizing equal letters from left to right. Restriction
to a numerical alphabet interval preserves Knuth equivalence. See
\cite{PipeGreene}; \cite[Theorems~2 and~8]{PipeKrattenthaler} recalls the required form.

Induct on the number of nonempty columns of $\mathsf D$. Set $T=P(v)$,
and remove its initial first-column segment $V$. The column word $u$
of the remaining skew filling is a subword of $\op{cw}(T)$, and $P(u)$
is peelable for $\mathsf D'$, the diagram with its first column removed.
By the induction hypothesis, $u|_H$ has as many disjoint full $H$-chains
as there are full $H$-columns remaining in $\mathsf D'$. This follows
from the hypothesis for $P(u)$ by interval restriction and Knuth invariance.
If a member of $H$ disappears entirely, there is no full $H$-column left,
and this particular assertion is empty.

If the deleted column contains $H$, its occurrences in $V$ form one more
decreasing full $H$-chain, disjoint from $u$. Otherwise no additional
chain is needed. Transfer the resulting inequality from $\op{cw}(T)|_H$ to
$v|_H$ by Knuth invariance. Empty rows in the numerical interval never occur
as letters, so restriction to $H$ is exactly interval restriction throughout.
\end{proof}

We now identify the precise plactification convention. For words in the rank
$N$ alphabets set
\[ \kappa_N(a_1\cdots a_L)=(N-a_L)\cdots(N-a_1),\qquad \eta_N(v_1\cdots v_L)=(N+1-v_L)\cdots(N+1-v_1). \]
The left plactification of \cite{PipeRS} is
$\phi(i\,a)=i\,\sigma_i(\phi(a))$. The pairing definition gives
$\eta_N\sigma_i=\sigma_{N-i}\eta_N$, and induction in \eqref{solpipe:eq:psi} gives
\[ {\eta_N\psi(a)=\phi(\kappa_N(a)).} \]
If $a$ represents $w$, then $\kappa_N(a)$ represents
$(w_0ww_0)^{-1}$. Reiner--Shimozono's Theorem~19(B), with the peelability
description of their Theorem~22, therefore says that
$\eta_N\psi(a)$ is $D(w_0ww_0)$-peelable. 
We deduce the cover inequalities from this peelability statement;
they are not hypotheses of the result just cited.

\subsection{A two-merge example}

\begin{example}\label{ex:two-merge-pipe-dreams}
Let $w=(4,3,8,7,2,6,1,5)$. Its rank word is
$(1,1,3,3,1,4,1,5)$, and its nontrivial cover component is
\[ 4\longrightarrow6\longrightarrow8,\qquad 5\longrightarrow6,\qquad7\longrightarrow8; \]
slot $2$ is isolated. Two cover tableaux agree in columns
\[ C_2=(1),\quad C_4=(3),\quad C_5=(1,2,3,4),\quad C_7=(1,2,3,4,5,6), \]
and have, respectively,
$(C_6,C_8)=((2,4),(1,3,4))$ and $(C_6,C_8)=((3,4),(1,2,4))$. 
Their row weights are both $(4,3,4,4,1,1)$, but their quotient
vectors differ by $\mathbf e_{4,2}-\mathbf e_{2,2}$.
For the first pipe dream, the crossing columns in rows $1,2$ are
$\{1,2,3,7\}$, $\{1,2,4\}$; for the second they are
$\{1,2,3,5\}$, $\{1,2,6\}$. Both have crossing columns
$\{1,2,3,4\}$ in rows $3,4$ and $\{1\}$ in rows
$5,6$. Thus distinct coefficient monomials can have the same ordinary
Schubert weight.
\end{example}

\subsection{The terminal-insertion recurrence}\label{app:enumeration-state}
\subsubsection{A normal form for the allowed next ranks}
The rank word of $w\in S_n$ is
\[ r_i=1+|\{j<i:w_j<w_i\}|. \]
Appending rank $h$ means
\[ I_h(w)=(w_1+\one_{w_1\ge h},\ldots,w_n+\one_{w_n\ge h},h). \]
An empty slot has $r_i=i$. Let its nonempty slots be
$j_1<\cdots<j_k=m$, with ranks $\beta_i=r_{j_i}$, and write
$r=\beta_k$. Removing all terminal fixed maxima leaves the
standardized prefix of length $m$, which we call the \emph{core}.
For a nonempty core, define
$D$ to be the number of terminal fixed maxima in its restriction
to values $1,\ldots,r$. In particular $1\le D\le r$.

We use the terminal record test of
Lemma~\ref{lem:rs-terminal-test}. In a test at rank $h\ge2$, set
$b=h-1$. Strip terminal maxima. A core of size at most $b$
accepts; otherwise its last rank above $b$ rejects and its last rank
$b$ accepts. For a last rank $r<b$, delete that letter and
start $(A,\ell)=(b-1,r)$. Repeatedly strip terminal maxima and:
\begin{enumerate}[label=\textup{(\roman*)},leftmargin=2em]
\item accept a core of size at most $A$;
\item at a last rank $b$, accept precisely if
\[ \max_{\ell\le a\le A}p(a)<p(A+1)<\cdots<p(b); \]
\item delete a last rank below $\ell$, decreasing both
$A,\ell$ by one; otherwise reject.
\end{enumerate}
Rank one always accepts. This test is valid for arbitrary parents, although
we apply it only to parents in $\mathcal C$.

Set
\begin{equation}\label{solenum:eq:stop}
 l=\max\bigl(\{0\}\cup\{i<k:\beta_i-i>r-k\}\bigr),
 \qquad t=k-l,\qquad j_0=0,
 \qquad K=\max(r+1,j_l+t+1).
\end{equation}
If $l>0$, write $B=\beta_l$ and let $D_l$ be the length
of the terminal fixed suffix of the restriction to $1,\ldots,B$
of the standardized prefix through $j_l$.
Define an optional rank and capacity by
\begin{equation}\label{solenum:eq:optional}
 (q,c)=
 \begin{cases}
 (B+1,D_l-r),&l>0,\ B>r,\ D_l\ge r,\\
 (0,0),&\text{otherwise}.
 \end{cases}
\end{equation}
The justification below shows that $B>r$ already forces
$(\beta_{l+1},\ldots,\beta_k)=(1,\ldots,r)$.

\begin{lemma}[Allowed-rank normal form]\label{solenum:lem:allowed}
For every nonidentity $w\in\mathcal C_n$, the next allowed ranks are exactly
\begin{equation}\label{solenum:eq:active}
 \{1,r+1\}\ \cup\ \{q:q>0\}\ \cup\ [K,n+1].
\end{equation}
When present, $r+1<q<K$. For the identity all ranks are allowed.
\end{lemma}
\begin{proof}
The first branches of the terminal test accept rank one and rank $r+1$
and reject $2\le h\le r$. Consider $h>r+1$. After $d$
nonempty letters have been deleted, the current core ends at original
slot $j_{k-d}$, and the test has
$A=h-1-d$, $\ell=r-d+1$. Earlier prefix ranks are unchanged
by deleting later letters. Its next last rank can be deleted precisely when
\[ \beta_{k-d}<r-d+1, \quad\text{or equivalently}\quad \beta_{k-d}-(k-d)\le r-k. \]
Thus the possible deletion sequence stops exactly at $l$ in
\eqref{solenum:eq:stop}, unless it has accepted sooner. The terminal size bounds
along it are $h\ge j_i+k-i+1$. Since $j_i-i$ is nondecreasing,
the smallest is $j_l+t+1$. These give exactly the terminal interval
in \eqref{solenum:eq:active}, after its overlap with $r+1$ is removed.

Below that interval the only possible acceptance is at the remaining
anchor $\beta_l=h-1=B>r$. Each of the deleted ranks is at most
$r-t+i$ at its $i$-th position after the anchor, and therefore
is smaller than $B$. We show that the $i$-th such rank is at
most $i$. If a first violation had rank $h>i$, all previous
low ranks would be below $h-1$. Its terminal test could not accept
at any of them. It also could not accept by the terminal size bound: the earlier
anchor core has length $j_l>B>h$. Reaching that anchor rejects,
since its last rank $B$ exceeds $h-1$. Terminal fixed maxima
are stripped during this argument and change none of these ranks.
This contradicts membership in $\mathcal C$.

In particular $r\le t$; the deletion inequalities give $t\le r$,
so $t=r$. At the last low insertion of rank $t$, the test must
meet an earlier rank $t-1$. By the inequalities just proved it can
occur only at the preceding low position. Repeating gives the entire word
$1,2,\ldots,t$. The remaining position condition is therefore
\[ \max_{1\le a\le B-t}p(a)<p(B-t+1)<\cdots<p(B), \]
which is exactly $D_l\ge t=r$. This proves \eqref{solenum:eq:optional}.
Since $B<j_l$, its rank $B+1$ is strictly below
$j_l+t+1$, so it is not in the terminal interval. It is above
$r+1$. There are no other acceptance branches.
\end{proof}

\subsubsection{The state updates}
Use the state
\[ S=(n,m,r,D,K,q,c),\qquad S_0=(0,0,0,0,1,0,0). \]
For the identity permutation, take $m=r=D=q=c=0$ and $K=1$. For other
permutations use the definitions above. The following rows are taken in
order, so a terminal maximum always uses the first row.
\begin{proposition}\label{solenum:prop:state}
The allowed next ranks are \eqref{solenum:eq:active}, together with $n+1$
(which is already in the interval), and the update $S\mapsto T_hS$
is
\[
\begin{array}{c|l}
\text{condition}&T_hS\\\hline
h=n+1&(n+1,m,r,D,K,q,c)\\
h=1&(n+1,n+1,1,1,m+2,q_1,c_1)\\
h=r+1&(n+1,n+1,h,D+1,K+1,q_2,c_2)\\
h=q>0&(n+1,n+1,h,1,K+1,0,0)\\
\text{remaining }h\ge K&(n+1,n+1,h,\max(1,h-m),h+1,0,0).
\end{array}
\]
Here $(q_1,c_1)=(r+1,D-1)$ for $r>1$, and $(0,0)$
otherwise. Also $(q_2,c_2)=(q,c-1)$ when $q>h+1$ and
$c\ge1$, and $(0,0)$ otherwise.
\end{proposition}
\begin{proof}
A terminal maximum changes no core data. At a new rank one, the previous
last nonempty slot becomes the stopping anchor, so the new threshold
is $m+2$; its optional rank and position condition give the second row.
For a primary rank $h=r+1$, the new value of
$\beta_{k+1}-(k+1)$ equals $\beta_k-k$. The stopping index
$l$ is unchanged and its threshold increases by one. The optional
staircase has grown by one, consuming one unit of capacity; it disappears
if its anchor is no longer strictly above the current rank.

For the optional rank $h=B+1$, the new adjusted rank is
$B-k$. It exceeds every adjusted rank after $l$, but is
strictly below $B-l$. Thus $l$ remains the stopping index,
its threshold increases by one, and its last rank is now below the current
rank, so there is no optional acceptance.

For any remaining allowed $h>r+1$, the new adjusted rank
$h-k-1$ exceeds all preceding ones. Indeed, if $l>0$, it is
at least $j_l-l$, whereas for $i\le l$
$\beta_i-i\le j_i-i-1\le j_l-l-1$; after $l$ use the
definition of the stopping index. If $l=0$, the assertion follows
directly from $h>r+1$. The new stopping index is zero, and the
normalized threshold is $h+1$.

It remains to verify $D$. The restriction of the old permutation to
$1,\ldots,b$ has fixed-suffix length zero for $r<b\le m$,
and length $b-m$ for $b>m$. At $b=r$ it has length
$D$. Appending rank $h$ increases the fixed-suffix length of
the restriction to $1,\ldots,h-1$ by one. This gives respectively
$1,D+1,1,\max(1,h-m)$ in the four nonmaximum cases. It also gives
the stated values when the parent is an identity.
\end{proof}

Assign the empty state multiplicity one and apply every allowed update,
adding multiplicities whenever states agree. This is a closed recurrence
for $c_n$. At fixed rank all six other coordinates are bounded by
$n+1$, so the number of possible states is polynomial in $n$.
In particular, the recurrence depends only on these state variables, not on
the permutation or its complete rank word.

\subsection{The threshold recurrence}\label{app:threshold-recurrence}\label{solenum:sec:threshold}
For $w\in\mathcal C$, let $b(w)$ be the first allowed position for
inserting a new global maximum. To the state above adjoin four integers
$(b,M,R,P)$: $M$ is the position of its maximum, $R$ the
maximum value after that position, and $P$ the last earlier position
whose value exceeds the final value of $w$. An empty set has value
zero, and the empty initial word has all four entries zero.

\begin{proposition}\label{solenum:prop:threshold}
The following updates compute these four statistics, including $b(w)$,
without constructing the permutation. If $h=n+1$, set
\[ (b',M',R',P')=(\max(b,P),n+1,0,0). \]
For $h\le n$, set
\begin{align*}
 b'&=\begin{cases}\max(b,M),&2\le h\le n,\ R\le h-2,\ M\ne h-1,\\
 b,&\text{otherwise},\end{cases}\\
 M'&=M,\qquad R'=\max\{h,R+\one_{R\ge h}\},\\
 P'&=\begin{cases}n,&h\le L,\\P,&h>L,\end{cases}
 \qquad L=\begin{cases}n,&n>m,\\r,&n=m.\end{cases}
\end{align*}
\end{proposition}
\begin{proof}
Insert an additional marked maximum into $w$. The threshold is
the greatest position selected by an extra record that fails its
required rank equality, or zero if no such record occurs. We examine
how these selections change when a rank is appended to $w$.
When $h=n+1$, the extra R1 records select the old right-to-left
maxima to the right of the marked entry. All except the last
position are defective: a right-to-left maximum with a later letter has
some smaller value later. The latest defective selection is precisely
$P$. For $2\le h\le n$, the only
extra R2 selection is the old maximum. It qualifies exactly when all
values after it are below $h-1$, namely $R\le h-2$. Its own
prefix rank is its position $M$, so it is defective precisely when
$M\ne h-1$. Rank one gives no extra R2. Earlier tests do not change.
Taking the maximum of the new obstruction and the old threshold proves
the update for $b$.

The updates of $M,R$ follow by insertion. If $h\le L$, the
old last position is the nearest larger predecessor, giving $P'=n$.
If $h>L$ and the insertion is not a maximum, it removes the old
last value from the right-to-left maxima. It cannot remove another one:
every removed old right-to-left maximum would give an R1 record, whereas
an insertion satisfying the record condition has at most one. Thus the preceding right-to-left
maximum stays at position $P$, giving $P'=P$.
\end{proof}

Adding multiplicities on these augmented states computes
\[ B_n(t)=\sum_{w\in\mathcal C_n}t^{b(w)} \]
using a fixed number of integer coordinates, each bounded in terms
of $n$. Consequently it
also gives $c_{n+1}=(n+1)B_n(1)-B_n'(1)$. The extra coordinates
are needed: $12$ and $21$ both have threshold zero, but the thresholds
of their three maximum-insertion children are respectively
$(0,2,0)$ and $(0,0,1)$.

\subsection{The generating-function equations}\label{app:enumeration-kernels}\label{solenum:sec:gf}
\subsubsection{Two types of terminal cores}
A core has no terminal fixed maximum, so $n=m$. We distinguish type
$\mathcal A$, where $K=r+1$, from type $\mathcal B$, where
$K>r+1$. Set
\[
 a=m-K+1,\qquad s=D-\one_{D=r},\qquad
 \ell=\begin{cases}\min(c+1,q-r-1),&q>0,\\0,&q=0.\end{cases}
\]
For type $\mathcal A$ we have $a\ge1$, $s\ge0$, and $\ell=0$.
For type $\mathcal B$ all three parameters are nonnegative. The
optional rank is available exactly when $\ell>0$, and $\ell$ is
the number of successive primary insertions for which it remains
available.

Between two cores, insert $g\ge0$ terminal fixed maxima and then one
nonmaximum rank. Proposition~\ref{solenum:prop:state} gives the table
below; every row has size weight $z^{g+1}$.
\[
\begin{array}{c|l}
\text{next rank}&\text{new core}\\\hline
1&\mathcal B(g,0,s)\\
r+1&\text{same type, with }(a+g,s+1,\max(\ell-1,0))\\
q\quad(\ell>0)&\mathcal B(a+g,1,0)\\
\text{other allowed rank}&
 \mathcal A(a+g-j,\max(1,1+j-a)).
\end{array}
\]
In the last row, $1\le j\le a+g-1$ for type $\mathcal A$ and
$0\le j\le a+g-1$ for type $\mathcal B$. Each $j$ gives a
different rank, namely $h=K+j$. The first nonempty core is of type
$\mathcal A(a,s)$, with $a\ge1$, $s\ge0$, and weight
$z^{a+s+1}$; every preceding slot is empty. Any number of terminal
fixed maxima may follow the last core.

To check that these parameters suffice, substitute $h=K+j$ in the
last row of Proposition~\ref{solenum:prop:state}. It gives
$D=\max(1,1+j-a)$, strictly below the new rank, so this is also
the new $s$. At a rank-one insertion, the optional capacity is
$\min(D,r-1)=D-\one_{D=r}=s$. Thus the table records all information needed for later insertions.

Let $\mathcal A(\xi,y)$ and $\mathcal B(\xi,y,t)$ count the two
types by $\xi^ay^s$ and $\xi^ay^st^\ell$, respectively, with $z$
marking size. Write
\[
\begin{gathered}
 H=\mathcal A(1,1)+\mathcal B(1,1,1),\qquad
 G(y)=\mathcal A(1,y)+\mathcal B(1,y,1),\\
 \mathcal B_0(\xi,y)=[t^0]\mathcal B(\xi,y,t),\qquad
 Q(\xi)=\mathcal B(\xi,1,1)-\mathcal B_0(\xi,1),\\
 A_\xi=\frac{z}{1-z\xi},\qquad
 I_\xi(y)=\frac{z^2\xi}{(1-z\xi)(1-zy)}.
\end{gathered}
\]
Here $Q$ counts cores with an available optional rank. Removing
terminal fixed maxima gives
\begin{equation}\label{eq:ordinary-core-series}
 C(z)=\frac{1+H}{1-z}.
\end{equation}

\subsubsection{Eliminating the insertion parameters}
The transition table gives two functional equations:
\begin{align}
 (t-A_\xi y)\mathcal B(\xi,y,t)
 &=A_\xi tG(t)+A_\xi y(t-1)\mathcal B_0(\xi,y)
       +A_\xi ytQ(\xi),\notag\\
 (1-A_\xi y)\mathcal A(\xi,y)
 &=I_\xi(y)(1+yH)
 +\frac{A_\xi y}{1-\xi}
   \{\xi H-\mathcal A(\xi,1)-\xi\mathcal B(\xi,1,1)\}.
 \label{solenum:eq:Akernel}
\end{align}
In the first equation the three terms come from rank one, the primary
rank, and the optional rank. In the second, $I_\xi(y)$ counts the
first core. Subsequent ranks with new $a'\le g$ give
$I_\xi(y)yH$; the remaining ranks give the finite sums
$\xi+\cdots+\xi^{a-\epsilon}$, where $\epsilon=1$ for type
$\mathcal A$ and $0$ for type $\mathcal B$.

Set $t=A_\xi y$ in the first equation, and also evaluate it at $t=1$.
We obtain
\begin{align}
 \mathcal B_0(\xi,y)
 &=\frac{A_\xi G(A_\xi y)+A_\xi yQ(\xi)}{1-A_\xi y},\nonumber\\
 \mathcal B(\xi,y,1)
 &=\frac{A_\xi H+A_\xi yQ(\xi)}{1-A_\xi y},
 \label{solenum:eq:Bclosed}\\
 Q(\xi)&=\frac{A_\xi}{1-A_\xi}\{H-G(A_\xi)\}.
 \label{solenum:eq:Qclosed}
\end{align}
The substitutions are valid as formal power series: $A_\xi$ has
positive $z$-adic order, and for each fixed size the auxiliary
exponents range over finite sets.

We next use \eqref{solenum:eq:Akernel} at $y=1$ to eliminate its
last term, set $\xi=1$, and add \eqref{solenum:eq:Bclosed}.
The optional contribution $Q(1)$ cancels. Simplification gives
\begin{equation}\label{solenum:eq:Gclosed}
 G(y)=\frac{z^2(1-y)}{(1-z)(1-zy)}
 +\frac{1}{1-z-zy}
 \left(z+(1-3z)y-\frac{z^3y(1-y)}{(1-z)(1-zy)}\right)H.
\end{equation}
Thus only $H$ remains unknown.

At $y=1$, the other kernel is
\[ 1-\xi+A_\xi\xi=\frac{1-\xi+z\xi^2}{1-z\xi}. \]
Let $X=1+zX^2$ be its power-series root and set $\omega=X-1$.
Then $z=\omega/(1+\omega)^2$ and $A_X=\omega/(1+\omega)$. Substitution in
\eqref{solenum:eq:Akernel} gives
\[ \mathcal B(X,1,1) =H-\frac{\omega^2}{1+\omega+\omega^2}(1+H). \]
On the other hand, \eqref{solenum:eq:Bclosed} and
\eqref{solenum:eq:Qclosed} give
\[ \mathcal B(X,1,1) =\omega H+\omega^2\left(H-G\left(\frac{\omega}{1+\omega}\right)\right). \]
Equating these two expressions and using \eqref{solenum:eq:Gclosed}
therefore yields a single linear equation for $H$. Its solution is
\[ H=\frac{\omega^2(1+2\omega+\omega^2+\omega^3)(1+3\omega+\omega^2+\omega^3)}{(1+2\omega)D_0(\omega)}, \]
where $D_0(\omega)$ is the denominator polynomial in
Theorem~\ref{solenum:thm:count}. Every denominator has nonzero constant
term. Finally \eqref{eq:ordinary-core-series} gives
\eqref{solenum:eq:ordinary} by multiplying the polynomials above. The other Catalan branch replaces $\omega$ by $1/\omega$;
the resulting series has coefficient $1/2$ at $z$, whereas $c_1=1$.
Hence $C(z)$ has degree exactly two over $\QQ(z)$.

\subsection{The dominant pole}\label{app:enumeration-poles}\label{solenum:sec:asymptotic}
We prove Theorem~\ref{thm:record-growth} by isolating the dominant
pole and computing its residue. Eliminating $\omega$ from
$D_0(\omega)=0$ and
$z=\omega/(1+\omega)^2$ gives
\[ p_0(z)=1-10z+37z^2-64z^3+43z^4+20z^5-48z^6+24z^7-4z^8. \]
\begin{lemma}\label{solenum:lem:pole}
The polynomial $p_0$ has exactly one root in $|z|<1/4$, none
on $|z|=1/4$, and that root is simple and equals $\rho$.
\end{lemma}
\begin{proof}
The map $z=(1-s)/(4(1+s))$ sends the disk to $\Re s>0$.
After multiplying by $16384(1+s)^8$, $p_0$ becomes
\[
\begin{split}
 P_{\mathrm R}(s)={}&113831s^8+562712s^7+1139476s^6+1238792s^5\\
 &+786362s^4+291688s^3+57524s^2+4088s-169.
\end{split}
\]
Form the Routh array from the alternating coefficient rows
$R_0,R_1$, with subsequent entries
\[ R_{i,j}=\frac{R_{i-1,0}R_{i-2,j+1}-R_{i-2,0}R_{i-1,j+1}}{R_{i-1,0}}. \]
Its first two entries are $113831,562712$; its next six entries
satisfy the following rational bounds:
\[ (888880,888881),\ (778334,778335),\ (435230,435231),\ 
 (162970,162971),\ (39895,39896),\ (5187,5188). \]
Its last entry is $-169$. There are no exceptional zero rows or
pivots. The Routh root count \cite{EnumRouth} therefore gives exactly one
root in the right half-plane and none on its boundary. This includes
multiplicity; also $p_0(-1/4)\ne0$, the point omitted by the map.

The coefficients of $D_0$ have one sign change in descending order,
so there is one positive root. Direct substitution brackets it between
$0.6198$ and $0.6199$. Its image lies in $(0,1/4)$, and
hence must be the unique disk root just counted. The count proves
simplicity.
\end{proof}

\begin{proof}[Proof of Theorem~\ref{thm:record-growth}]
The Catalan branch $\omega(z)$ is analytic in $|z|<1/4$.
By \eqref{solenum:eq:ordinary}, the possible singularities of $C(z)$
in this disk are poles, and they are among the zeros of $p_0(z)$.
By Lemma~\ref{solenum:lem:pole}, the only such zero is
$\rho$, and it is simple. The bounds
$0.6198<\omega_0<0.6199$ give
$A_0(\omega_0)B_0(\omega_0)>0$, so the numerator in
\eqref{solenum:eq:ordinary} does not vanish at $\rho$.
Since $\omega'(\rho)=(1+\omega_0)^3/(1-\omega_0)$, the residue
of $C(z)$ at $\rho$ is $-K\rho$. Thus
\[ C(z)-\frac{K}{1-z/\rho} \]
is analytic in $|z|<1/4$. Cauchy's coefficient estimate on
$|z|=R$ gives \eqref{eq:record-growth} for every
$\rho<R<1/4$.
\end{proof}

\section*{Declaration of generative AI and AI-assisted technologies}
During the development of Appendix B,D,E,F the author used ChatGPT to assist in exploring and formulating most proof arguments. The author subsequently checked, revised, and independently verified all arguments and takes full responsibility for the mathematical content of the paper.

\FloatBarrier
\endgroup
\providecommand{\bysame}{\leavevmode\hbox to3em{\hrulefill}\thinspace}
\providecommand{\MR}{\relax\ifhmode\unskip\space\fi MR }
\providecommand{\MRhref}[2]{%
  \href{http://www.ams.org/mathscinet-getitem?mr=#1}{#2}
}
\providecommand{\href}[2]{#2}

\end{document}